%% file: main.tex
\newif\ifHAL
\HALtrue  
 
\ifHAL
\documentclass[11pt, a4paper, twoside, oneside]{article}
\else
\documentclass[10pt]{m2an}
\fi

\ifHAL 
\input{preamble.tex}

\else
\input{preamble_m2an.tex}
\fi
\input{macros.tex} 

\begin{document} 

\title{Hybrid high-order methods for elasto-acoustic wave propagation in the time domain}
\ifHAL
\makeatletter
\newcommand{\keywords}[1]{\def\@keywords{#1}}   
\newcommand{\subjclass}[1]{\def\@subjclass{#1}}   
\begin{center}
\huge \@title \\[0.5cm]
\large 
Romain Mottier\footnote[2]{CEA, DAM, DIF, F-91297 Arpajon, France, and CERMICS, ENPC, Institut Polytechnique de Paris, F-77455 Marne-la-Vall\'ee cedex 2, and SERENA Project-Team, Centre Inria de Paris, F-75647 Paris, France. Email: romain.mottier@enpc.fr}, 
$\qquad$ Alexandre Ern\footnote[3]{CERMICS, ENPC, Institut Polytechnique de Paris, F-77455 Marne-la-Vall\'ee cedex 2, and SERENA Project-Team, Centre Inria de Paris, F-75647 Paris, France. Email: alexandre.ern@enpc.fr}, 
$\qquad$ Rekha Khot\footnote[4]{SERENA Project-Team, Centre Inria de Paris, F-75647 Paris, and CERMICS, ENPC, Institut Polytechnique de Paris, F-77455 Marne-la-Vall\'ee cedex 2, France. Email: rekha.khot@inria.fr}, 
$\qquad$ Laurent Guillot\footnote[5]{CEA, DAM, DIF, F-91297 Arpajon, France. Email: laurent.guillot@cea.fr}\\[0.5cm]
\end{center}
\else
\author{R. Mottier$^{1,2}$}
\footnotetext[1]{CERMICS, ENPC, Institut Polytechnique de Paris, F-77455 Marne-la-Vall\'ee cedex 2, \& SERENA Project-Team, Centre Inria de Paris, F-75647 Paris, France.; \email{romain.mottier@enpc.fr, alexandre.ern@enpc.fr, rekha.khot@inria.fr}}
\footnotetext[2]{CEA, DAM, DIF, F-91297 Arpajon, France; \email{laurent.guillot@cea.fr}}
\author{A. Ern$^1$}
\author{R. Khot$^1$}
\author{L. Guillot$^2$}
\fi
\begin{abstract}
We devise a Hybrid High-Order (HHO) method for the coupling between the acoustic and elastic wave equations in the time domain. A first-order formulation in time is considered. The HHO method can use equal-order and mixed-order settings \cred{with polynomial degree $k\geq 0$ for the face unknowns}, together with $\cal{O}(1)$- or $\cal{O}(\frac{1}{h})$-stabilization. An energy-error estimate is established in the time-continuous case. A numerical spectral analysis is performed, showing that $\cal{O}(1)$-stabilization is required to avoid excessive CFL limitations for explicit time discretizations. Moreover, the spectral radius of the stiffness matrix is  found to be fairly independent of the geometry of the mesh cells. For analytical solutions on general meshes, optimal convergence rates of order $(k+1)$ are shown in both equal- and mixed-order settings using $\cal{O}(1)$-stabilization, whereas order $(k+2)$ is achieved in the mixed-order setting using $\cal{O}(\frac{1}{h})$-stabilization. Test cases with a Ricker wavelet as an initial condition showcase the relevance of the proposed method for the simulation of elasto-acoustic wave propagation across media with contrasted material properties.
\ifHAL
\\[0.15cm]
{\small 
\noindent \textbf{Mathematics Subjects Classification.} \@subjclass.\\
\noindent \textbf{Keywords.} \@keywords
}
\else
\fi
\end{abstract}
\ifHAL
\else
\maketitle 
\fi

\section{Introduction}

The propagation of acoustic and elastic waves plays an important role in the modeling of various physical phenomena in many applications, such as medical imaging and geophysical exploration. In many of these applications, the interaction between solid and fluid domains plays a central role. The main field of application that interests us in this paper is the propagation of waves through rocks, water and air leading to the coupling of elastic and acoustic waves in several media with contrasted material properties. Scenarios of interest range from relatively simple configurations, such as wave propagation in layered media, to complex cases involving heterogeneous domains with intricate geometries. Capturing the wave propagation as well as the dynamics of wave transmission and reflection at fluid-solid interfaces is essential to accurately predict the physical behavior of the phenomena under study. Additionally, employing discretization schemes with moderate numerical dispersion and dissipation is critical to the reliability of simulations. \cred{These challenges motivate the development of advanced numerical methods that can address the complexity of coupled wave problems in the time domain.}
 
Finite Differences (FD) are one of the most widely used methods \cred{for space discretization} due to their conceptual simplicity and computational efficiency. However, they have several limitations. In particular, FD methods meet with difficulties in complex geometries and are subject to geometrical errors owing to inaccuracies at interfaces \cite{VRC_2002}, although geometric mappings can enhance flexibility \cite{AP_2009}. Moreover, FD methods are also subject to numerical dispersion, which can be tempered by using high-order schemes. However, high-order FD schemes require large stencils which can hinder parallel scalability, and their explicit time discretization can be challenging \cite{PS_2017}. An alternative to high-order FD methods are high-order continuous finite elements (see \cite{COHEN_2003} for a review). These methods provide a natural way to handle non-planar complex geometries with interfaces. However, their efficiency is hampered by the presence of a global non-diagonal mass matrix, and the simulations can be polluted by spurious modes \cite{MARFURT_1984}. To make explicit time-stepping possible, Spectral Element Methods (SEM) were introduced. These methods align quadrature points with Lagrangian interpolation nodes that are Gauss--Lobatto--Legendre points, leading to a diagonal mass matrix (mass lumping), without loss in accuracy. The main drawback of SEM is to rely almost exclusively on quadrangular/hexahedral meshes with tensorization of quadrature nodes to be very efficient, making the discretization of complicated geometries quite challenging.

To obtain more geometric flexibility, discontinuous Galerkin (dG) methods were introduced both in first-order \cite{FR_1999, MR_2005} and second-order time formulations \cite{GSS_2006} of the wave equation (see \cite{ABM_2020} for the elasto-acoustic coupled problem). In general, in dG schemes for the \cred{first-order} formulation in time, stabilization acts as a  dissipative mechanism, whereas it is possible to identify a discrete energy that is conserved for the second-order time formulation. For an energy-conserving dG discretization of the first-order formulation, see \cite{CE_2006} where some continuity of the unknowns is enforced. The main drawback of dG methods is their computational cost since they involve much more degrees of freedom than continuous finite element methods. Hybridizable Discontinuous Galerkin (HDG) methods \cite{CGL_2009} introduce an additional unknown defined over the mesh skeleton and offer a reduction of the computational cost by the use of static condensation. Moreover, in a coupling wave context, HDG methods weakly enforce  the transmission conditions at the interface in a seamless way \cite{TVG_2017}. \cred{In comparison with the second-order formulation,  the first-order formulation allows to approximate the primal and dual variables simultaneously. Furthermore, the first-order formulation offers \cred{a wider choice of} high-order time-discrete schemes, such as Runge-Kutta schemes, whereas high-order approximation schemes of the second-order time derivative are less straightforward.}

Coupled elasto-acoustic wave propagation has also been adressed in the frequency domain. In particular, mixed finite element formulations have been proposed in \cite{GMM_2007} using a dual-mixed formulation in the solid region, and a standard primal formulation in the fluid region. In this context, the key challenge lies in coupling the solid stress tensor with the Helmholtz equation governing the fluid pressure. This is achieved by enforcing one of the transmission conditions weakly using a Lagrange multiplier.

The present work focuses on the coupling between elastic and acoustic wave equations in the time domain in the first-order formulation within the framework of hybrid high-order (HHO) methods for space semi-discretization. HHO methods make use of polynomials of arbitrary order ($k \geq 1$ for elasticity, $k\ge0$ for diffusion) attached to the mesh faces and polynomials of order $k' \in \{k, k+1\}$ attached to the mesh cells \cite{CEP_2021,DD_2020}. Initially developed for linear diffusion problems \cite{DEL_2014} and locking-free linear elasticity \cite{DE_2015}, HHO methods offer several advantages, including the natural handling of polyhedral and nonconforming meshes, local conservativity, and optimal convergence rates (order $(k+1)$ in the energy norm). Additionally, a static condensation procedure enables the local elimination of the cell unknowns, enhancing computational efficiency. HHO methods rely on two locally defined operators: a gradient reconstruction operator and a stabilization operator. The close connection between HHO, HDG and Weak Galerkin (WG) methods has been established in \cite{CDE_2016}. HHO methods have been extended to wave propagation in \cite{BDES_2021,BDE_2022} for both first- and second-order time formulations, see also \cite{ES_2024, SEJD_2023} for further developments on explicit time schemes for the second-order in time formulation of the wave equation. As HDG methods, a key advantage of HHO methods is their easy handling of coupling conditions through face-based degrees of freedom, which enables a natural and efficient treatment of interface conditions in multiphysics problems.

The present work brings several advances in the development and analysis of Hybrid High-Order (HHO) methods for coupled elasto-acoustic wave propagation. Our first contribution is an energy-error estimate in the space semi-discrete case. In particular, we leverage the fact that the coupling terms exhibit a skew-symmetric, and thus non-dissipative, structure. We improve on the analysis in \cite{BDES_2021} since we treat coupled elasto-acoustic wave problems, and we simplify the error estimate by exploiting tighter consistency properties of the method. Our second contribution is a spectral analysis of the resulting algebraic formulation, that reveals a behavior of the spectral radius of the stiffness matrix as $\min(\eta,1/\eta)$ with $\eta$ the scaling of the stabilization. Interestingly, this scaling is fairly independent of the geometry of the mesh cells (triangular, quadrangular or polygonal). Thus, explicit time schemes are recommended with $\cal{O}(1)$-stabilization, whereas implicit time schemes can be combined with either $\cal{O}(1)$- or $\cal{O}(\frac{1}{h})$-stabilization (see \hyperref[rem::stab_strategies]{\Cref{rem::stab_strategies}} for the scaling of the stabilization). The third contribution concerns optimal convergence rates for smooth solutions where we observe that $\cal{O}(\frac{1}{h})$-stabilization leads to improved rates in the energy norm (order $(k+2)$ instead of $(k+1)$). The last contribution is a more realistic study featuring interface, Rayleigh-type waves and complex transmission phenomena where we perform a comparison of HHO solutions with \cred{a reference solution obtained by a numerical computation using Green functions}. Finally, we notice that our discretization method differs from \cite{TVG_2017} since the primal variable in the fluid domain is the pressure here, whereas it is the fluid velocity in \cite{TVG_2017}. Moreover, we allow for implicit and explicit schemes as well as $\cal{O}(1)$- and $\cal{O}(\frac{1}{h})$-stabilizations, whereas \cite{TVG_2017} focuses on explicit time schemes and $\cal{O}(1)$-stabilization.

The paper is organized as follows. In \hyperref[sec::model_problem]{\Cref{sec::model_problem}}, we present the model problem for the elasto-acoustic coupling as well as its weak formulation. In \hyperref[sec::HHO_discretization]{\Cref{sec::HHO_discretization}}, we detail the HHO space semi-discretization. In \hyperref[sec::Error_analysis]{\Cref{sec::Error_analysis}}, we present the energy-error analysis in the time-continuous setting. In \hyperref[sec::algebraic_realization]{\Cref{sec::algebraic_realization}}, the algebraic realization of the space semi-discrete problem is discussed. Finally, numerical results are presented in \hyperref[sec::numerical_results]{\Cref{sec::numerical_results}}.

\section{Model problem}\label{sec::model_problem}

This section introduces the domain configuration, and the coupling of the acoustic and elastic wave equations.  We use boldface (resp. blackboard) fonts for vectors (resp. tensors), as well as for vector-valued  (resp. tensor-valued) fields and spaces composed of such fields. 

Let $J:= (0,T_{\rm{f}})$ be the time interval with the final time $T_{\rm{f}} > 0$, and  $\Omega$ be a polyhedral domain in $\bb{R}^d$, $d \in \{2,3\}$ (open, bounded, connected, Lipschitz subset of $\bb{R}^d$). We consider a partition of $\Omega$ such that $\ol{\Omega}:= \ol{\domain{s}} \cup \ol{\domain{f}}$ into two disjoint, open, polyhedral subdomains $\domain{s}$ and $\domain{f}$ constituting the elastic medium and the acoustic medium, respectively, sharing the polygonal interface $\G := \partial \domain{s} \cap \partial \domain{f}$. We fix the unit normal vector $\bd{n}_{\G}$ to $\G$ as conventionally pointing from $\domain{s}$ to $\domain{f}$.
\begin{figure}[H]
\centering
\begin{tikzpicture}[scale=0.65]
\begin{scope}[xshift=2cm]
\coordinate (A) at (-1,0);
\coordinate (B) at (3,0);
\coordinate (C) at (3,1);
\coordinate (D) at (3,3);
\coordinate (E) at (3,4);
\coordinate (F) at (-1,4);
\coordinate (G) at (-1,3);
\coordinate (H) at (-1,1);
\draw[line width = 3, color = ceared] (A) -- (B) -- (C) -- (D) -- (E) -- (F) -- (G) -- (H) -- cycle;
\coordinate (A2) at (3,0);
\coordinate (B2) at (7,0);
\coordinate (C2) at (7,1);
\coordinate (D2) at (7,3);
\coordinate (E2) at (7,4);
\coordinate (F2) at (3,4);
\coordinate (G2) at (3,3);
\coordinate (H2) at (3,1);
\draw[line width = 3, color = blue] (A2) -- (B2) -- (C2) -- (D2) -- (E2) -- (F2) -- (G2) -- (H2) -- cycle;
\draw[line width = 4, color = black] (3,4.09) -- (3,-0.09);
\draw[<-, line width=0.35mm] (2.9,4.35) -- (-1.35,4.35); 
\draw[-,  line width=0.35mm] (-1.35,4.35) -- (-1.35,-0.35);   
\draw[->, line width=0.35mm] (-1.35,-0.35) -- (2.9,-0.35);  
\draw[<-, line width=0.35mm] (3.1,4.35) -- (7.35,4.35); 
\draw[-,  line width=0.35mm] (7.35,4.35) -- (7.35,-0.35);   
\draw[->, line width=0.35mm] (7.35,-0.35) -- (3.1,-0.35);    
\draw[->, line width=0.35mm] (3,2.125) -- (4.75,2.125);
\node at (2.5,3.5) {\color{black} $\Gamma$};
\node at (5,2.75) {\color{black} $\bd{n}_\Gamma$};
\node at (-2.75,2) {\color{black} ${\partial \Omega^{\sc{s}} \backslash \Gamma}$};
\node at (1,0.9) {\color{ceared} \LARGE ${\Omega^\sc{s}}$};
\node at (8.75,2) {\color{black} ${\partial \Omega^\sc{f} \backslash \Gamma}$};
\node at (5,0.9) {\color{blue} \LARGE ${\Omega^\sc{f}}$};
\end{scope}
\end{tikzpicture}
\caption{Elastic domain $\Omega^{\sS}$, acoustic domain $\Omega^{\sF}$, and  unit normal $\bd{n}_\Gamma$ along the  interface $\Gamma$}
\label{fig:model}
\end{figure}
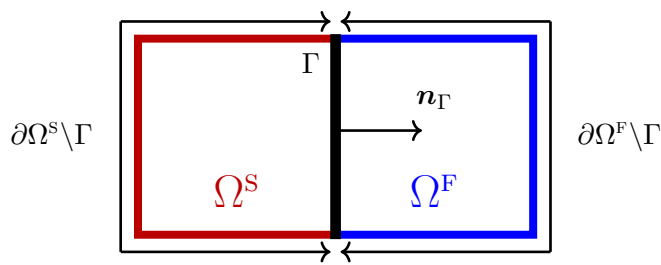

\subsection{Strong formulation}

\noindent \textbf{Acoustic wave equation.} The acoustic wave equation governs the scalar pressure field $ p \left[ \rm{Pa} \right]$ and the velocity field $\bd{m}\left[\frac{\rm{m}}{\rm{s}}\right]$ solving the following PDE system in $J \times \domain{f}$: 
\begin{subequations}\label{continuous_acoustic}
\begin{align}
\rho^\sc{f} \partial_t \bd{m} - \nabla p & = \bd{0},\label{continuous_acoustic_eq_1}\\
\frac{1}{\kappa}\partial_t p -  \nabla {\cdot} \bd{m} & = f^{\sc{f}},
\label{continuous_acoustic_eq_2}
\end{align}
\end{subequations}
\noindent with the fluid density  $\rho^\sc{f}\left[\frac{\rm{kg}}{\rm{m}^3} \right]$, the fluid bulk modulus $\kappa  \left[\rm{Pa}\right]$, and the source term $f^{\sc{f}} \left[\frac{1}{\rm{s}}\right]$. The celerity of the acoustic  waves is  $c_{\sc{p}}^\sc{f} := \sqrt{\kappa / \rho^\sc{f}} \left[\frac{\rm{m}}{\rm{s}}\right]$. We assign the initial conditions 
\begin{align}
p(0) = p_0 \quad\text{and}\quad\bd{m}{(0)}  = \bd{m}{_0},
\label{IC_BC_acoustic}
\end{align}
with given data $p_0$ and $\bd{m}{_0}$. For simplicity, we enforce homogeneous Dirichlet boundary conditions on $p$ on $\partial \domain{f} \backslash \G$.\\

\noindent \textbf{Elastic wave equation.} Let $\nabla_{\rm{sym}}:=\frac{1}{2}(\nabla+\nabla^{\cred{\top}})$ be the symmetric gradient operator.  The elastic wave equation governs the (linearized) Cauchy stress tensor $\bbm{s}\left[\rm{Pa}\right]$ and  the velocity field $\bd{v} \left[\frac{\rm{m}}{\rm{s}}\right]$ solving the following PDE system in $J \times \domain{s}$:
\begin{subequations}\label{conitnuous_elastic_eq}
\begin{align}
\bbm{C}^{-1} \partial_t \bbm{s} - \nabla_{\rm{sym}} \bd{v} & = \bd{0},\label{continuous_elastic_1} \\
\rho^\sc{s} \partial_t \bd{v} - \nabla {\cdot} \bbm{s} & = \bd{f}^{\sc{s}},
\label{continuous_elastic_2}
\end{align}
\end{subequations}
\noindent with the solid density $\rho^\sc{s} \left[\frac{\rm{kg}}{\rm{m}^3}\right]$ and the source term $\bd{f}^{\sc{s}} \left[\frac{\rm{Pa}}{\rm{m}}\right]$.  In the framework of isotropic elasticity, the 4$^{th}$-order Hooke tensor $\bbm{C}\left[\rm{Pa} \right]$ only depends on the Lamé parameters $\lambda \left[\rm{Pa}\right]$ and $\mu \left[\rm{Pa}\right]$, and is such that $\bbm{C}_{i j k l} := \lambda \delta_{i j} \delta_{k l} + \mu(\delta_{i k} \delta_{j l} + \delta_{i l} \delta_{j k})$, where the $\delta$'s are Kronecker symbols. In this setting, there are two wave speeds  $\left[\frac{\rm{m}}{\rm{s}}\right]$ related to two types of body waves, far from material interfaces:
\begin{subequations}
\label{cs}
\begin{alignat}{3}
c_{\sc{p}}^\sc{s}&:= \sqrt{\left(\lambda + 2 \mu\right) / \rho^\sc{s}} &\qquad &\text{for compressional (or P-) waves},\label{cs:a}\\
c_{\sc{s}}^\sc{s}&:=\sqrt{\mu / \rho^\sc{s}}& \qquad & \text{for shear (or S-) waves.}\label{cs:b}
\end{alignat}
\end{subequations}
We do not consider here the incompressible limit as $\frac{\lambda}{\mu} \gg 1$, so that both wave speeds in \eqref{cs} are of similar magnitude. We assign the initial conditions 
\begin{align}\label{IC_BC_elastic}
\bd{v}(0) = \bd{v}_{0}\quad\text{and}\quad \bbm{s}(0) = \bbm{s}_{0},
\end{align}
with given data $\bd{v}_{0}$ and $\bbm{s}_{0}$. For simplicity, we enforce homogeneous Dirichlet boundary conditions on $\bd{v}$ on $\partial \domain{s} \backslash \G$.\\

\noindent \textbf{Coupled problem.} The interface conditions on $J \times \G$ are
\begin{subequations}\label{continuous_coupling}
\begin{align}
\label{coupling1}
\bd{v} {\cdot} \bd{n}_\G & = \bd{m} {\cdot} \bd{n}_\G, \\
\bbm{s} {\cdot} \bd{n}_\G & = p~\bd{n}_\G,
\label{coupling2}
\end{align}
\end{subequations}
where the first equation is a kinematic condition and the second equation is a balance of forces per unit surface (namely, tractions) at the interface. 

\cred{We mention that \eqref{continuous_acoustic_eq_1}-\eqref{continuous_acoustic_eq_2} can be interpreted as the momentum balance and the (linearized) mass conservation equations in the fluid domain, and  \eqref{continuous_elastic_1}-\eqref{continuous_elastic_2} as the constitutive assumption and momentum balance equation in the solid domain. Finally, we notice that in \eqref{continuous_coupling}, only the normal component of the velocity is continuous at the interface.}

\subsection{Weak formulation}

We define the functional spaces 
\begin{subequations}
\begin{align}
H^1_{0\G}(\domain{f}) &:= \{p \in H^1(\domain{f}) : p|_{\partial \domain{f} \backslash \G} = 0\},\\
\bd{H}^1_{0\G} (\domain{s}) &:= \{\bd{v} \in \bd{H}^1(\domain{s}) : \bd{v}|_{\partial \domain{s} \backslash \G} = \bd{0}\},
\end{align}
\end{subequations}
taking into account the homogeneous Dirichlet boundary conditions. \cred{For weighted Lebesgue spaces, we use the notation $(u, v)_{L^2(\kappa; \Omega)} := \int_{\Omega} \kappa u v ~\mathrm{d}\Omega$ for all $u, v \in L^2(\Omega)$ and for a bounded and uniformly positive weight $\kappa$. A similar notation is used for inner products involving vector- and tensor-valued fields.}  Focusing for simplicity on a smooth solution in time \cred{which requires initial conditions $(\bd{m}_0,p_0)\in \bd{L}^2(\domain{f})\times H^1_{0\G}(\domain{f})$ and $(\bbm{s}_0,\bd{v}_0)\in \bbm{L}^2_{\rm{sym}}(\domain{s})\times \bd{H}^1_{0\G}(\domain{s}) $}, the coupled elasto-acoustic wave problem consists of finding $(\bd{m},p)\in C^1(\ol{J};\bd{L}^2(\domain{f}))\times (C^1(\ol{J};L^2(\domain{f}))\cap C^0(\ol{J};H^1_{0\G}(\domain{f})))$ and $(\bbm{s},\bd{v})\in C^1(\ol{J};\bbm{L}^2_{\rm{sym}}(\domain{s}))\times (C^1(\ol{J};\bd{L}^2(\domain{s}))\cap C^0(\ol{J};\bd{H}^1_{0\G}(\domain{s})))$ such that, for all $(\bd{r},q) \in \bd{L}^2(\domain{f}) \times H_{0\G}^1(\domain{f})$, all $(\bbm{b}, \bd{w}) \in \bbm{L}^2_{\rm{sym}}(\domain{s}) \times \bd{H}_{0\G}^1(\domain{s})$, and all $t \in \ol{J},$ 
\begin{subequations} \label{weak_form_acoustic_eq}
\begin{align}
(\partial_t \bd{m}(t), \bd{r})_{\bd{L}^2(\rho^\sc{f} ;\domain{f})} - (\nabla p(t), \bd{r})_{\bd{L}^2(\domain{f})} &= 0,\label{aweak1} \\
 (\partial_t p(t),q)_{L^2(\frac{1}{\kappa};\domain{f})} + (\bd{m}(t), \nabla q)_{\bd{L}^2(\domain{f})} + (\bd{v}(t) {\cdot} \bd{n}_{\G}, q)_{L^2(\G)} &= (f^{\sc{f}}(t), q)_{L^2(\domain{f})},\label{aweak2}
\end{align}
\end{subequations}
and
\begin{subequations}\label{weak_form_elastic_eq}
\begin{align}
\label{eweak1}
(\partial_t \bbm{s}(t), \bbm{b})_{\bbm{L}^2(\bbm{C}^{-1};\domain{s})} - (\nabla_{\rm{sym}}\bd{\bd{v}}(t), \bbm{b})_{\bbm{L}^2(\domain{s})} &= 0, \\
(\partial_t \bd{v}(t), \bd{w})_{\bd{L}^2(\rho^\sc{s};\domain{s})} + (\bbm{s}(t), \nabla_{\rm{sym}}\bd{w})_{\bbm{L}^2 (\domain{s})} - (p(t)  \bd{n}_{\G}, \bd{w})_{\bd{L}^2(\G)}  &= (\bd{f}^{\sc{s}}(t), \bd{w})_{\bd{L}^2(\domain{s})}.
\label{eweak2}
\end{align}
\end{subequations}
\noindent Notice that the coupling condition \eqref{coupling1} is enforced weakly in \eqref{aweak2}, and the coupling condition \eqref{coupling2} is enforced weakly in \eqref{eweak2}.

\subsection{Mechanical energy.} 

The total mechanical energy $\cal{E}(t) := \cal{E}^\sc{s}(t) + \cal{E}^\sc{f}(t)$ of a wave propagating through an elasto-acoustic medium is expressed as the sum of the mechanical energy in each medium involving the kinetic and the potential energy as follows:
\begin{equation*}
\cal{E}^\sc{f}(t) := \frac{1}{2} \|\bd{m}(t) \|^2_{\bd{L}^2(\rho^\sc{f};\domain{f})}+\frac{1}{2} \|p(t)\|^2_{L^2(\frac{1}{\kappa};\domain{f})} , \quad \cal{E}^\sc{s}(t) := \frac{1}{2} \|\bd{v}(t)\|^2_{\bd{L}^2(\rho^\sc{s};\domain{s})} + \frac{1}{2} \|\bbm{s}(t)\|^2_{\bbm{L}^2(\bbm{C}^{-1};\domain{s})}.
\end{equation*}
The following result is well-known, but we present it for completeness.
\begin{lemma}[Energy balance]\label{lem:energy-cons} 
The following energy balance holds: For all $t\in\ol{J}$,
\begin{equation}   
\cal{E}(t) = \cal{E}(0) + \int_0^t \left\{ (f^{\sc{f}}(\tau),p(\tau))_{L^2(\domain{f})} + (\bd{f}^{\sc{s}}(\tau),\bd{v}(\tau))_{\bd{L}^2(\domain{s})} \right\} \rm{d}\tau.
\end{equation}
\end{lemma}
\begin{proof}
Testing  \eqref{weak_form_acoustic_eq} with $(\bd{m}(t),p(t))$ and \eqref{weak_form_elastic_eq} with $(\bbm{s}(t),\bd{v}(t))$ gives
\begin{subequations}
\begin{align}
(\partial_t \bd{m}(t), \bd{m}(t))_{\bd{L}^2(\rho^\sc{f};\domain{f})} - (\nabla p(t), \bd{m}(t))_{\bd{L}^2(\domain{f})} &= 0,\label{ec1}\\
(\partial_t p(t), p(t))_{L^2(\frac{1}{\kappa};\domain{f})} + (\bd{m}(t), \nabla p(t))_{\bd{L}^2(\domain{f})} + (\bd{v}(t) {\cdot} \bd{n}_\G, p(t))_{L^2(\Gamma)} &= (f^{\sc{f}}(t), p(t))_{L^2(\domain{f})},\label{ec2}
\end{align}
\end{subequations}
and
\begin{subequations}
\begin{align}
 (\partial_t\bbm{s}(t),\bbm{s}(t))_{\bbm{L}^2(\bbm{C}^{-1};\domain{s})} - (\nabla_{\rm{sym}} \bd{v}(t), \bbm{s}(t))_{\bbm{L}^2(\domain{s})} &= \bd{0},\label{ec3} \\
 (\partial_t \bd{v}(t),\bd{v}(t))_{\bd{L}^2(\rho^\sc{s};\domain{s})} + (\bbm{s}(t), \nabla_{\rm{sym}} \bd{v}(t))_{\bbm{L}^2(\domain{s})} - (p(t) \bd{n}_\G, \bd{v}(t))_{\bd{L}^2(\Gamma)}&= (\bd{f}^{s}(t),\bd{v}(t))_{\bd{L}^2(\domain{s})}.\label{ec4}
\end{align}
\end{subequations}
Summing \eqref{ec1}-\eqref{ec2} and \eqref{ec3}-\eqref{ec4}, we get
\begin{align*}
\frac{\mathrm{d}}{\mathrm{dt}} \mathcal{E}^\textsc{f}(t) & = (f^{\sc{f}}(t),p(t))_{L^2(\domain{f})} - (\bd{v}(t) {\cdot} \bd{n}_\G, p(t))_{L^2(\Gamma)},\\
\frac{\mathrm{d}}{\mathrm{dt}} \mathcal{E}^\textsc{s}(t) & = (\bd{f}^{\sc{s}}(t),\bd{v}(t))_{\bd{L}^2(\domain{s})} + (p(t) \bd{n}_\G, \bd{v}(t))_{\bd{L}^2(\Gamma)}.
\end{align*}
Summing the two equations and integrating over $(0,t)$ for all $t \in \ol{J}$ proves the claim.
\end{proof}

\section{HHO space semi-discretization}\label{sec::HHO_discretization}

This section presents the key ingredients of the HHO discretization, namely the discrete spaces and the discrete operators leading to the space semi-discrete HHO formulation. 

\subsection{Meshing and discrete spaces}

\noindent\textbf{Admissible mesh.} Let $\cal{T}$ be a polyhedral mesh of $\Omega$ that fits the partition of $\Omega$ into $\domain{s}$ and $\domain{f}$. For simplicity, we assume that all the material properties are piecewise constant on $\cal{T}$. We define  the two sub-meshes $\cal{T}^{\sc{s}}$ and $\cal{T}^{\sc{f}}$ which cover exactly $\domain{s}$ and $\domain{f}$, respectively.  The mesh faces are collected in the set $\cal{F}$ which is split into $\cal{F} := \cal{F}^{\circ} \cup \cal{F}^{\sd}$, where $\cal{F}^{\circ}$ collects all the mesh interfaces (inside $\Omega$, including on $\G$) and $\cal{F}^{\sd}$ collects all the mesh boundary faces on $\partial\Omega$. With obvious notation, we further decompose  $\cal{F}^{\circ} := \cal{F}^{\circ \sc{f}} \cup \cal{F}^{\circ \sc{s}} \cup \cal{F}^{\sg}$ and $\cal{F}^{\sd}:= \cal{F}^{\sd \sc{f}} \cup \cal{F}^{\sd \sc{s}}$. Later on, we also use the notation $\cal{M}^\sF:=(\T^\sF,\cF^\sF)$ and $\cal{M}^\sS:=(\T^\sS,\cF^\sS)$ where $\cF^\sF:=\cF^{\circ\sF}\cup \cF^{\sd\sc{f}}\cup \cF^{\sg}$ and $\cF^\sS:=\cF^{\circ\sS}\cup \cF^{\sd\sc{s}}\cup\cF^{\sg}$.  A generic mesh cell is denoted  $T \in \cal{T}$, its diameter $h_T$, its unit outward normal  $\bd{n}_T$, and the faces composing the boundary of $T$ are collected in the subset $\cal{F}_{\partial T}\subset \cF$.  We also set $\tih_T:=\frac{h_T}{\ell_\Omega}$,  where the scaling by $\ell_\Omega:=\text{diam}(\Omega)$ is introduced for dimensional consistency.\\
 
\noindent\textbf{Approximation spaces.} In each subdomain, we consider a mixed formulation with one primal variable ($p$ and $\bd{v}$) and one dual variable ($\bd{m}$ and $\bbm{s}$). The idea is to discretize the primal variables using the HHO method and the dual variables using a classical dG approach. Let $k \geq 1$ be the polynomial degree. The dG variables are piecewise polynomials of order $k$, whereas the HHO variables are composed of a pair with one cell component and one face component. The cell component is a piecewise polynomial of order $k^\prime\in \{k, k+1\}$ and the face component is a piecewise polynomial of order $k \geq 1$. The HHO discretization is said to be of equal-order if $k^\prime = k$ and of mixed-order if $k^\prime=k+1$.
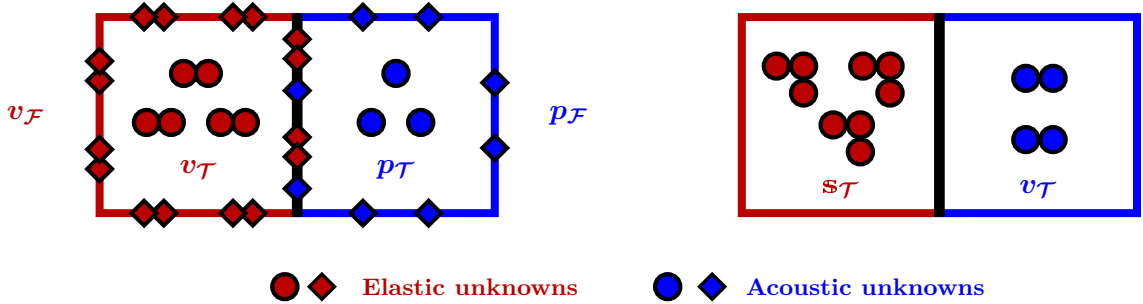
\begin{figure}[H]
\centering 
\begin{tikzpicture}[scale=0.65]
\begin{scope}[xshift=2cm]
\coordinate (A) at (-1,0);
\coordinate (B) at (3,0);
\coordinate (C) at (3,1);
\coordinate (D) at (3,3);
\coordinate (E) at (3,4);
\coordinate (F) at (-1,4);
\coordinate (G) at (-1,3);
\coordinate (H) at (-1,1);
\draw[line width = 3, color = ceared] (A) -- (B) -- (C) -- (D) -- (E) -- (F) -- (G) -- (H) -- cycle;
\coordinate (A2) at (3,0);
\coordinate (B2) at (7,0);
\coordinate (C2) at (7,1);
\coordinate (D2) at (7,3);
\coordinate (E2) at (7,4);
\coordinate (F2) at (3,4);
\coordinate (G2) at (3,3);
\coordinate (H2) at (3,1);
\draw[line width = 3, color = blue] (A2) -- (B2) -- (C2) -- (D2) -- (E2) -- (F2) -- (G2) -- (H2) -- cycle;
\draw[line width = 4, color = black] (3,4.075) -- (3,-0.075);
\filldraw[black, line width=1.5pt] ( 0.3, 0)  node[shape=diamond, draw, fill=ceared, inner sep=0pt, minimum size=10pt] {};
\filldraw[black, line width=1.5pt] (-0.1, 0)  node[shape=diamond, draw, fill=ceared, inner sep=0pt, minimum size=10pt] {};
\filldraw[black, line width=1.5pt] ( 2.1, 0)  node[shape=diamond, draw, fill=ceared, inner sep=0pt, minimum size=10pt] {};
\filldraw[black, line width=1.5pt] ( 1.7, 0)  node[shape=diamond, draw, fill=ceared, inner sep=0pt, minimum size=10pt] {};
\filldraw[black, line width=1.5pt] ( 0.3, 4)  node[shape=diamond, draw, fill=ceared, inner sep=0pt, minimum size=10pt] {};
\filldraw[black, line width=1.5pt] (-0.1, 4)  node[shape=diamond, draw, fill=ceared, inner sep=0pt, minimum size=10pt] {};
\filldraw[black, line width=1.5pt] ( 2.1, 4)  node[shape=diamond, draw, fill=ceared, inner sep=0pt, minimum size=10pt] {};
\filldraw[black, line width=1.5pt] ( 1.7, 4)  node[shape=diamond, draw, fill=ceared, inner sep=0pt, minimum size=10pt] {};
\filldraw[black, line width=1.5pt] ( 3, 1.55) node[shape=diamond, draw, fill=ceared, inner sep=0pt, minimum size=10pt] {};        
\filldraw[black, line width=1.5pt] ( 3, 1.15) node[shape=diamond, draw, fill=ceared, inner sep=0pt, minimum size=10pt] {};  
\filldraw[black, line width=1.5pt] ( 3, 0.5)  node[shape=diamond, draw, fill=blue, inner sep=0pt, minimum size=10pt] {};
\filldraw[black, line width=1.5pt] ( 3, 3.55) node[shape=diamond, draw, fill=ceared, inner sep=0pt, minimum size=10pt] {};        
\filldraw[black, line width=1.5pt] ( 3, 3.15) node[shape=diamond, draw, fill=ceared, inner sep=0pt, minimum size=10pt] {};  
\filldraw[black, line width=1.5pt] ( 3, 2.5)  node[shape=diamond, draw, fill=blue, inner sep=0pt, minimum size=10pt] {};
\filldraw[black, line width=1.5pt] (-1, 0.9)  node[shape=diamond, draw, fill=ceared, inner sep=0pt, minimum size=10pt] {};
\filldraw[black, line width=1.5pt] (-1, 1.3)  node[shape=diamond, draw, fill=ceared, inner sep=0pt, minimum size=10pt] {};   
\filldraw[black, line width=1.5pt] (-1, 2.7)  node[shape=diamond, draw, fill=ceared, inner sep=0pt, minimum size=10pt] {};
\filldraw[black, line width=1.5pt] (-1, 3.1)  node[shape=diamond, draw, fill=ceared, inner sep=0pt, minimum size=10pt] {};        
\filldraw[black, line width=1.5pt] ( 4.33, 0) node[shape=diamond, draw, fill=blue, inner sep=0pt, minimum size=10pt] {};
\filldraw[black, line width=1.5pt] ( 5.66, 0) node[shape=diamond, draw, fill=blue, inner sep=0pt, minimum size=10pt] {};
\filldraw[black, line width=1.5pt] ( 4.33, 4) node[shape=diamond, draw, fill=blue, inner sep=0pt, minimum size=10pt] {};
\filldraw[black, line width=1.5pt] ( 5.66, 4) node[shape=diamond, draw, fill=blue, inner sep=0pt, minimum size=10pt] {};
\filldraw[black, line width=1.5pt] ( 7, 1.33) node[shape=diamond, draw, fill=blue, inner sep=0pt, minimum size=10pt] {};
\filldraw[black, line width=1.5pt] ( 7, 2.66) node[shape=diamond, draw, fill=blue, inner sep=0pt, minimum size=10pt] {};
\draw[black, line width=4.5pt] (1.2,2.85) circle (5pt);
\fill[ceared] (1.2,2.85) circle (6pt);
\draw[black, line width=4.5pt] (0.7,2.85) circle (5pt);
\fill[ceared] (0.7,2.85) circle (6pt);
\draw[black, line width=4.5pt] (1.95,1.85) circle (5pt);
\fill[ceared] (1.95,1.85) circle (6pt);
\draw[black, line width=4.5pt] (1.45,1.85) circle (5pt);
\fill[ceared] (1.45,1.85) circle (6pt);
\draw[black, line width=4.5pt] (0.45,1.85) circle (5pt);
\fill[ceared] (0.45,1.85) circle (6pt);
\draw[black, line width=4.5pt] (-0.05,1.85) circle (5pt);
\fill[ceared] (-0.05,1.85) circle (6pt);
\draw[black, line width=4.5pt] (5,2.85) circle (5pt);
\fill[blue] (5,2.85) circle (6pt);
\draw[black, line width=4.5pt] (4.5,1.85) circle (5pt);
\fill[blue] (4.5,1.85) circle (6pt);
\draw[black, line width=4.5pt] (5.5,1.85)  circle (5pt);
\fill[blue] (5.5,1.85)  circle (6pt);
\node at (-2.5,2) {\color{ceared} $\bd{v_\cal{F}}$};
\node at (1,0.9) {\color{ceared} $\bd{v_\cal{T}}$};
\node at (8.5,2) {\color{blue} $\bd{p_\cal{F}}$};
\node at (5,0.9) {\color{blue} $\bd{p_\cal{T}}$};
\begin{scope}[xshift=13cm]
\coordinate (A) at (-1,0);
\coordinate (B) at (3,0);
\coordinate (C) at (3,1);
\coordinate (D) at (3,3);
\coordinate (E) at (3,4);
\coordinate (F) at (-1,4);
\coordinate (G) at (-1,3);
\coordinate (H) at (-1,1);
\draw[line width = 3, color = ceared] (A) -- (B) -- (C) -- (D) -- (E) -- (F) -- (G) -- (H) -- cycle;
\coordinate (A2) at (3,0);
\coordinate (B2) at (7,0);
\coordinate (C2) at (7,1);
\coordinate (D2) at (7,3);
\coordinate (E2) at (7,4);
\coordinate (F2) at (3,4);
\coordinate (G2) at (3,3);
\coordinate (H2) at (3,1);
\draw[line width = 3, color = blue] (A2) -- (B2) -- (C2) -- (D2) -- (E2) -- (F2) -- (G2) -- (H2) -- cycle;
\draw[line width = 4, color = black] (3,4.075) -- (3,-0.075);
\draw[black, line width=4.5pt] (0.25, 3.0)  circle (5pt); \fill[ceared] (0.25, 3.0)  circle (6pt);
\draw[black, line width=4.5pt] (-0.3, 3.0)  circle (5pt); \fill[ceared] (-0.3, 3.0)  circle (6pt);
\draw[black, line width=4.5pt] (0.25, 2.45) circle (5pt); \fill[ceared] (0.25, 2.45) circle (6pt);
\draw[black, line width=4.5pt] (2.0,  3.0)  circle (5pt); \fill[ceared] (2.0,  3.0) circle (6pt);
\draw[black, line width=4.5pt] (1.45, 3.0)  circle (5pt); \fill[ceared] (1.45, 3.0) circle (6pt);
\draw[black, line width=4.5pt] (2.0,  2.45) circle (5pt); \fill[ceared] (2.0,  2.45) circle (6pt);
\draw[black, line width=4.5pt] (1.4, 1.8) circle (5pt);  \fill[ceared] (1.4, 1.8) circle (6pt);
\draw[black, line width=4.5pt] (0.85,  1.8) circle (5pt);  \fill[ceared] (0.85,  1.8) circle (6pt);
\draw[black, line width=4.5pt] (1.4, 1.25) circle (5pt); \fill[ceared] (1.4, 1.25) circle (6pt);
\draw[black, line width=4.5pt] (4.75-0.75, 1.5) circle (5pt); \fill[blue]  (4.75-0.75, 1.5) circle (6pt);
\draw[black, line width=4.5pt] (5.3 -0.75,  1.5) circle (5pt); \fill[blue]  (5.3-0.75,  1.5) circle (6pt);
\draw[black, line width=4.5pt] (4.75+0.75, 1.5) circle (5pt); \fill[blue]  (4.75+0.75, 1.5) circle (6pt);
\draw[black, line width=4.5pt] (5.3 +0.75,  1.5) circle (5pt); \fill[blue]  (5.3+0.75,  1.5) circle (6pt);
\draw[black, line width=4.5pt] (4.75, 2.75) circle (5pt); \fill[blue]  (4.75, 2.75) circle (6pt);
\draw[black, line width=4.5pt] (5.3,  2.75) circle (5pt); \fill[blue]  (5.3,  2.75) circle (6pt);
\node at (1,0.5) {\color{ceared} $\bd{\bbm{s}_{\cal{T}}}$};
\node at (5,0.5) {\color{blue} $\bd{m_\cal{T}}$};
\end{scope}
\end{scope}
\begin{scope}[xshift=-2.5cm]
\begin{scope}[yshift=-2.25cm]
\draw[black, line width = 4.5] (7.25,0.75) circle (5pt);
\fill[ceared] (7.25,0.75) circle (6pt);
\filldraw[black, line width=1.25pt] (8,0.75) node[shape=diamond, draw, fill=ceared, inner sep=0pt, minimum size=10pt] {};
\node at (11,0.75) {$\footnotesize \color{ceared} \textbf{\color{ceared} Elastic unknowns}$};
\draw[black, line width = 4.5] (15,0.75) circle (5pt);
\fill[blue] (15,0.75) circle (6pt);
\filldraw[black, line width=1.25pt] (15.9,0.75) node[shape=diamond, draw, fill=blue, inner sep=0pt, minimum size=10pt] {};
\node at (19,0.75) {$\footnotesize \color{blue} \textbf{\color{blue} Acoustic unknowns}$};    
\end{scope}
\end{scope}
\end{tikzpicture}
\caption{Elasto-acoustic unknowns with a equal-order discretization $(k^\prime = k = 1$). \textbf{Left panel}: Primal variables discretized using HHO. \textbf{Right panel}: Dual variables discretized using dG.}
\label{coupling_unknowns}
\end{figure}
Let $\ell\geq 0$. We introduce the local polynomial spaces $\bbm{P}_{\rm{sym}}^\ell(T)$, $\bd{P}^\ell(T)$ and  $P^\ell(T)$ (resp. $\bd{P}^\ell(F)$ and  $P^\ell(F)$) as the restrictions to $T$ (resp. $F$) of symmetric tensor-, vector- and scalar-valued $d$-variate polynomials of degree at most $\ell$ (resp. $(d-1)$-variate polynomials of degree at most $\ell$).\\

\noindent \textbf{Acoustic wave equation.} The discrete dG and HHO spaces are
\begin{equation}
\bd{M}^{k}(\Tf) := \underset{T \in \Tf}{\bigtimes} \bd{P}^k(T), \qquad {\widehat{P}}^{k}(\Mf) := {P}^{k^\prime}(\Tf) \times {P}^{k}(\Ff),
\end{equation}
where $P^{k^\prime}(\Tf) := \underset{T \in \Tf}{\bigtimes} P^{k^\prime}(T)$ and $P^{k}(\Ff) := \underset{F \in \Ff}{\bigtimes} P^{k}(F)$. The global and local generic elements in $\widehat{P}^{k}(\Mf)$  are denoted  by
\begin{equation}
\hat{p}_{\Mf} := (p_{\Tf}, p_{\Ff})\in\widehat{P}^{k}(\Mf), \quad \hat{p}_T := (p_T, p_{\partial T}) \in \widehat{P}^{k}_T := P^{k^\prime}(T) \times P^{k}(\cF_{\partial T})\quad\forall T \in\Tf,
\end{equation}
where $p_{\partial T} := (p_F)_{F \in \cal{F}_{\partial T}}$ and $P^{k}(\cF_{\partial T}) := \underset{F \in \cal{F}_{\partial T}}{\bigtimes} P^{k}(F)$. Moreover, to enforce the homogeneous Dirichlet boundary condition on $p$, we consider the subspace 
\begin{equation}
{\widehat{P}}^{k}_0(\Mf) := \left\{\hat{p}_{\Mf} \in \widehat{P}^{k}(\Mf) \mid p_F = 0, ~ \forall F \in \cal{F}^{\sd\sc{f}} \right\}.
\end{equation}
For all $t\in\ol{J}$, we approximate the pressure $p(t)$ by the HHO unknown $\hat{p}_{\Mf}(t) \in \widehat{P}^{k}_0(\Mf)$ and the velocity  $\bd{m}(t)$ by the dG unknown $\bd{m}_{\Tf}(t) \in \bd{M}^{k}(\Tf)$.\\

\noindent \textbf{Elastic wave equation.} The discrete dG and HHO spaces are
\begin{equation}
\Usigmas := \underset{T \in \cal{T}^{\sc{s}}_h}{\bigtimes} \bbm{P}_{\rm{sym}}^k(T), \qquad \bd{\widehat{V}}{}^{k}(\Ms) := \bd{{V}}^{k^\prime}(\Ts) \times \bd{{V}}^{k}(\Fs), 
\end{equation}
where $\bd{{V}}^{k^\prime}(\Ts) := \underset{T \in \Ts}{\bigtimes} \bd{P}^{k^{\prime}}(T)$ and $\bd{{V}}^{k}(\Fs) := \underset{F \in \Fs}{\bigtimes} \bd{P}^k(F)$. The global and local generic elements in $\widehat{\bd{V}}{}^{k}(\Ms)$ are denoted by  
\begin{equation}
\hat{\bd{v}}_{\Ms} := (\bd{v}_{\Ts}, \bd{v}_{\Fs})\in \bd{\widehat{V}}{}^{k}(\Ms), \quad \hat{\bd{v}}_T := (\bd{v}_T, \bd{v}_{\partial T}) \in \bd{\widehat{V}}{}^{k}_{T} := \bd{P}^{k^{\prime}}(T) \times \bd{P}^k(\cF_{\partial T})\quad \forall T \in \Ts,
\end{equation}
where $ \bd{v}_{\partial T} := (\bd{v}_F)_{F \in \cal{F}_{\partial T}}$ and $\bd{P}^{k}(\cF_{\partial T}) := \underset{F \in \cal{F}_{\partial T}}{\bigtimes} \bd{P}^{k}(F)$. Moreover, to enforce the homogeneous Dirichlet boundary condition on $\bd{v}$, we consider the subspace 
\begin{equation}
\bd{\widehat{V}}{}^{k}_0(\Ms) := \left\{\hat{\bd{v}}_{\Ms} \in \bd{\widehat{V}}{}^{k}(\Ms) \mid \bd{v}_F =\bd{0}, ~ \forall F \in \cal{F}^{\sd\sc{s}} \right\}.
\end{equation}
For all $t\in\ol{J}$, we approximate the velocity $\bd{v}(t)$ by the HHO unknown $\hat{\bd{v}}_{\Ms}(t) \in \bd{\widehat{V}}{}^{k}_0(\Ms)$ and the stress tensor $\bbm{s}(t)$ by the dG unknown $\bbm{s}_{\Ts}(t) \in \Usigmas$.\\

\noindent \textbf{$\bd{L}^2$-orthogonal projections.} Let $\Pi_T^{k^{\prime}}$ (resp., $\Pi_F^k$ and  $\Pi_{\partial T}^k$ ) be the local $L^2(T)$- (resp., $L^2(F)$- and $L^2(\partial T)$-) orthogonal projection onto $P^{k^\prime}(T)$ (resp., $P^k(F)$ and $P^{k}(\cF_{\partial T})$). Let $\Pi_{\T^\bullet}^{k^{\prime}}$ (resp. $\Pi_{\cF^\bullet}^k$) be the global $L^2$-orthogonal projections onto $P^{k^\prime}(\T^\bullet)$ (resp. $P^k(\cF^\bullet)$) with $\bullet\in\{\sc{f},\sc{s}\}$. A similar notation is used for vector- and tensor-valued fields. 

\subsection{HHO local operators}

\noindent The HHO discretization is formulated locally using the following two key operators:
\begin{enumerate}[label=\roman*)]
\item a local gradient reconstruction operator for the acoustic wave equation and a local symmetric gradient reconstruction operator for the elastic wave equation; 
\item a stabilization operator that penalizes the difference between the trace of the cell unknowns and the face unknowns for the HHO components on both subdomains. 
\end{enumerate}

\noindent The discrete problem is then assembled by summing the contributions of all the mesh cells.\\

\noindent \textbf{Acoustic wave equation.} We define the local gradient reconstruction operator $\bd{g}_T: {\widehat{P}}^{k}_{T} \rightarrow \bd{P}^k(T)$ such that, for all $\hat{p}_T \in {\widehat{P}}^{k}_{T}$, 
\begin{equation}
(\bd{g}_T(\hat{p}_T), \bd{r})_{\bd{L}^2(T)} = (\nabla p_T, \bd{r})_{\bd{L}^2(T)} - (p_{T} - p_{\partial T}, \bd{r} {\cdot} \bd{n}_T)_{L^2(\partial T)}, \qquad \forall \bd{r} \in \bd{P}^k(T).\label{def:G}
\end{equation}
Notice that $\bd{g}_T(\hat{p}_T)$ can be evaluated componentwise by inverting the mass matrix associated with a basis of the scalar-valued polynomial space $P^k(T)$. 

We define the local stabilization operator $S_{\dt}: {\widehat{P}}^{k}_{T} \rightarrow P^k(\cF_{\dt})$  such that, for all $\hat{p}_T \in \widehat{P}^{k}_{T}$, 
\begin{equation}
S_{\partial T}(\hat{p}_T) := \Pi_{\partial T}^k (\delta_{\partial T}(\hat{p}_T)) \quad \text{with} \quad \delta_{\partial T}(\hat{p}_T) := p_{T}|_{ \partial T} - p_{\partial T}.\label{stab-op-F}
\end{equation}
We define the global gradient reconstruction operator $\bd{g}_{\Tf} : \widehat{P}^{k}(\Mf) \rightarrow \bd{M}^{k}(\Tf)$ as $\bd{g}_{\Tf}(\hat{p}_{\cred{\Mf}})|_{T} := \bd{g}_T(\hat{p}_T)$ for all $T \in \cal{T}^{\sc{f}}$ and all $\hat{p}_{\Mf} \in \widehat{P}^{k}(\Mf)$, and the global stabilization bilinear form $s_{\Mf}$ on $\widehat{P}^{k}(\Mf) \times \widehat{P}^{k}(\Mf)$ as
\begin{equation}
s_{\Mf}(\hat{p}_{\Mf}, \hat{q}_{\Mf}) := \sum_{T \in \Tf} \tau^{\sc{f}}_{T} (S_{\partial T}(\hat{p}_T), S_{\partial T}(\hat{q}_T))_{L^2(\partial T)}, \qquad \forall \hat{p}_{\Mf}, \hat{q}_{\Mf} \in \widehat{P}^{k}(\Mf),\label{stab-F}
\end{equation}
\noindent where, for all $T \in \Tf$, the stabilization parameter $\tau^\sc{f}_{T}>0$ is taken equal to (see also \hyperref[rem::stab_strategies]{\Cref{rem::stab_strategies}} below)
\begin{align}
\tau^\sc{f}_{T} := \cred{\zeta}^\sc{f} \tih_T^{-\alpha}\quad\text{with} \quad \cred{\zeta}^\sc{f} := (\rho^\sc{f}c_\sc{p}^\sc{f})^{-1} = c_\sc{p}^\sc{f}\kappa^{-1},\;\alpha\in\{0,1\}.\label{stab-param-F}
\end{align} 
\noindent \textbf{Elastic wave equation.}
We define the local symmetric gradient reconstruction operator $\bbm{g}_T^{\rm{sym}}: \bd{\widehat{V}}{}^{k}_{T} \rightarrow \bbm{P}_{\rm{sym}}^k(T)$ such that, for all $\hat{\bd{v}}_T \in \bd{\widehat{V}}{}^{k}_{T}$, 
\begin{equation}
(\bbm{g}_T^{\rm{sym}}(\hat{\bd{v}}_T), \bbm{b})_{\bbm{L}^2(T)} = (\nabla_{\rm{sym}} \bd{v}_T, \bbm{b})_{\bbm{L}^2(T)} - (\bd{v}_T - \bd{v}_{\partial T}, \bbm{b} {\cdot} \bd{n}_T)_{\bd{L}^2(\partial T)},\qquad\forall  \bbm{b} \in \bbm{P}_{\rm{sym}}^k(T).\label{g-elastic}
\end{equation}
Notice that $\bbm{g}_T^{\rm{sym}}(\cred{\hat{\bd{v}}_T})$ can be evaluated componentwise by inverting the mass matrix associated with a basis of the scalar-valued polynomial space $P^k(T)$. 

We define the local stabilization operator $\bd{S}_{\partial T}:\bd{\widehat{V}}{}^{k}_{T} \rightarrow \bd{P}^k(\cF_{\dt})$ such that,  for all $\hat{\bd{v}}_T \in \bd{\widehat{V}}{}^{k}_{T}$,
\begin{equation}
\bd{S}_{\partial T}(\hat{\bd{v}}_T):=\bd{\Pi}_{\partial T}^k(\bd{\delta}_{\partial T} (\hat{\bd{v}}_T))\quad
\text{with}\quad \bd{\delta}_{\partial T}(\hat{\bd{v}}_T) := \bd{v}_{T} |_{\partial T} - \bd{v}_{\partial T}.\label{stab-op-S}
\end{equation}
We define the global symmetric gradient reconstruction operator $\bbm{g}_{\Ts}^{\rm{sym}} : \bd{\widehat{V}}{}^{k}(\Ms) \rightarrow \bbm{S}^{k}_{\rm{sym}}(\Ts)$ as $\bbm{g}_{\Ts}^{\rm{sym}} (\hat{\bd{v}}_{\Ms})|_T := \bbm{g}_T^{\rm{sym}}(\hat{\bd{v}}_{\cred{T}})$ for all $T \in \cal{T}^\sc{s}$ and all $\hat{\bd{v}}_{\Ms} \in \bd{\widehat{V}}{}^{k}(\Ms)$, and the global stabilization bilinear form $s_{\Ms}$ on $\bd{\widehat{V}}{}^{k}(\Ms) \times \bd{\widehat{V}}{}^{k}(\Ms)$ as
\begin{equation}
s_{\Ms}(\hat{\bd{v}}_{\Ms} , \hat{\bd{w}}_{\Ms}) := \sum_{T \in \Ts} \tau^\sc{s}_{T}(\bd{S}_{\partial T} (\hat{\bd{v}}_T), \bd{S}_{\partial T} (\hat{\bd{w}}_T))_{\bd{L}^2(\partial T)},\qquad\forall \hat{\bd{v}}_{\Ms}, \hat{\bd{w}}_{\Ms} \in \bd{\widehat{V}}{}^{k}(\Ms),\label{stab-S}
\end{equation}
where, for all $T \in \Ts$, the stabilization parameter $\tau^\sc{s}_{T}>0$ is taken equal to 
\begin{align}
\tau^\sc{s}_{T} = \cred{\zeta}^\sc{s} \tih_T^{-\alpha} \quad\text{with}\quad \cred{\zeta}^\sc{s}:=\rho^\sc{s}c^\sc{s},\quad\alpha\in\{0,1\},\label{stab-param-S}
\end{align}
and $c^\sc{s}$ can be any of the two wave speeds defined in \eqref{cs}. 
\begin{remark}[Stabilization parameter] \label{rem::stab_strategies}
Notice from \eqref{stab-op-F} and \eqref{stab-op-S} that the stabilization operators considered in the paper are the same regardless of the discretization setting (equal- or mixed-order) and correspond to plain least-squares stabilization in the equal-order setting. This choice is standard for $\mathcal{O}(1)$-stabilization ($\alpha=0$) in the first-order formulation of wave problems discretized using dG methods and was also considered recently in \cite{EK_2024} in the context of HHO methods. However, $\mathcal{O}(\frac{1}{h})$-stabilization ($\alpha=1$) can be useful in certain situations (see \hyperref[rem-2]{\Cref{rem-2}} and the numerical results from \hyperref[sec::numerical_results]{\Cref{sec::numerical_results}} for further discussion). Notice also that the scaling of $\cred{\zeta}^\sc{f}$ in (\ref{stab-param-F}) and $\cred{\zeta}^\sc{s}$ in (\ref{stab-param-S}) differs; this is actually the physically consistent scaling and stems from the fact that the primal variables in both subdomains have different \cred{physical dimensions}.
\end{remark}

\subsection{HHO discretization for the first-order coupling formulation}

The space semi-discrete problem for the coupled elasto-acoustic wave problem reads as follows: Find $(\bd{m}_{\Tf},\hat{p}_{\Mf})\in C^1(\ol{J};\bd{M}^{k}(\Tf) \times {\widehat{P}}^{k}_0(\Mf))$ and $(\bbm{s}_{\Ts},\hat{\bd{v}}_{\Ms})\in  C^1(\ol{J};\Usigmas \times \bd{\widehat{V}}{}^{k}_0(\Ms))$ such that, for all $(\bd{r}_{\Tf}, \hat{{q}}_{\Mf}) \in \bd{M}^{k}(\Tf) \times {\widehat{P}}^{k}_0(\Mf)$ and all $(\bbm{b}_{\Ts},\hat{\bd{w}}_{\Ms}) \in \Usigmas \times \bd{\widehat{V}}{}^{k}_0(\Ms)$, and all $t\in \ol{J}$,
\begin{subequations}\label{HHO1}
\begin{align}
\label{acoustic1}
(\partial_t \bd{m}_{\Tf}(t), \bd{r}_{\Tf})_{{\bd{L}^2(\rho^\sc{f};\domain{f})}} & - (\bd{g}_{\Tf}(\hat{p}_{\Mf}(t)), \bd{r}_{\Tf})_{{\bd{L}^2(\domain{f})}} = 0,\\
(\partial_t p_{\Tf}(t),q_{\Tf})_{{L^2(\frac{1}{\kappa};\domain{f})}} & + (\bd{m}_{\Tf}(t), \bd{g}_{\Tf}(\hat{q}_{\Mf}))_{{\bd{L}^2(\domain{f})}} \nonumber\\
 & + s_{\Mf}(\hat{p}_{\Mf}(t),\hat{q}_{\Mf}) + (\bd{v}_{\Fs}(t) {\cdot} \bd{n}_{\G}, q_{\Ff})_{L^2(\G)} = (f^{\sc{f}}(t),q_{\Tf})_{L^2(\domain{f})}, \label{acoustic2}
\end{align}
\end{subequations}  
and 
\begin{subequations}
\begin{align}
\label{elastic1}
(\partial_t \bbm{s}_{\Ts}(t), \bbm{b}_{\Ts})_{\bbm{L}^2(\bbm{C}^{-1};\domain{s})} & - (\bbm{g}^{\rm{sym}}_{\Ts}(\hat{\bd{v}}_{\Ms}(t)), \bbm{b}_{\Ts})_{\bbm{L}^2(\domain{s})} = 0,\\
\label{elastic2}
(\partial_t \bd{v}_{\Ts}(t), \bd{w}_{\Ts})_{\bd{L}^2(\rho^\sc{s};\domain{s})} & + (\bbm{s}_{\Ts}(t), \bbm{g}_{\Ts}^{\rm{sym}}(\hat{\bd{w}}_{\Ms}))_{\bbm{L}^2(\domain{s})} \nonumber\\
& + s_{\Ms}(\hat{\bd{v}}_{\Ms}(t), \hat{\bd{w}}_{\Ms}) -(p_{\Ff}(t) \bd{n}_{\G},\bd{w}_{\Fs})_{\bd{L}^2(\G)} = (\bd{f}^{\sc{s}}(t), \bd{w}_{\Ts})_{\bd{L}^2(\domain{s})}.
\end{align}
\label{HHO2}
\end{subequations}  
The initial conditions for the discrete coupled problem $\bd{m}_{\Tf}(0),\hat{p}_{\Mf}(0), \bbm{s}_{\Ts}(0)$, and $\hat{\bd{v}}_{\Ms}(0)$ are further discussed in Section~4. 

We emphasize the seamless enforcement of the coupling conditions in \eqref{acoustic2} and \eqref{elastic2} exploiting the fact that face unknowns are readily available on $\G$ in the HHO setting. Moreover, the total discrete mechanical energy over the elasto-acoustic domain $\Omega$ is defined as $\cal{E}_h(t) := \cal{E}_h^\sc{f}(t) + \cal{E}_h^\sc{s}(t)$ with 
\begin{equation*}
\cred{\cal{E}_h^\sc{f}(t) := \frac{1}{2} \|\bd{m}_{\Tf}(t)\|^2_{\bd{L}^2(\rho^\sc{f};\domain{f})} + \frac{1}{2} \|p_{\Tf}(t) \|^2_{L^2(\frac{1}{\kappa};\domain{f})}, \quad \cal{E}_h^\sc{s}(t) := \dfrac{1}{2} \|\bd{v}_{\Ts}(t)\|^2_{\bd{L}^2(\rho^\sc{s}; \domain{s})} + \frac{1}{2} \|\bbm{s}_{\Ts}(t)\|^2_{\bbm{L}^2(\bbm{C}^{-1};\domain{s})}.}
\end{equation*}
\begin{lemma}[Semi-discrete energy balance]
The following discrete energy balance holds: For all $t\in\ol{J}$,
\begin{multline}
\cal{E}_h(t) + \int_0^t \Big\{s_{\Mf}(\hat{p}_{\cred{\Mf}}(\tau), \hat{p}_{\cred{\Mf}}(\tau)) + s_{\Ms}(\hat{\bd{v}}_{\cred{\Ms}}(\tau), \hat{\bd{v}}_{\cred{\Ms}}(\tau))\Big\} \rm{d\tau} = \\ 
\cal{E}_h(0) + \int_0^t \Big\{(f^{\sc{f}}(\tau), p_{\cred{\Tf}}(\tau))_{L^2(\domain{f})} + (\bd{f}^{\sc{s}}(\tau), \bd{v}_{\cred{\Ts}}(\tau))_{\bd{L}^2(\domain{s})} \Big\} \rm{d\tau}.
\label{discrete_energy}
\end{multline}
\end{lemma}
\begin{proof}
Similar to the proof of Lemma~\ref{lem:energy-cons}.
\end{proof}

\section{Error analysis}\label{sec::Error_analysis}

In this section, we prove an energy-error estimate for the space semi-discrete problem \eqref{HHO1}-\eqref{HHO2} by using suitable interpolation operators. Here onwards, the inequality $a \leq Cb$ for positive numbers $a$ and $b$ is abbreviated as
$a \lesssim b$, where the value of $C$ is independent of the mesh-size $h$, the material parameters, the length scale $\ell_\Omega$, and the time scale $T_{\rm{f}}$. The value of $C$ can depend on the mesh shape-regularity, the polynomial degree,  the space dimension, and the ratio $\frac{\lambda}{\mu}$.

\subsection{Interpolation operators}
Inspired \cred{by} \cite{EK_2024}, we use, in both subdomains, the H$^+$ interpolation operator from \cite{DS_2021}  to approximate  the dG variable and the classical HHO interpolation operator (based on $L^2$-orthogonal projections) to approximate the HHO variable. 

In the acoustic part, we employ the H$^+$ interpolation operator $\bd{I}^{\hdg}_T:\bH^{\nu}(T)\to \bd{P}^k(T)$, $\nu\in(\frac{1}{2},1]$, defined for all $T\in\cal{T}^{\sc{f}}$ as follows: We consider the $L^2$-orthogonal decomposition 
\begin{equation} 
\bd{P}^k(T) = \nabla P^{k+1}_{\mbox{*}}(T)\oplus \bd{Z}^k(T),
\end{equation}
where $P^{k+1}_{\mbox{*}}(T):=\{q\in P^{k+1}(T):(q,1)_{L^2(T)}=0\}$ and $\bd{Z}^k(T):= \nabla P^{k+1}_{\mbox{*}}(T)^\perp\cap \bd{P}^k(T)$ (orthogonalities are understood in $\bd{L}^2$). Then, for all $\bm\in\bH^{\nu}(T)$, we define  $\bd{I}^{\hdg}_T(\bm)\in \bd{P}^k(T)$  from the following conditions: 
\begin{subequations} 
\begin{alignat}{2}
(\bd{I}^{\hdg}_T(\bm)-\bm, \br)_{\bd{L}^2(T)} &= 0 &\qquad&\forall\br\in \bd{Z}^k(T),\label{hdg:1a}\\
(\bd{I}^{\hdg}_T(\bm)-\bm, \nabla q)_{\bd{L}^2(T)} &= (\Pi^k_{\dt}(\bm{\cdot}\bd{n}_T)-\bm{\cdot}\bd{n}_T, q)_{L^2(\dt)}&\qquad&\forall q\in P^{k+1}_{\mbox{*}}(T).\label{hdg.1b}
\end{alignat} 
\end{subequations}
Notice that \eqref{hdg.1b} actually holds true for all $q\in P^{k+1}(T)$. For all $T\in \Tf$, the interpolation operator $\bd{I}^{\hdg}_T$ is well-defined and the following holds for all $\ell_{\bm}\in[\nu,k+1]$ (see \cite[Prop.~2.1]{DS_2021} for a proof):
\begin{align}
\|\bd{I}_T^{\hdg}(\bm)-\bm\|_{\bd{L}^2(T)}+h_T^{\frac{1}{2}}\|\bd{I}_T^{\hdg}(\bm)-\bm\|_{\bd{L}^2(\dt)} \lesssim h_T^{\ell_{\bm}}|\bm|_{\bH^{\ell_{\bm}}(T)}.\label{est:approx-h+}
\end{align}
The global interpolation operator $\bd{I}^{\hdg}_{\Tf}:\bH^{\nu}(\domain{f})\to \bd{P}^k(\Tf)$ is defined as $(\bd{I}^{\hdg}_{\Tf}(\bm))|_T:=\bd{I}^{\hdg}_T(\bm|_T)$ for all $T\in\Tf$ and all $\bm\in\bH^{\nu}(\domain{f})$.  For the HHO variable, we employ the standard HHO interpolation operator $\hat{I}_{\Mf}^{\hho}:H^1_{0\G}(\domain{f})\to \UhatfD$ defined, for all $p\in H^1_{0\G}(\domain{f})$, by
\begin{align}
\hat{I}_{\Mf}^{\hho}(p):=(\Pi^{k'}_{\Tf}(p),\Pi^k_{\Ff}(p|_{\Ff}))\in \UhatfD.
\end{align}
 
In the elastic part,  we employ the H$^+$ interpolation operator $\bbm{I}^{\hdg}_{\Ts}:\bbm{H}^{\nu}_{\text{sym}}(\domain{s})\to \bbm{P}^k_{\text{sym}}(\Ts)$, $\nu\in(\frac{1}{2},1]$, defined for all $T\in\Ts$ as follows: We consider the $L^2$-orthogonal decomposition 
\begin{equation}
\bbm{P}_{\text{sym}}^k(T) = \nabla_{\rm{sym}} \bd{P}^{k+1}_{\mbox{*}}(T)\oplus \bbm{Z}^k_{\text{sym}}(T),  
\end{equation}
where $\bd{P}^{k+1}_{\mbox{*}}(T):=[P_*^{k+1}(T)]^d:=\{\bd{q}\in \bd{P}^{k+1}(T):(\bd{q},\bd{e}_i)_{\bd{L}^2(T)}=0\quad\forall i\in\{1,\dots,d\}\}$ with  the canonical basis $(\bd{e}_i)_{i\in\{1,\dots,d\}}$ of $\bb{R}^d$, and $\bbm{Z}^k_{\text{sym}}(T):= \nabla_{\rm{sym}} \bd{P}^{k+1}_{\mbox{*}}(T)^\perp\cap \bbm{P}^k_{\text{sym}}(T)$. Then, for all $\bbm{s}\in  \bbm{H}^{\nu}_{\text{sym}}(\domain{s})$, we define $\bbm{I}^{\hdg}_T(\bbm{s})\in \bbm{P}^k_{\text{sym}}(T)$ from the following conditions:  
\begin{subequations} 
\begin{alignat}{2}
(\bbm{I}^{\hdg}_T(\bbm{s})-\bbm{s}, \bbm{b})_{\bbm{L}^2(T)} &= 0 &\qquad&\forall\bbm{b}\in \bbm{Z}_{\text{sym}}^k(T),\label{hdgs:1a}\\
(\bbm{I}^{\hdg}_T(\bbm{s})-\bbm{s}, \nabla_{\rm{sym}} \bd{w})_{\bbm{L}^2(T)} &= (\bd{\Pi}^k_{\dt}(\bbm{s}{\cdot}\bd{n}_T)-\bbm{s}{\cdot}\bd{n}_T, \bd{w})_{\bd{L}^2(\dt)}&\qquad&\forall \bd{w}\in \bd{P}^{k+1}_{\mbox{*}}(T).\label{hdgs.1b}
\end{alignat} 
\end{subequations}
For all $T\in\Ts$, the interpolation operator $\bbm{I}^{\hdg}_{T}$ is well-defined and the following holds for all $\ell_{\bbm{s}}\in[\nu,k+1]$ (see \cite[Prop.~3.1]{DS_2021} for a proof):
\begin{align}
\|\bbm{I}_T^{\hdg}(\bbm{s})-\bbm{s}\|_{\bbm{L}^2(T)}+h_T^{\frac{1}{2}}\|\bbm{I}_T^{\hdg}(\bbm{s})-\bbm{s}\|_{\bbm{L}^2(\dt)} \lesssim h_T^{\ell_{\bbm{s}}}|\bbm{s}|_{\bbm{H}^{\ell_{\bbm{s}}}(T)}.\label{est:approxs-h+}
\end{align}
The global interpolation operator $\bbm{I}^{\hdg}_{\Ts}:\bbm{H}^{\nu}_{\text{sym}}(\domain{s})\to \bbm{P}^k_{\text{sym}}(\Ts)$ is defined as $(\bbm{I}^{\hdg}_{\Ts}(\bbm{s}))|_T:=\bbm{I}^{\hdg}_T(\bbm{s}|_T)$ for all $T\in\Ts$ and all $\bbm{s}\in\bbm{H}^{\nu}_{\text{sym}}(\domain{s})$. For the HHO variable, we employ the HHO interpolation operator $\hat{\bd{I}}{}_{\Ms}^{\hho}:\bH^1_{0\G}(\domain{s})\to \UhatsD$ defined, for all $\bd{v}\in \bH^1_{0\G}(\domain{s})$, by
\begin{align}
\hat{\bd{I}}{}_{\Ms}^{\hho}(\bd{v}):=(\bd{\Pi}^{k'}_{\Ts}(\bd{v}),\bd{\Pi}^k_{\Fs}(\bd{v}|_{\Fs}))\in \UhatsD.
\end{align}

Now, we define some notation to be used for the error analysis in the next section.  For all $\hat{p}_{\Mf}\in\UhatfD$ and all $\hat{\bd{v}}_{\Ms}\in\UhatsD$, the HHO norms are (classically) defined as follows:
\begin{align}
\|\hat{p}_{\Mf}\|_{\hho,\sF}^2 : =\sum_{T\in\Tf}\cred{\zeta}^\sc{f}\Big(\|\nabla p_T\|_{\bd{L}^2(T)}^2+ h_T^{-1}\|p_{\dt}-p_T\|^2_{L^2(\dt)}\Big),
\end{align}
and
\begin{align}
\|\hat{\bd{v}}_{\Ms}\|_{\hho,\sS}^2 : =\sum_{T\in\Ts}\cred{\zeta}^\sc{s}\Big(\|\nabla_{\rm{sym}} \bd{v}_T\|_{\bbm{L}^2(T)}^2+ h_T^{-1}\|\bd{v}_{\dt}-\bd{v}_T\|^2_{\bd{L}^2(\dt)}\Big),
\end{align}
where $\cred{\zeta}^\sF$ and $\cred{\zeta}^\sS$ are defined in \eqref{stab-param-F} and \eqref{stab-param-S}, respectively. We set  
\begin{equation}
|\hat{p}_{\Mf}|_{\text{S}^{\sF}}^2:=s_{\Mf}(\hat{p}_{\Mf}, \hat{p}_{\Mf}) \quad \text{and} \quad |\hat{\bd{v}}_{\Ms}|_{\text{S}^{\sS}}^2:=s_{\Ms}(\hat{\bd{v}}_{\Ms}, \hat{\bd{v}}_{\Ms}).
\end{equation}
For linear functionals $\phi_{\Mf}\in(\UhatfD)^\prime$  and $\phi_{\Ms}\in (\UhatsD)^\prime$, we define the following quantities: 
\begin{subequations}\label{seminorms}
\begin{alignat}{2}
\|\phi_{\Mf}\|_{(\hho, \sF)'}&:= \sup_{\hat{q}_{\Mf}\in \UhatfD } \frac{|\phi_{\Mf}(q_{\Mf})|}{\|\hat{q}_{\Mf}\|_{\hho, \sF}},\quad \|\phi_{\Ms}\|_{(\hho, \sS)'}&:= \sup_{\hat{\bw}_{\Ms}\in \UhatsD } \frac{|\phi_{\Ms}(\hat{\bw}_{\Ms})|}{\|\hat{\bw}_{\Ms}\|_{\hho, \sS}},\\
\|\phi_{\Mf}\|_{(\text{S}^\sF)'}&:= \sup_{\hat{q}_{\Mf}\in \UhatfD } \frac{|\phi_{\Mf}(\hat{q}_{\Mf})|}{|\hat{q}_{\Mf}|_{\text{S}^\sF}},\quad \|\phi_{\Ms}\|_{(\text{S}^\sS)'}&:= \sup_{\hat{\bw}_{\Ms}\in \UhatsD } \frac{|\phi_{\Ms}(\hat{\bw}_{\Ms})|}{|\hat{\bw}_{\Ms}|_{\text{S}^\sS}}.\label{seminorm2}
\end{alignat}
\end{subequations}
The seminorms in \eqref{seminorm2} may be unbounded for general linear functionals $\phi_{\Mf}$ and $\phi_{\Ms}$, but we will see that they remain bounded for the consistency errors. For all $(\bm,p)\in\bd{H}^\nu(\domain{f})\times H^1_{0\G}(\domain{f})$, $\nu\in (\frac{1}{2},1]$, we define the augmented seminorm 
\begin{align}
|(\bm,p)|^2_{\#,\sF}:= \sum_{T\in\Tf}\left\{\tih_T^{\alpha}\|\bd{B}_{\dt}(\bm){\cdot}\bd{n}_T\|^2_{L^2(\frac{1}{\cred{\zeta}^\sc{f}};\partial T)}+h_T\tih^{-\alpha}_T\|\nabla B_T(p)\|^2_{\bd{L}^2(\cred{\zeta}^\sc{f};T)}\right\},\label{19}
\end{align}
with $\bd{B}_{\dt}(\bm):=(\bm-\bd{I}^{\hdg}_T(\bm))|_{\dt}$ and $B_T(p):=p-\Pi_T^{k^\prime}(p)$ for all $T\in\Tf$. For all $(\bbm{s},\bd{v})\in\bbm{H}^\nu(\domain{s})\times \mathbf{H}^1_{0\G}(\domain{s})$, $\nu\in(\frac{1}{2},1]$, we define the augmented seminorm
\begin{align}
|(\bbm{s},\bd{v})|^2_{\#,\sS}:= \sum_{T\in\Ts}\left\{\tih_T^{\alpha}\|\bbm{B}_{\dt}(\bbm{s}){\cdot}\bd{n}_T\|^2_{\bd{L}^2(\frac{1}{\cred{\zeta}^\sc{s}};\partial T)}+h_T\tih^{-\alpha}_T\|\nabla_{\rm{sym}}\bd{B}_T(\bd{v})\|^2_{\bbm{L}^2(\cred{\zeta}^\sc{s};T)}\right\},\label{20}
\end{align}
with $\bbm{B}_{\dt}(\bbm{s}):=(\bbm{s}-\bbm{I}^{\hdg}_T(\bbm{s}))|_{\dt}$ and $\bd{B}_T(\bd{v}):=\bd{v}-\bd{\Pi}_T^{k^\prime}(\bd{v})$ for all $T\in\Ts$. 

Finally, for all $t \in (0,T_{\rm{f}}]$ and the time interval $J_t:=(0,t)$, we set the following notation:
\begin{align} 
\|q\|_{L^p(J_t;*)}^{p}:=\int_{J_t}\|q(\tau)\|_{*}^p\,\rm{d\tau}\quad\text{for all}\;p\in[1,\infty),\qquad \|q\|_{C^0(\ol{J}_t;*)}:=\sup_{s\in\ol{J}_t}\|q(s)\|_{*},
\end{align}
where $\|\cdot\|_{*}$ is a seminorm or a norm depending on the context.  
\subsection{Energy-error estimate}
We are now ready to state and prove our main error estimate. 
\begin{theorem}[Energy-error estimate]\label{thm:err-est}
Let $(\bd{m}, p)$ and $(\bbm{s},\bd{v})$ solve \eqref{weak_form_acoustic_eq} and \eqref{weak_form_elastic_eq} with the initial conditions \eqref{IC_BC_acoustic} and \eqref{IC_BC_elastic}, respectively, and assume that $\bd{m} \in C^1(\overline{J};\bd{H}^\nu(\domain{f})) $ and $\bbm{s}\in C^1(\overline{J};\bbm{H}^\nu(\domain{s}))$, $\nu\in(\frac{1}{2},1]$.  Let $(\bd{m}_{\Tf},\hat{p}_{\Mf})$ and $(\bbm{s}_{\Ts},\hat{\bd{v}}_{\Ms})$ solve \eqref{HHO1} and \eqref{HHO2} with the initial conditions 	$\hat{p}_{\Mf}(0)  = \hat{I}_{\Mf}^{\hho}(p_0),  \bd{m}_{\Tf}(\bd{0}) = \bd{I}_{\Tf}^{\hdg}(\bd{m}_0), \hat{\bd{v}}_{\Ms}(\bd{0})  = \hat{\bd{I}}{}_{\Ms}^{\hho}(\bd{v}_0),$ and  $ \bbm{s}_{\Ts}(0) = \bbm{I}_{\Ts}^{\hdg}(\bbm{s}_0)$, respectively. The following holds for all $t \in (0,T_{\rm{f}}]$:
\ifHAL
\begin{align}
&\|\bd{m}-\bmt\|^2_{C^{0}(\ol{J}_t;\bd{L}^2(\rho^\sc{f};\domain{f}))}+\|p-p_{\Tf}\|^2_{C^{0}(\ol{J}_t;L^2(\frac{1}{\kappa};\domain{f}))} +\|\bbm{s}-\bbm{s}_{\Ts}\|^2_{C^{0}(\ol{J}_t;\bbm{L}^2(\bbm{C}^{-1};\domain{s}))}\nonumber\\& \quad+\|\bd{v}-\bd{v}_{\Ts}\|^2_{C^{0}(\ol{J}_t;\bd{L}^2(\rho^\sc{s};\domain{s}))}\nonumber\\&\lesssim\|\bd{m}-\bd{I}_{\Tf}^{\hdg} (\bd{m})\|^2_{C^{0}(\ol{J}_t;\bd{L}^2(\rho^\sc{f};\domain{f}))}+\|\partial_t \bd{m}-\bd{I}_{\Tf}^{\hdg}(\partial_t \bd{m})\|^2_{L^1(J_t;\bd{L}^2(\rho^\sF;\domain{f}))}+\|p-\Pi^k_{\Tf}(p)\|_{C^{0}(\ol{J}_t;L^2(\frac{1}{\kappa};\domain{f}))}^2\nonumber\\&\quad +\|\bbm{s}-\bbm{I}_{\Ts}^{\hdg} (\bbm{s})\|^2_{C^{0}(\ol{J}_t;\bbm{L}^2(\bbm{C}^{-1};\domain{s}))}+\|\partial_t \bbm{s}-\bbm{I}_{\Ts}^{\hdg}(\partial_t \bbm{s})\|^2_{L^1(J_t;\bbm{L}^2(\bbm{C}^{-1};\domain{s}))}+\|\bd{v}-\bd{\Pi}_{\Ts}^{k}(\bd{v})\|^2_{C^{0}(\ol{J}_t;\bd{L}^2(\rho^\sc{s};\domain{s}))}\nonumber\\&\quad+\|(\bd{m},p)\|^2_{L^2(J_t;\#,\sF)}+\|(\bbm{s},\bd{v})\|^2_{L^2(J_t;\#,\sS)}.\label{eq:err-est}
\end{align}
\else
\begin{align}
&\|\bd{m}-\bmt\|^2_{C^{0}(\ol{J}_t;\bd{L}^2(\rho^\sc{f};\domain{f}))}+\|p-p_{\Tf}\|^2_{C^{0}(\ol{J}_t;L^2(\frac{1}{\kappa};\domain{f}))} +\|\bbm{s}-\bbm{s}_{\Ts}\|^2_{C^{0}(\ol{J}_t;\bbm{L}^2(\bbm{C}^{-1};\domain{s}))} +\|\bd{v}-\bd{v}_{\Ts}\|^2_{C^{0}(\ol{J}_t;\bd{L}^2(\rho^\sc{s};\domain{s}))}\nonumber\\&\lesssim\|\bd{m}-\bd{I}_{\Tf}^{\hdg} (\bd{m})\|^2_{C^{0}(\ol{J}_t;\bd{L}^2(\rho^\sc{f};\domain{f}))}+\|\partial_t \bd{m}-\bd{I}_{\Tf}^{\hdg}(\partial_t \bd{m})\|^2_{L^1(J_t;\bd{L}^2(\rho^\sF;\domain{f}))}+\|p-\Pi^k_{\Tf}(p)\|_{C^{0}(\ol{J}_t;L^2(\frac{1}{\kappa};\domain{f}))}^2\nonumber\\&\quad +\|\bbm{s}-\bbm{I}_{\Ts}^{\hdg} (\bbm{s})\|^2_{C^{0}(\ol{J}_t;\bbm{L}^2(\bbm{C}^{-1};\domain{s}))}+\|\partial_t \bbm{s}-\bbm{I}_{\Ts}^{\hdg}(\partial_t \bbm{s})\|^2_{L^1(J_t;\bbm{L}^2(\bbm{C}^{-1};\domain{s}))}+\|\bd{v}-\bd{\Pi}_{\Ts}^{k}(\bd{v})\|^2_{C^{0}(\ol{J}_t;\bd{L}^2(\rho^\sc{s};\domain{s}))}\nonumber\\&\quad+\|(\bd{m},p)\|^2_{L^2(J_t;\#,\sF)}+\|(\bbm{s},\bd{v})\|^2_{L^2(J_t;\#,\sS)}.\label{eq:err-est}
\end{align}
\fi
\end{theorem}
\begin{proof}
(1) Error equations. For all $t\in \ol{J}$, we define the discrete errors as follows:
\begin{subequations}
\begin{alignat}{3}
\etaf(t)&:= \bmt(t)-\bd{I}_{\Tf}^{\hdg} (\bd{m}(t)),\quad &\etas(t)&:= \bbm{s}_{\Ts}(t)-\bbm{I}_{\Ts}^{\hdg}(\bbm{s}(t)), \\
\ehf(t)&:=\hat{p}_{\Mf}(t)-\hat{I}_{\Mf}^{\hho}(p(t)),\quad&\ehs(t)&:= \hat{\bd{v}}_{\Ms}(t)- \hat{\bd{I}}{}_{\Ms}^{\hho}(\bd{v}(t)).
\end{alignat}
\end{subequations}
The first equation \eqref{acoustic1} in the discrete problem leads, for all $\br_{\Tf}\in \Usigmaf$ and all $t \in \ol{J}$, to	
\ifHAL
\begin{align}
&(\partial_t\etaf(t),\br_{\Tf})_{\bd{L}^2(\rho^\sF;\domain{f})}\!-\!(\bG_{\Tf}(\ehf(t)),\br_{\Tf})_{\bd{L}^2(\domain{f})}\nonumber\\&\quad=-(\bd{I}_{\Tf}^{\hdg}(\partial_t \bd{m}(t)),\br_{\Tf})_{\bd{L}^2(\rho^\sF;\domain{f})}\!+\! (\bG_{\Tf}(\hat{I}_{\Mf}^{\hho}(p(t))),\br_{\Tf})_{\bd{L}^2(\domain{f})}\nonumber\\&\quad=(\partial_t \bd{m}(t)-\bd{I}_{\Tf}^{\hdg}(\partial_t \bd{m}(t)),\br_{\Tf})_{\bd{L}^2(\rho^\sF;\domain{f})},\label{id1}
\end{align}
\else
\begin{align}
(\partial_t\etaf(t),\br_{\Tf})_{\bd{L}^2(\rho^\sF;\domain{f})}\!-\!(\bG_{\Tf}(\ehf(t)),\br_{\Tf})_{\bd{L}^2(\domain{f})}&\!=\!-(\bd{I}_{\Tf}^{\hdg}(\partial_t \bd{m}(t)),\br_{\Tf})_{\bd{L}^2(\rho^\sF;\domain{f})}\!+\! (\bG_{\Tf}(\hat{I}_{\Mf}^{\hho}(p(t))),\br_{\Tf})_{\bd{L}^2(\domain{f})}\nonumber\\&=(\partial_t \bd{m}(t)-\bd{I}_{\Tf}^{\hdg}(\partial_t \bd{m}(t)),\br_{\Tf})_{\bd{L}^2(\rho^\sF;\domain{f})},\label{id1}
\end{align}
\fi
where we used that $\bd{I}_{\Tf}^{\hdg}(\partial_t\bullet) = \partial_t\bd{I}_{\Tf}^{\hdg}(\bullet)$, $\bG_{\Tf}(\hat{I}_h^{\hho}(p(t)))=\bd{\Pi}_{\Tf}^k(\nabla p(t))=\bd{\Pi}_{\Tf}^k(\rho^\sF\partial_t\bd{m}(t))$ (the first equality follows from the definition of $\bG_{\Tf}$ and the second from \eqref{continuous_acoustic_eq_1}), and the $\bd{L}^2$-orthogonality of $\bd{\Pi}^k_{\Tf}$. 
\ifHAL
Analogously, using \eqref{elastic1}, $\bbm{I}_{\Ts}^{\hdg}(\partial_t\bullet) \!=\! \partial_t\bbm{I}_{\Ts}^{\hdg}(\bullet)$ and the identity  $\bbm{g}_{\Ts}^{\rm{sym}}(\hat{\bd{I}}{}_h^{\hho}(\bd{v}(t))\!=\!{\bbPi}_{\Ts}^k(\nabla_{\rm{sym}} \bd{v}(t))$\\$=\bbPi_{\Ts}^k(\bbm{C}^{-1}\partial_t\bbm{s}(t))$,
\else
Analogously, using \eqref{elastic1}, $\bbm{I}_{\Ts}^{\hdg}(\partial_t\bullet) = \partial_t\bbm{I}_{\Ts}^{\hdg}(\bullet)$ and the identity 
$\bbm{g}_{\Ts}^{\rm{sym}}(\hat{\bd{I}}{}_h^{\hho}(\bd{v}(t))={\bbPi}_{\Ts}^k(\nabla_{\rm{sym}} \bd{v}(t))=\bbPi_{\Ts}^k(\bbm{C}^{-1}\partial_t\bbm{s}(t))$,
\fi
we have, for all $\bbm{b}_{\Ts}\in \Usigmas$ and all $t\in \ol{J}$,
\begin{align}
(\partial_t\etas(t),\bbm{b}_{\Ts})_{\bbm{L}^2(\bbm{C}^{-1};\domain{s})} - (\bbm{g}_{\Ts}^{\rm{sym}}(\ehs(t)),\bbm{b}_{\Ts})_{\bbm{L}^2(\domain{s})} = (\partial_t \bbm{s}(t)-\bbm{I}_{\Ts}^{\hdg}(\partial_t \bbm{s}(t)),\bbm{b}_{\Ts})_{\bbm{L}^2(\bbm{C}^{-1};\domain{s})}.\label{id2}
\end{align}
The second equation \eqref{acoustic2} in the discrete problem  and \eqref{continuous_acoustic_eq_2} in the continuous problem  show that,  for all $\hat{q}_{\Mf}\in\UhatfD$ and all $t\in \ol{J}$,
\begin{align}
&(\partial_te_{\Tf}(t),q_{\Tf})_{L^2(\frac{1}{\kappa};\domain{f})}+(\etaf(t),\bG_{\Tf}(\hat{q}_{\Mf}))_{\bd{L}^2(\domain{f})}+s_{\Mf}(\ehf(t),\hat{q}_{\Mf})+(\be_{\Fs}(t){\cdot}\bd{n}_\G,q_{\Ff})_{L^2(\Gamma)}\nonumber\\
&\quad = (\tfrac{1}{\kappa}\partial_tp(t)-\nabla{\cdot} \bd{m}(t),q_{\Tf})_{L^2(\domain{f})}-(\partial_t\Pi_{\Tf}^{k'}(p(t)),q_{\Tf})_{L^2(\frac{1}{\kappa};\domain{f})}-(\bd{I}_{\Tf}^{\hdg}(\bd{m}(t)),\bG_{\Tf}(\hat{q}_{\Mf}))_{\bd{L}^2(\domain{f})}\nonumber\\&\qquad-s_{\Mf}(\hat{I}_{\Mf}^{\hho}(p(t)),\hat{q}_{\Mf})-(\bd{\Pi}^k_{\Fs}(\bd{v}(t)){\cdot}\bd{n}_\G,q_{\Ff})_{L^2(\Gamma)}\nonumber\\
&\quad=\sum_{T\in\Tf}\Big\{(\partial_t(p(t)-\Pi_T^{k'}(p(t))),q_T)_{L^2(\frac{1}{\kappa};\domain{f})}+(\bd{m}(t)-\bd{I}_T^{\hdg}(\bd{m}(t)),\nabla q_T)_{\bd{L}^2(T)}\nonumber\\
&\qquad -(\bd{m}(t){\cdot}\bd{n}_T,q_T)_{L^2(\partial T)}+(\bd{I}_T^{\hdg}(\bd{m}(t)){\cdot}\bd{n}_T,q_T-q_{\partial T})_{L^2(\dt)}\Big\}-s_{\Mf}(\hat{I}_{\Mf}^{\hho}(p(t)),\hat{q}_{\Mf})\nonumber\\&\qquad-(\bd{\Pi}^k_{\Fs}(\bd{v}(t)){\cdot}\bd{n}_\G,q_{\Ff})_{L^2(\Gamma)}\nonumber\\
& \quad=\sum_{T\in\Tf}\Big\{-(\Pi^k_{\dt}(\bd{m}(t){\cdot}\bd{n}_T),q_T)_{L^2(\partial T)}+(\bd{I}_T^{\hdg}(\bd{m}(t)){\cdot}\bd{n}_T,q_T-q_{\partial T})_{L^2(\dt)}\Big\}\nonumber\\&\qquad-s_{\Mf}(\hat{I}_{\Mf}^{\hho}(p(t)),\hat{q}_{\Mf})-(\Pi^k_{\F^\G}(\bd{v}(t){\cdot}\bd{n}_\G),q_{\Ff})_{L^2(\Gamma)},\label{eq:18}
\end{align}
with an integration by parts  and the definition of $\bG_{\Tf}$ from \eqref{def:G} in the second step (since $\bd{I}_{\Tf}^{\hdg}(\bd{m}(t))\in \bd{P}^k(\Tf))$, and the $L^2$-orthogonality of $\Pi_T^{k'}$ and  the definition \eqref{hdg.1b} of $\bd{I}_T^{\hdg}$  in the last step. Since $\bd{\Pi}^k_{\dt}(\bd{m}(t))$ and $q_{\dt}$ are single-valued and since $q_F =0$ for all $F\in \cF_h^{\sd \sc{f}}$, we have    
\begin{align*}
\sum_{T\in\Tf}(\Pi^k_{\dt}(\bd{m}(t){\cdot}\bd{n}_T),q_{\dt})_{L^2(\partial T)}=-(\Pi^k_{\F^\G}(\bd{m}(t){\cdot}\bd{n}_\G),q_{\Ff})_{L^2(\Gamma)}=-(\Pi^k_{\F^\G}(\bd{v}(t){\cdot}\bd{n}_\G),q_{\Ff})_{L^2(\Gamma)},
\end{align*}
where we also used that $\bd{n}_T|_{{\dt}\cap\Gamma} =  -\bd{n}_\G$, and the coupling condition \eqref{coupling1}. This in \eqref{eq:18} and the observation that $(\bd{\Pi}^k_{\dt}(\bd{m}(t))-\bd{I}_T^{\hdg}(\bd{m}(t))|_{\dt}){\cdot}\bd{n}_T \in P^k(\cF_{\dt})$ result in
\begin{align}
&(\partial_te_{\Tf}(t),q_{\Tf})_{L^2(\frac{1}{\kappa};\domain{f})}+(\etaf(t),\bG_{\Tf}(\hat{q}_{\Mf}))_{\bd{L}^2(\domain{f})}+s_{\Mf}(\ehf(t),\hat{q}_{\Mf})+(\be_{\Fs}(t){\cdot}\bd{n}_\G,q_{\Ff})_{L^2(\Gamma)}\nonumber\\
&\quad =\sum_{T\in\Tf}((\bd{\Pi}^k_{\dt}(\bd{m}(t))-\bd{I}_T^{\hdg}(\bd{m}(t))){\cdot}\bd{n}_T,q_{\dt}-q_T)_{L^2(\partial T)}-s_{\Mf}(\hat{I}_{\Mf}^{\hho}(p(t)),\hat{q}_{\Mf})\nonumber\\&\quad =\sum_{T\in\Tf}((\bd{m}(t)-\bd{I}_T^{\hdg}(\bd{m}(t))){\cdot}\bd{n}_T,\Pi^k_{\dt}(q_{\partial T}-q_T))_{L^2(\partial T)}-s_{\Mf}(\hat{I}_{\Mf}^{\hho}(p(t)),\hat{q}_{\Mf}).\label{erroreqn1}
\end{align}
Similarly, for the elastic problem, the second equation \eqref{elastic2} in the discrete problem  and \eqref{continuous_elastic_2} in the continuous problem  show that,  for all $\hat{\bw}_{\Ms}\in\UhatsD$ and all $t\in \ol{J}$,
\begin{align}
&(\partial_t\be_{\Ts}(t),\bw_{\Ts})_{\bd{L}^2(\rho^\sc{s};\domain{s})}+(\etas(t),\bbm{g}_{\Ts}^{\rm{sym}}(\hat{\bw}_{\Ms}))_{\bbm{L}^2(\domain{s})}+s_{\Ms}(\ehs(t),\hat{\bw}_{\Ms})-(e_{\Ff}(t)\bd{n}_\G,\bw_{\Fs})_{\bd{L}^2(\Gamma)}\nonumber\\
&\quad = (\rho^\sc{s}\partial_t\bd{v}(t) - \nabla{\cdot} \bbm{s}(t),\bw_{\Ts})_{\bd{L}^2(\domain{s})}-(\partial_t\bd{\Pi}_{\Ts}^{k'}(\bd{v}(t)),\bw_{\Ts})_{\bd{L}^2(\rho^\sc{s};\domain{s})}-(\bbm{I}_{\Ts}^{\hdg}(\bbm{s}(t)),\bbm{g}_{\Ts}^{\rm{sym}}(\hat{\bw}_{\Ms}))_{\bbm{L}^2(\domain{s})}\nonumber\\&\qquad-s_{\Ms}(\hat{\bd{I}}{}_{\Ms}^{\hho}(\bd{v}(t)),\hat{\bw}_{\Ms})+(\Pi^k_{\Fs}(p(t))\bd{n}_\G,\bw_{\Fs})_{\bd{L}^2(\Gamma)}\nonumber\\
&\quad=\sum_{T\in\Ts}\Big\{(\partial_t(\bd{v}(t)-\bd{\Pi}_{T}^{k'}(\bd{v}(t))),\bw_{\Ts})_{L^2(\rho^\sc{s};\domain{s})}+(\bbm{s}(t)-\bbm{I}_T^{\hdg}(\bbm{s}(t)),\nabla_{\rm{sym}} \bw_T)_{\bbm{L}^2(T)}\nonumber\\
&\qquad -(\bbm{s}(t){\cdot}\bd{n}_T,\bw_T)_{\bd{L}^2(\partial T)}+(\bbm{I}_T^{\hdg}(\bbm{s}(t)){\cdot}\bd{n}_T,\bw_T-\bw_{\partial T})_{\bd{L}^2(\dt)}\Big\}\nonumber\\&\qquad-s_{\Ms}(\hat{\bd{I}}{}_{\Ms}^{\hho}(\bd{v}(t)),\hat{\bw}_{\Ms})+(\Pi^k_{\Ff}(p(t))\bd{n}_\G,\bw_{\Fs})_{L^2(\Gamma)}\nonumber\\
& \quad=\sum_{T\in\Tf}\Big\{-(\bd{\Pi}^k_{\dt}(\bbm{s}(t){\cdot}\bd{n}_T),\bw_T)_{\bd{L}^2(\partial T)}+(\bbm{I}_T^{\hdg}(\bbm{s}(t)){\cdot}\bd{n}_T,\bw_T-\bw_{\partial T})_{\bd{L}^2(\dt)}\Big\}\nonumber\\&\qquad-s_{\Ms}(\hat{\bd{I}}{}_{\Ms}^{\hho}(\bd{v}(t)),\hat{\bw}_{\Ms})+(\bd{\Pi}^k_{\F^\G}(p(t)\bd{n}_\G),\bw_{\Fs})_{\bd{L}^2(\Gamma)},\label{eqs:18}
\end{align}
with an integration by parts  and the definition of $\bbm{g}_{\Ts}^{\rm{sym}}$ from \eqref{g-elastic} in the second step, and the $L^2$-orthogonality of $\bd{\Pi}_T^{k'}$ and  the definition \eqref{hdgs.1b} of $\bbm{I}_T^{\hdg}$  in the last step.
Since $\bbPi^k_{\dt}(\bbm{s}(t))$ and $\bw_{\dt}$ are single-valued and since $\bw_F =\bd{0}$ for all $F\in \cF_h^{\sd \sc{s}}$, we have    
\begin{align*}
\sum_{T\in\Ts}(\bd{\Pi}^k_{\dt}(\bbm{s}(t){\cdot}\bd{n}_T),\bw_{\dt})_{\bd{L}^2(\partial T)}=(\bd{\Pi}^k_{\cF^\G}(\bbm{s}(t){\cdot}\bd{n}_\G),\bw_{\Fs})_{\bd{L}^2(\Gamma)}=(\bd{\Pi}^k_{\F^\G}(p(t)\bd{n}_\G),\bw_{\Fs})_{\bd{L}^2(\Gamma)},
\end{align*}
where we also used that $\bd{n}_T|_{{\dt}\cap\Gamma} =  \bd{n}_\G$, and the coupling condition \eqref{coupling2}.
This in \eqref{eqs:18} and the observation that $(\bbPi^k_{\dt}(\bbm{s}(t)-\bbPi_T^k(\bbm{s}(t))|_{\dt}){\cdot}\bd{n}_T \in \bd{P}^k(\cal{F}_{\dt})$ result in
\begin{align}
&(\partial_t\be_{\Ts}(t),\bw_{\Ts})_{\bd{L}^2(\rho^\sc{s};\domain{s})}+(\etas(t),\bbm{g}_{\Ts}^{\rm{sym}}(\hat{\bw}_{\Ms}))_{\bbm{L}^2(\domain{s})}+s_{\Ms}(\ehs(t),\hat{\bw}_{\Ms})-(e_{\Ff}(t)\bd{n}_\G,\bw_{\Fs})_{\bd{L}^2(\Gamma)}\nonumber\\
&\quad =\sum_{T\in\Ts}((\bbPi^k_{\dt}(\bbm{s}(t))-\bbm{I}_T^{\hdg}(\bbm{s}(t))){\cdot}\bd{n}_T,\bw_{\dt}-\bw_T)_{\bd{L}^2(\partial T)}-s_{\Ms}(\hat{\bd{I}}{}_{\Ms}^{\hho}(\bd{v}(t)),\hat{\bw}_{\Ms})\nonumber\\&\quad =\sum_{T\in\Ts}((\bbm{s}(t)-\bbm{I}_T^{\hdg}(\bbm{s}(t))){\cdot}\bd{n}_T,\bd{I}^k_{\dt}(\bw_{\partial T}-\bw_T))_{\bd{L}^2(\partial T)}-s_{\Ms}(\hat{\bd{I}}{}_{\Ms}^{\hho}(\bd{v}(t)),\hat{\bw}_{\Ms}).\label{erroreq2}
\end{align}
 The combination of \eqref{erroreqn1} and \eqref{erroreq2} leads, for all $\hat{q}_{\Mf}\in\UhatfD$ and all $\hat{\bw}_{\Ms}\in\UhatsD$, to
\begin{align}
&(\partial_te_{\Tf}(t),q_{\Tf})_{L^2(\frac{1}{\kappa};\domain{f})}+(\etaf(t),\bG_{\Tf}(\hat{q}_{\Mf}))_{L^2(\domain{f})}+s_{\Mf}(\hat{e}_{\Mf}(t),\hat{q}_{\Mf})+(\be_{\Fs}(t){\cdot}\bd{n}_\G,q_{\Ff})_{L^2(\Gamma)}\nonumber\\&\quad+(\partial_t\be_{\Ts}(t),\bw_{\Ts})_{\bd{L}^2(\rho^\sc{s};\domain{s})}+(\etas(t),\bbm{g}_{\Ts}^{\rm{sym}}(\hat{\bw}_{\Ms}))_{\bbm{L}^2(\domain{s})}+s_{\Ms}(\ehs(t),\hat{\bw}_{\Ms})-(e_{\Ff}(t)\bd{n}_\G,\bw_{\Fs})_{\bd{L}^2(\Gamma)}\nonumber\\&=:\psi_{\Mf}((\bd{m}(t),p(t));\hat{q}_{\Mf})+\psi_{\Ms}((\bbm{s}(t),\bd{v}(t));\hat{\bw}_{\Ms}),\label{erroreq:combined}
\end{align}
where the linear functionals $\psi_{\Mf}((\bd{m}(t),p(t));\cdot)\in(\UhatfD)'$ and $\psi_{\Ms}((\bbm{s}(t),\bd{v}(t));\cdot)\in(\UhatsD)'$ denote the consistency errors such that
\begin{align*}
\psi_{\Mf}((\bd{m}(t),p(t));\hat{q}_{\Mf})&\!:=\!\sum_{T\in\Tf}((\bd{m}(t)-\bd{I}_T^{\hdg}(\bd{m}(t))){\cdot}\bd{n}_T,\Pi^k_{\dt}(q_{\partial T}-q_T))_{L^2(\partial T)}-s_{\Mf}(\hat{I}_{\Mf}^{\hho}(p(t)),\hat{q}_{\Mf}),\\
\psi_{\Ms}((\bbm{s}(t),\bd{v}(t));\hat{\bw}_{\Ms})&\!:=\! \sum_{T\in\Ts}((\bbm{s}(t)-\bbm{I}_T^{\hdg}(\bbm{s}(t))){\cdot}\bd{n}_T,\bd{\Pi}^k_{\dt}(\bw_{\partial T}-\bw_T))_{\bd{L}^2(\partial T)}-s_{\Ms}(\hat{\bd{I}}{}_{\Ms}^{\hho}(\bd{v}(t)),\hat{\bw}_{\Ms}),
\end{align*}
for all $\hat{q}_{\Mf}\in \UhatfD$ and all $\hat{\bw}_{\Ms}\in\UhatsD$.\\

\medskip
\noindent (2) Stability. Choosing $\hat{q}_{\Mf}:=\ehf(t)$ and $\hat{\bw}_{\Ms}:=  \ehs(t)$  in the error equation \eqref{erroreq:combined} for all $t\in \ol{J}$, and using \eqref{id1} and \eqref{id2} with $\br_{\Tf}:=\etafT(t)$ and $\bbm{b}_{\Ts}:=\etasT(t)$, we obtain
\begin{align}
&\frac{1}{2}\frac{d}{dt}\left\{\|e_{\Tf}(t)\|^2_{L^2(\frac{1}{\kappa};\domain{f})}+\|\be_{\Ts}(t)\|^2_{\bd{L}^2(\rho^\sc{s};\domain{s})}+\|\etaf(t)\|^2_{\bd{L}^2(\rho^\sc{f};\domain{f})}+\|\etas(t)\|^2_{\bbm{L}^2(\bbm{C}^{-1};\domain{s})}\right\}\nonumber\\
&\;+s_{\Mf}(\ehf(t),\ehf(t))+s_{\Ms}(\ehs(t),\ehs(t))=(\partial_t \bd{m}(t)-\bd{I}_{\Tf}^{\hdg}(\partial_t \bd{m}(t)),\etaf(t))_{\bd{L}^2(\rho^\sF;\domain{f})}\nonumber\\
&\quad+(\partial_t \bbm{s}(t)-\bbm{I}_{\Ts}^{\hdg}(\partial_t \bbm{s}(t)),\etas(t))_{\bbm{L}^2(\bbm{C}^{-1};\domain{s})}+\psi_{\Mf}((\bd{m}(t),p(t));\ehf(t))+\psi_{\Ms}((\bbm{s}(t),\bd{v}(t));\ehs(t)).\label{2.13}
\end{align}
This is the critical step where the errors at the interface cancel, as do the interface terms in the energy balance. Integrating from $0$ to $t$, and noticing that $e_{\Tf}(0)=0$, $ \etaf(0) = \bd{0}=\be_{\Ts}(0)$, and  $\etas(0) = \rule{0.3pt}{0.66em}\hspace{-0.075cm}0$ owing to the initial conditions lead to
\begin{align*}
\frac{1}{2} &\Big\{\|e_{\Tf}(t)\|^2_{L^2(\frac{1}{\kappa};\domain{f})}+\|\be_{\Ts}(t)\|^2_{\bd{L}^2(\rho^\sc{s};\domain{s})}+\|\etaf(t)\|^2_{\bd{L}^2(\rho^\sc{f};\domain{f})}+\|\etas(t)\|^2_{\bbm{L}^2(\bbm{C}^{-1};\domain{s})}\Big\} \\
&+\|\ehf\|_{L^2(J_t;\rm{S}^{\sF})}^2+\|\ehs\|_{L^2(J_t;\rm{S}^{\sS})}^2  \\
={}&  \int_{J_t} \Big\{(\partial_t \bd{m}(\tau)-\bd{I}_{\Tf}^{\hdg}(\partial_t \bd{m}(\tau)),\etaf(\tau))_{\bd{L}^2(\rho^\sF;\domain{f})} 
+(\partial_t \bbm{s}(\tau)-\bbm{I}_{\Ts}^{\hdg}(\partial_t \bbm{s}(\tau)),\etas(\tau))_{\bbm{L}^2(\bbm{C}^{-1};\domain{s})}\\
&+\psi_{\Mf}((\bd{m}(\tau),p(\tau));\ehf(\tau))+\psi_{\Ms}((\bbm{s}(\tau),\bd{v}(\tau));\ehs(\tau))\Big\}\,\rm{d\tau}.
\end{align*}
H\"older's inequality in time for the first two terms on the right hand-side, the Cauchy-Schwarz inequality in time for the last two terms, and  Young's inequality  imply that
\begin{align*}
&\frac{1}{2}\Big\{\|e_{\Tf}(t)\|^2_{L^2(\frac{1}{\kappa};\domain{f})}+\|\be_{\Ts}(t)\|^2_{\bd{L}^2(\rho^\sc{s};\domain{s})}+\|\etaf(t)\|^2_{\bd{L}^2(\rho^\sc{f};\domain{f})}+\|\etas(t)\|^2_{\bbm{L}^2(\bbm{C}^{-1};\domain{s})}\Big\}\\
&\quad+\frac{3}{4}\Big\{\|\ehf\|_{L^2(J_t;\rm{S}^{\sF})}^2+\|\ehs\|_{L^2(J_t;\rm{S}^{\sS})}^2\Big\}\\&\leq \|\partial_t \bd{m}-\bd{I}_{\Tf}^{\hdg}(\partial_t \bd{m})\|^2_{L^1(J_t;\bd{L}^2(\rho^\sF;\domain{f}))}+\|\partial_t \bbm{s}-\bbm{I}_{\Ts}^{\hdg}(\partial_t \bbm{s})\|^2_{L^1(J_t;\bbm{L}^2(\bbm{C}^{-1};\domain{s}))}+\frac{1}{4}\Big\{\|\etaf\|^2_{C^0(\ol{J}_t;\bd{L}^2(\rho^\sF;\domain{f}))}\\
&\quad+\|\etas\|^2_{C^0(\ol{J}_t;\bbm{L}^2(\bbm{C}^{-1};\domain{s}))}\Big\}+\|\psi_{\Mf}((\bd{m},p);\cdot)\|^2_{L^2(J_t;(\text{S}^\sF)')}+\|\psi_{\Ms}((\bbm{s},\bd{v});\cdot)\|^2_{L^2(J_t;(\text{S}^{\sS})')}.
\end{align*}

\medskip
\noindent (3) Bound on consistency errors. Rewriting the definition of $\psi_{\Mf}$ leads to
\begin{align*}
\psi_{\Mf}((\bd{m}(t),p(t));\hat{q}_{\Mf})&=\sum_{T\in\Tf}((\tau^\sc{f}_{T})^{-\frac{1}{2}}(\bd{m}(t)-\bd{I}_{T}^{\hdg}(\bd{m}(t))){\cdot}\bd{n}_T,(\tau^\sc{f}_{T})^{\frac{1}{2}}\Pi^k_{\dt}(q_{\partial T}-q_T))_{L^2(\partial T)}\\&\quad-s_{\Mf}(\hat{I}_{\Mf}^{\hho}(p(t)),\hat{q}_{\Mf}).
\end{align*}
The Cauchy-Schwarz inequality and the definition \eqref{stab-op-F} of $S_{\dt}$  imply that
\begin{align*}
&|\psi_{\Mf}((\bd{m}(t),p(t));\hat{q}_{\Mf})|\\&\quad\lesssim	\bigg\{\sum_{T\in\Tf}\Big(\tih_T^{\alpha}\|\bd{B}_{\dt}(\bd{m}(t)){\cdot}\bd{n}_T\|^2_{L^2(\frac{1}{\cred{\zeta}^\sc{f}};\dt)}+s_T^{\sF}(\hat{I}_T^{\hho}(p(t)),\hat{I}_T^{\hho}(p(t)))\Big)\bigg\}^{\frac{1}{2}}s_{\Mf}(\hat{q}_{\Mf},\hat{q}_{\Mf})^{\frac{1}{2}}.
\end{align*}
\cred{where $s_T^{\sc{f}}$ is the local contribution to $s_{\Mf}$ defined in~\eqref{stab-F}.} Using the $L^2$-stability of $\Pi^k_{\dt}$, a multiplicative trace inequality, and the Poincar\'e inequality on $T$, we infer that 
\begin{align*} 
s_T^{\sF}(\hat{I}_T^{\hho}(p(t)),\hat{I}_T^{\hho}(p(t)))&\leq \tau^\sc{f}_{T}\|p(t)-\Pi_T^{k'}(p(t))\|_{L^2(\dt)}^2\\
&\lesssim \tau^\sc{f}_{T}(h_T^{-1}\|p(t)-\Pi_T^{k'}(p(t))\|_{L^2(T)}^2+ h_T\|\nabla(p(t)-\Pi_T^{k'}(p(t)))\|_{L^2(T)}^2)\\
&\lesssim \tau^\sc{f}_{T}h_T\|\nabla(p(t)-\Pi_T^{k'}(p(t)))\|_{L^2(T)}^2 = h_T\tih^{-\alpha}_T\|\nabla B_T(p(t)))\|_{\bd{L}^2(\cred{\zeta}^\sc{f};T)}^2,
\end{align*}
where we also used the definition of $\tau^\sc{f}_{T}$ from \eqref{stab-param-F}. Hence, we obtain
$\|\psi_{\Mf}((\bd{m}(t),p(t));\cdot)\|_{(\text{S}^{\sF})'}\lesssim |(\bd{m}(t),p(t))|_{\#,\sF}, $ and similarly, 
$\|\psi_{\Ms}((\bbm{s}(t),\bd{v}(t));\cdot)\|_{(\text{S}^{\sS})'}\lesssim |(\bbm{s}(t),\bd{v}(t))|_{\#,\sS}.$ \\

\medskip
\noindent(4) Conclusion: Combining Step (2) and Step (3) gives
\begin{align*}
&\frac{1}{2}\Big\{\|e_{\Tf}(t)\|^2_{L^2(\frac{1}{\kappa};\domain{f})}+\|\be_{\Ts}(t)\|^2_{\bd{L}^2(\rho^\sc{s};\domain{s})}+\|\etaf(t)\|^2_{\bd{L}^2(\rho^\sc{f};\domain{f})}+\|\etas(t)\|^2_{\bbm{L}^2(\bbm{C}^{-1};\domain{s})}\Big\}\\
&\quad+\frac{3}{4}\Big\{\|\ehf\|_{L^2(J_t;\rm{S}^{\sF})}^2+\|\ehs\|_{L^2(J_t;\rm{S}^{\sS})}^2\Big\}\\&\lesssim \|\partial_t \bd{m}-\bd{I}_{\Tf}^{\hdg}(\partial_t \bd{m})\|^2_{L^1(J_t;\bd{L}^2(\rho^\sF;\domain{f}))}+\|\partial_t \bbm{s}-\bbm{I}_{\Ts}^{\hdg}(\partial_t \bbm{s})\|^2_{L^1(J_t;\bbm{L}^2(\bbm{C}^{-1};\domain{s}))}+\frac{1}{4}\Big\{\|\etaf\|^2_{C^0(\ol{J}_t;\bd{L}^2(\rho^\sF;\domain{f}))}\\&\quad+\|\etas\|^2_{C^0(\ol{J}_t;\bbm{L}^2(\bbm{C}^{-1};\domain{s}))}\Big\}+\|(\bd{m},p)\|^2_{L^2(J_t;\#,\sF)}+\|(\bbm{s},\bd{v})\|^2_{L^2(J_t;\#,\sS)}.
\end{align*}
Since the left-hand side evaluated at any $t\in J_t$ is bounded by the right-hand side, we have 
\begin{align*}
&\frac{1}{2}\Big\{\|e_{\Tf}\|^2_{C^0(\ol{J}_t;L^2(\frac{1}{\kappa};\domain{f}))}+\|\be_{\Ts}\|^2_{C^0(\ol{J}_t;\bd{L}^2(\rho^\sc{s};\domain{s}))}\Big\}+\frac{1}{4}\Big\{\|\etaf\|^2_{C^0(\ol{J}_t;\bd{L}^2(\rho^\sc{f};\domain{f}))}+\|\etas\|^2_{C^0(\ol{J}_t;\bbm{L}^2(\bbm{C}^{-1};\domain{s}))}\Big\}\\
&\quad\lesssim \|\partial_t \bd{m}-\bd{I}_{\Tf}^{\hdg}(\partial_t \bd{m})\|^2_{L^1(J_t;\bd{L}^2(\rho^\sF;\domain{f}))}+\|\partial_t \bbm{s}-\bbm{I}_{\Ts}^{\hdg}(\partial_t \bbm{s})\|^2_{L^1(J_t;\bbm{L}^2(\bbm{C}^{-1};\domain{s}))}+\|(\bd{m},p)\|^2_{L^2(J_t;\#,\sF)}\\&\quad+\|(\bbm{s},\bd{v})\|^2_{L^2(J_t;\#,\sS)},
\end{align*}
where we dropped the positive stabilization terms on the left-hand side for simplicity.
This and the triangle inequality  result in (\ref{thm:err-est}).
\end{proof}
\begin{remark}[Gronwall's lemma]
Our proof of Theorem~\ref{thm:err-est} avoids the need to invoke Gronwall's lemma. The key argument is the fact that the consistency errors are controlled by the stabilization semi-norms of the discrete errors. This is the main reason for choosing the H$^+$ interpolation operators from \cite{DS_2021} instead of the usual $L^2$-projections for dG variables as in \cite{BDES_2021}.
\end{remark}
\begin{remark}[Convergence rates]\label{rem-1}
 If there are $\ell_1\in\{1,\dots,k+1\}$  and $\ell_2\in\{1,\dots,k^\prime+1\}$ such that $\bd{m} \in C^1(\overline{J};\bd{H}^{\ell_1}(\domain{f})),$ $ p\in C^0(\ol{J};H^{\ell_2}(\domain{f})) $ and $\bbm{s}\in C^1(\overline{J};\bbm{H}^{\ell_1}(\domain{s})), \bd{v}\in C^0(\ol{J}; \boldsymbol{H}^{\ell_2}(\domain{s}))$, we have
\begin{align}
&\|\bd{m}-\bmt\|_{C^{0}(\ol{J};\bd{L}^2(\rho^\sc{f};\domain{f}))}+\|p-p_{\Tf}\|_{C^{0}(\ol{J};L^2(\frac{1}{\kappa};\domain{f}))}+\|\bbm{s}-\bbm{s}_{\Ts}\|_{C^{0}(\ol{J};\bbm{L}^2(\bbm{C}^{-1};\domain{s}))} \nonumber\\&\quad+\|\bd{v}-\bd{v}_{\Ts}\|_{C^{0}(\ol{J};\bd{L}^2(\rho^\sc{s};\domain{s}))}\lesssim \mathcal{O}(\tih^{\frac{\alpha}{2}} h^{\ell_1-\frac{1}{2}}+\tih^{-\frac{\alpha}{2}} h^{\ell_2-\frac{1}{2}}).\label{conv-rates}
\end{align}
 In the case where $\ell_1=k+1$ and $\ell_2=k^\prime+1$, this gives  $\mathcal{O}(h^{k+\frac{1}{2}})$ for $\cal{O}(1)$-stabilization (i.e., $\alpha=0$ in \eqref{stab-param-F} and \eqref{stab-param-S}), and $\mathcal{O}(h^{k+1}+h^{k^\prime})$ for $\cal{O}(\frac{1}{h})$-stabilization (i.e., $\alpha=1$ in \eqref{stab-param-F} and \eqref{stab-param-S}), that is, $\mathcal{O}(h^{k})$ for equal-order and $\mathcal{O}(h^{k+1})$ for mixed-order. Finally, notice from Step (4) of the proof of Theorem~\ref{thm:err-est} that we can also bound the stabilization semi-norms of the discrete errors as
\begin{align}
\|\hat{p}_{\Mf}\|_{L^2(J;\text{S}^\sF)}+\|\hat{\bd{v}}_{\Ms}\|_{L^2(J;\text{S}^\sS)}\lesssim \mathcal{O}(\tih^{\frac{\alpha}{2}} h^{\ell_1-\frac{1}{2}}+\tih^{-\frac{\alpha}{2}} h^{\ell_2-\frac{1}{2}}).
\end{align}
Moreover, as shown in \cite{EK_2024}, the above rates for $\alpha=0$ can be improved to $\mathcal{O}(h^{k+1})$ on simplices. \cred{A discussion of the above regularity assumption on the exact solution for smooth data  and computational domain can be found in \cite[Section~8.5.1]{allaire2007}.} 
\end{remark}

\begin{remark}[Initial condition]
For the dG variables, the initial conditions $ \bd{m}_{\Tf}(\bd{0}) = \bd{I}_{\Tf}^{\hdg}(\bd{m}_0)$ and  $ \bbm{s}_{\Ts}(0) = \bbm{I}_{\Ts}^{\hdg}(\bbm{s}_0)$ are chosen for simplicity. Instead, we can set the initial conditions to  $ \bd{m}_{\Tf}(\bd{0}) = \bd{\Pi}_{\Tf}^{k}(\bd{m}_0)$ and  $ \bbm{s}_{\Ts}(0) = \bbPi_{\Ts}^{k}(\bbm{s}_0)$ using the usual $L^2$-projections, and this will only lead to the two extra contributions $\|\bd{I}_{\Tf}^{\hdg}(\bd{m}_0)-\bd{\Pi}_{\Tf}^{k}(\bd{m}_0)\|_{\bd{L}^2(\rho^\sc{f};\domain{f})}$ and $\|\bbm{I}_{\Ts}^{\hdg}(\bbm{s}_0)-\bbPi_{\Ts}^{k}(\bbm{s}_0)\|_{\bd{L}^2(\rho^\sc{s};\domain{s})}$ in the error estimate. The triangle inequalities $\|\bd{I}_{\Tf}^{\hdg}(\bd{m}_0)-\bd{\Pi}_{\Tf}^{k}(\bd{m}_0)\|_{\bd{L}^2(\rho^\sc{f};\domain{f})}\leq \|\bd{I}_{\Tf}^{\hdg}(\bd{m}_0)-\bd{m}_0\|_{\bd{L}^2(\rho^\sc{f};\domain{f})}+\|\bd{m}_0-\bd{\Pi}_{\Tf}^{k}(\bd{m}_0)\|_{\bd{L}^2(\rho^\sc{f};\domain{f})}$ and $\|\bbm{I}_{\Ts}^{\hdg}(\bbm{s}_0)-\bbPi_{\Ts}^{k}(\bbm{s}_0)\|_{\bd{L}^2(\rho^\sc{s};\domain{s})}\leq \|\bbm{I}_{\Ts}^{\hdg}(\bbm{s}_0)-\bbm{s}_0\|_{\bd{L}^2(\rho^\sc{s};\domain{s})}+\|\bbm{s}_0-\bbPi_{\Ts}^{k}(\bbm{s}_0)\|_{\bd{L}^2(\rho^\sc{s};\domain{s})}$ followed by interpolation estimates show that these additional terms converge optimally.
\end{remark}
\begin{remark}[Equal-order setting with HHO stabilization]\label{rem-2}
For equal-order setting with the high-order HHO stabilization, the stabilization operator in the acoustic domain  is defined as
\begin{equation}
S_{\partial T}(\hat{p}_T) := \Pi_{\partial T}^k (\delta_{\partial T}(\hat{p}_T) + 
((I-\Pi_T^k)R_T(0,\delta_{\partial T}(\hat{p}_T)))|_{\partial T})\qquad\forall \hat{p}_T\in\hat{P}_T^k,\label{HHO-stab-F}
\end{equation}
and in the  elastic domain as
\begin{equation}
\bd{S}_{\partial T} (\hat{\bd{v}}_T) := \bd{\Pi}_{\partial T}^k 
(\bd{\delta}_{\partial T} (\hat{\bd{v}}_T) + ((\bd{I}-\bd{\Pi}_T^k)
\bd{R}_T(\bd{0},\bd{\delta}_{\partial T} (\hat{\bd{v}}_T)))|_{\partial T})\qquad \forall\hat{\bd{v}}_T\in\hat{\bd{V}}{}_T^k,\label{HHO-stab-S}
\end{equation}
using suitable reconstruction operators $R_T:\hat{P}^k_T\to P^{k+1}(T)$ and $\bd{R}_T:\hat{\bd{V}}{}_T^k\to \bd{P}^{k+1}(T)$ respectively (see, e.g., \cite{DE_2015}). These choices combined with $\mathcal{O}(\frac{1}{h})$-stabilization improve the convergence rate to  $\mathcal{O}(h^{k+1})$ even on polyhedra in contrast to the stabilization operators \eqref{stab-op-F} and \eqref{stab-op-S}. However, as discussed in \cite[Sec.~7.1]{EK_2024}, this choice of stabilization is generally suitable only for implicit time-stepping schemes.
\end{remark}
\begin{remark}[Comparison with \cite{BDES_2021}]
The arguments in \cite{BDES_2021} for uncoupled acoustic or elastic waves are based on the $L^2$-orthogonal projection for the dual variable instead of the $H^+$ interpolation operator used herein. The analysis in \cite{BDES_2021} requires to bound the consistency errors $\psi_{\Mf}((\bd{m}(t),p(t));\cdot)$ and $\psi_{\Ms}((\bbm{s}(t),\bd{v}(t));\cdot)$ by means of the HHO norm, instead of the stabilization seminorm as done in the proof of Theorem~\ref{thm:err-est}. Both approaches (using either the $L^2$-projection or the H$^+$ interpolation operator for the dual variable) lead to the same convergence rates. However, in the context of explicit time-schemes, the analysis crucially hinges on bounding the consistency errors in terms of the stabilization seminorms only, as highlighted in \cite{EK_2024}. This is why we preferred to use the H$^+$ interpolation operator in the present work although its focus is on the time-continuous case.  
\end{remark}

\section{Algebraic realization of the semi-discrete problem}\label{sec::algebraic_realization}

This section details the algebraic realization of the space semi-discrete system (\ref{HHO1})-(\ref{HHO2}). For the acoustic wave equation, we define the dimensions of the following polynomial spaces:
\begin{equation}
N_{\Tf}^{k^\prime} := \dim(P^{k^\prime}(\Tf)), \quad N_{\Ff}^{k} := \dim(P^{k}(\Ff)), \quad M_{\Tf}^{k} := \dim(\bd{M}^{k}(\Tf)),
\end{equation} 
and denote the respective bases as
\begin{equation}
\left\{\varphi_i\right\}_{1 \leq i \leq N_{\Tf}^{k^\prime}}, \quad 
\left\{\psi_i \right\}_{1 \leq i \leq N_{\Ff}^{k}}, \quad 
\left\{\bd{\zeta}_k\right\}_{1 \leq k \leq M_{\Tf}^{k}}.    
\end{equation}
The basis $\left\{\bd{\zeta}_k\right\}_{1 \leq k \leq M_{\Tf}^{k}}$ is constructed as products of Cartesian basis vectors in $\bbm{R}^d$ with scalar-valued basis functions of $P^{k}(\Tf)$. Let $(\dofs{P}{T}{f}(t), \dofs{P}{F}{f}(t)) \in \bbm{R}^{N_{\Tf}^{k^\prime}} \times \bbm{R}^{N_{\Ff}^{k}}$ and $\dofs{M}{T}{f}(t) \in \bbm{R}^{M_{\Tf}^{k}}$ be the time-dependent component vectors of $(p_{\Tf}(t), p_{\Ff}(t)) \in \wh{P}^{k}_0(\Mf)$ and $\bd{m}_{\Tf}(t) \in \bd{M}^{k}(\Tf)$, respectively, in these bases. Let $\mass{\rho^{\sc{f}}}{f}$ and $\mass{\frac{1}{\kappa}}{f}$ denote the mass matrices associated with the inner products in $\bd{L}^2(\rho^\sc{f}; \Omega^\sc{f})$ and \(L^2(\frac{1}{\kappa}; \Omega^\sc{f})\), respectively, using the above bases. Let $\grad{}{T}{f} \in \bbm{R}^{M_{\Tf}^{k} \times N_{\Tf}^{k^\prime}}$ and $\cred{\gradd{}{T}{F}{f}} \in \bbm{R}^{M_{\Tf}^{k} \times N_{\Ff}^{k}}$ denote the two blocks of the gradient reconstruction matrix, so that
\begin{equation}
(\bd{g}_{\T}(\hat{p}_{\Mf}(t)),\bd{r}_{\Tf})_{{\bd{L}^2(\domain{f})}} = (\grad{}{T}{f}\dofs{\rm{P}}{T}{f} + \cred{\gradd{}{T}{F}{f}}\dofs{\rm{P}}{F}{f})^{\cred{\top}} {R}_{\T^\sc{f}},
\end{equation}
for all $\bd{r}_{\Tf} \in \bd{M}^{k}(\Tf)$ with components ${R}_{\T^\sc{f}} \in \bbm{R}^{M_{\Tf}^{k}}$. Finally, let \(\cred{\stabdiag{f}{T}}, \stab{f}{T}{F}, \cred{\stab{f}{T}{F}^{\top}}, \cred{\stabdiag{f}{F}{}}\) represent the four blocks of the matrix associated with the stabilization bilinear form $s_{\cal{M}^\sc{f}}$ defined in (\ref{stab-F}).

For the elastic wave equation, we define the dimensions of the following polynomial spaces:
\begin{equation}
L_{\Ts}^{k^\prime} := \dim(\bd{{V}}^{k^\prime}(\Ts)), \quad L_{\Fs}^{k} := \dim(\bd{{V}}^{k}(\Fs)), \quad H_{\Ts}^{k} := \dim(\bbm{S}_{\rm{sym}}^k(\Ts)),    
\end{equation}
and denote the respective bases as
\begin{equation}
\left\{\bd{\phi}_i\right\}_{1 \leq i \leq L_{\Ts}^{k^\prime}}, \quad 
\left\{\bd{\theta}_i \right\}_{1 \leq i \leq L_{\Fs}^{k}}, \quad 
\left\{\bbm{Y}_k\right\}_{1 \leq k \leq H_{\Ts}^{k}}.    
\end{equation}
The basis $\left\{\bbm{Y}_k\right\}_{1 \leq k \leq H_{\Ts}^{k}}$ is naturally built as tensor products of basis vectors in $\bbm{R}_{\text{sym}}^{d \times d}$ and scalar-valued basis functions in $P^k(\Ts)$. Let $(\dofs{V}{T}{s}(t), \dofs{V}{F}{s}(t)) \in \bbm{R}^{L_{\Ts}^{k^\prime}} \times \bbm{R}^{L_{\Fs}^{k}}$ and $\dofs{S}{T}{s} \in \bbm{R}^{H_{\Ts}^{k}}$ represent the time-dependent component vectors of $(\bd{v}_{\Ts}(t), \bd{v}_{\Fs}(t)) \in \bd{\wh{V}}^{k}_0(\Ms)$ and $\bbm{s}_{\Ts}(t) \in \bbm{S}_{\rm{sym}}^k(\Ts)$, respectively, in these bases. Let $\mass{\bbm{C}^{-1}}{s}$ and $\mass{\rho^{\sc{s}}}{s}$ denote the mass matrices for the inner products in $\bbm{L}^2(\bbm{C}^{-1};\Omega^\sc{s})$ and $\bd{L}^2(\rho^\sc{s}; \Omega^\sc{s})$, respectively, in these bases. Let $\strain{}{T^{\sc{s}}} \in \bbm{R}^{H_{\Ts}^{k} \times L_{\Ts}^{k^\prime}}$ and $\cred{\strain{}{T^{\sc{s}}F^{\sc{s}}}} \in \bbm{R}^{H_{\Ts}^{k} \times L_{\Fs}^{k}}$ denote the two blocks of the symmetric gradient reconstruction matrix, so that
\begin{equation}
(\bbm{g}^{\rm{sym}}_{\Ts}(\hat{\bd{v}}_{\Ms}(t)),\bbm{b}_{\Ts})_{{\bbm{L}^2(\domain{f})}} = (\strain{}{T^{\sc{s}}}\dofs{\rm{V}}{T}{s} + \cred{\strain{}{T^{\sc{s}}F^{\sc{s}}}}\dofs{\rm{V}}{F}{s})^{\cred{\top}} \rm{B}_{\T^\sc{s}}.
\end{equation}
for all $\bbm{b}_{\Ts} \in \bbm{S}_{\rm{sym}}^k(\Ts)$ with components $\rm{B}_{\T^\sc{s}} \in \bbm{R}^{H_{\Ts}^{k}}$. Finally, let $\cred{\stabdiag{s}{T}}, \stab{s}{T}{F}, \cred{\stab{s}{T}{F}^{\top}}, \cred{\stabdiag{s}{F}}$ denote the four blocks of the matrix associated with the stabilization bilinear form $s_{\cal{M}^\sc{s}}$ defined in (\ref{stab-S}).

Let $\cred{\coupling{}}$ be the matrix representing the coupling terms, so that 
\begin{equation}
(\bd{v}_{\Fs} \cdot \bd{n}_\Gamma, q_{\cal{F}^\sc{f}})_{{L^2(\Gamma)}} = \rm{Q}_{\F^\sc{f}}^{\cred{\top}} \cred{\coupling{}} \dofs{V}{F}{s}.
\end{equation}
for all $q_{\cal{F}^\sc{f}} \in P^{k}(\Ff)$ with components $\rm{Q}_{\F^\sc{f}} \in \bbm{R}^{N_{\Ff}^{k}}$ and all $\bd{v}_{\Fs} \in \bd{P}^k(\cal{\Fs})$ with components $\dofs{\rm{V}}{F}{s} \in \bbm{R}^{L_{\Fs}^{k}}$. Notice that $\cred{\coupling{}}$ is block-diagonal with a nonzero block only for all $F \in \cal{F}^{\Gamma}$. 

Altogether, the algebraic realization of (\ref{HHO1}) and (\ref{HHO2}) can be formulated in the following way:
\noindent For all $t \in \overline{J}$,
\begin{equation}
\resizebox{0.93\textwidth}{!}{%
$
\pmatset{4}{12pt} 
\pmatset{5}{10pt} 
\resizebox{!}{13.5ex}{ 
$
\begin{pmat}[{..|}]
\cred{\normalizecell{\mass{\rho^{\sc{f}}}{f}}} & 0      & 0 & 0      & 0      & 0 \cr
0      & \cred{\normalizecell{\mass{\frac{1}{\kappa}}{f}}} & 0 & 0      & 0      & 0 \cr
0      & 0      & 0 & 0      & 0      & 0 \cr \-
0      & 0      & 0 & \cred{\normalizecell{\mass{\bbm{C}^{-1}}{s}}} & 0      & 0 \cr
0      & 0      & 0 & 0      & \cred{\normalizecell{\mass{\rho^{\sc{s}}}{s}}} & 0 \cr
0      & 0      & 0 & 0      & 0      & 0 \cr
\end{pmat}
$
}
\dfrac{\rm{d}}{\rm{dt}}
\pmatset{4}{12pt} 
\pmatset{5}{10pt} 
\begin{pmat}[{}]
\normalizecell{\dofs{\rm{M}}{T}{f}} \cr
\normalizecell{\dofs{\rm{P}}{T}{f}} \cr
\normalizecell{\dofs{\rm{P}}{F}{f}} \cr \- 
\normalizecell{\dofs{\rm{S}}{T}{s}} \cr
\normalizecell{\dofs{\rm{V}}{T}{s}} \cr
\normalizecell{\dofs{\rm{V}}{F}{s}} \cr
\end{pmat} +
\pmatset{4}{12pt} 
\pmatset{5}{10pt} 
\begin{pmat}[{..|}]
0 & \normalizecell{\grad{}{T}{f}} & \cred{\normalizecell{\gradd{}{T}{F}{f}}} & 0 & 0 & 0 \cr
\normalizecell{-\grad{{\cred{\top}}}{T}{f}} & \cred{\normalizecell{\stabdiag{f}{T}}} & {\normalizecell{\stab{f}{T}{F}}} & 0 & 0 & 0 \cr
\cred{\normalizecell{-\gradd{{\top}}{T}{F}{f}}} & \cred{\normalizecell{\stab{f}{T}{F}^{\cred{\top}}}} & \cred{\normalizecell{\stabdiag{f}{F}}} & 0 & 0 & \normalizecell{\cred{\coupling{}}} \cr\-
0 & 0 & 0 & 0 & \normalizecell{\strain{}{T^{\sc{s}}}} & \cred{\normalizecell{\strain{}{T^{\sc{s}}F^{\sc{s}}}}} \cr
0 & 0 & 0 & {\normalizecell{-\strain{\cred{\top}}{T^{\sc{s}}}}} & \cred{\normalizecell{\stabdiag{s}{T}}} & \normalizecell{\stab{s}{T}{F}} \cr
0 & 0 & \cred{\normalizecell{-\coupling{{\top}}}} & \cred{\normalizecell{-\strain{{\top}}{T^{\sc{s}}F^{\sc{s}}}}} & \cred{\normalizecell{\stab{s}{T}{F}^{\cred{\top}}}} & \cred{\normalizecell{\stabdiag{s}{F}}} \cr
\end{pmat}
\pmatset{4}{12pt} 
\pmatset{5}{10pt} 
\begin{pmat}[{}]
\normalizecell{\dofs{\rm{M}}{T}{f}} \cr
\normalizecell{\dofs{\rm{P}}{T}{f}} \cr
\normalizecell{\dofs{\rm{P}}{F}{f}} \cr \- 
\normalizecell{\dofs{\rm{S}}{T}{s}} \cr
\normalizecell{\dofs{\rm{V}}{T}{s}} \cr
\normalizecell{\dofs{\rm{V}}{F}{s}} \cr
\end{pmat} =
\pmatset{4}{12pt} 
\pmatset{5}{10pt} 
\begin{pmat}[{}]
\normalizecell{0} \cr
\normalizecell{\cred{\dofs{F}{T}{f}}} \cr
\normalizecell{0} \cr 
\- 
\normalizecell{0} \cr
\normalizecell{\cred{\dofs{F}{T}{s}}} \cr
\normalizecell{0} \cr
\end{pmat} 
$.
}
\label{algebraic_coupling}
\vspace{0.1cm}
\end{equation}

Notice that, as shown in the discrete energy balance (\ref{discrete_energy}), the coupling between acoustic and elastic waves produces no energy, resulting in antisymmetric coupling matrices in (\ref{algebraic_coupling}). For convenience, we re-arrange the unknowns by grouping first the (elastic and acoustic) cell unknowns and then the (elastic and acoustic) face unknowns. Then, (\ref{algebraic_coupling}) rewrites as
\begin{equation}
\resizebox{0.93\textwidth}{!}{%
$
\pmatset{4}{12pt} 
\pmatset{5}{10pt} 
\resizebox{!}{13.5ex}{
$
\begin{pmat}[{..|}]
\cred{\normalizecell{\mass{\rho^{\sc{f}}}{f}}} & 0      & 0 & 0      & 0      & 0 \cr
0      & \cred{\normalizecell{\mass{\frac{1}{\kappa}}{f}}} & 0 & 0      & 0      & 0 \cr
0      & 0      & 0 & 0      & 0      & 0 \cr \-
0      & 0      & 0 & \cred{\normalizecell{\mass{\bbm{C}^{-1}}{s}}} & 0      & 0 \cr
0      & 0      & 0 & 0      & \cred{\normalizecell{\mass{\rho^{\sc{s}}}{s}}} & 0 \cr
0      & 0      & 0 & 0      & 0      & 0 \cr
\end{pmat}
$
}
\dfrac{\rm{d}}{\rm{dt}}
\pmatset{4}{12pt} 
\pmatset{5}{10pt} 
\begin{pmat}[{}]
\normalizecell{\dofs{M}{T}{f}} \cr
\normalizecell{\dofs{P}{T}{f}} \cr \-
\normalizecell{\dofs{S}{T}{s}} \cr
\normalizecell{\dofs{V}{T}{s}} \cr \-
\normalizecell{\dofs{P}{F}{f}} \cr 
\normalizecell{\dofs{V}{F}{s}} \cr
\end{pmat} +
\pmatset{4}{12pt} 
\pmatset{5}{10pt} 
\begin{pmat}[{.|.|}]
0 & \normalizecell{\grad{}{T}{f}} & 0 & 0 & \cred{\normalizecell{\gradd{}{T}{F}{f}}} & 0 \cr
\normalizecell{-\grad{{\cred{\top}}}{T}{f}} & \cred{\normalizecell{\stabdiag{f}{T}}} & 0 & 0 & \normalizecell{\stab{f}{T}{F}} & 0  \cr \-
0 & 0 & 0 & \normalizecell{\strain{}{T^{\sc{s}}}} & 0 & \cred{\normalizecell{\strain{}{T^{\sc{s}}F^{\sc{s}}}}}\cr
0 & 0 & \normalizecell{-\strain{{\cred{\top}}}{T^{\sc{s}}}} & \cred{\normalizecell{\stabdiag{s}{T}}} & 0 & \normalizecell{\stab{s}{T}{F}} \cr \- 
\cred{\normalizecell{-\gradd{{\cred{\top}}}{T}{F}{f}}} & \cred{\normalizecell{\stab{f}{T}{F}^{\cred{\top}}}} & 0 & 0 & \cred{\normalizecell{\stabdiag{f}{F}}} & \normalizecell{\cred{\coupling{}}} \cr
0 & 0 & \cred{\normalizecell{-\strain{{\top}}{T^{\sc{s}}F^{\sc{s}}}}} & \cred{\normalizecell{\stab{s}{T}{F}^{\cred{\top}}}} & \cred{\normalizecell{-\coupling{{\top}}}} & \cred{\normalizecell{\stabdiag{s}{F}}} \cr
\end{pmat}
\pmatset{4}{12pt} 
\pmatset{5}{10pt} 
\begin{pmat}[{}]
\normalizecell{\dofs{M}{T}{f}} \cr
\normalizecell{\dofs{P}{T}{f}} \cr \-
\normalizecell{\dofs{S}{T}{s}} \cr
\normalizecell{\dofs{V}{T}{s}} \cr \-
\normalizecell{\dofs{P}{F}{f}} \cr 
\normalizecell{\dofs{V}{F}{s}} \cr
\end{pmat} =
\pmatset{4}{12pt} 
\pmatset{5}{10pt} 
\begin{pmat}[{}]
\normalizecell{0}  \cr
\normalizecell{\dofs{F}{T}{f}} \cr \-
\normalizecell{0}  \cr 
\normalizecell{\dofs{F}{T}{s}} \cr \-
\normalizecell{0}  \cr
\normalizecell{0}  \cr
\end{pmat} 
$.  
}
\label{algebraic_system}
\end{equation}    
This system can be rewritten in the following compact form: 
\begin{equation}
\pmatset{4}{12pt} 
\pmatset{5}{12pt}
\begin{pmat}[{}]
\cred{\normalizecell{\mass{}{}}} & 0 \cr
0 & 0  \cr 
\end{pmat}
\partial_t\bd{\rm{U}} + 
\pmatset{4}{12pt} 
\pmatset{5}{10pt}
\begin{pmat}[{}]
\normalizecell{\cred{\cal{K}_{\cal{T}}}} & \normalizecell{\cal{K}_{\cal{TF}}} \cr
\normalizecell{\cal{K}_{\cal{FT}}} & \normalizecell{\cred{\cal{K}_{\cal{F}}}}  \cr 
\end{pmat}\bd{\rm{U}} = \bd{\rm{F}},
\label{compact_AR}
\end{equation}
where $\bd{\rm{U}}$ is the vector of unknowns, the blocks with index $\cred{\cal{T}}$ corresponds to the $4 \times 4$ upper-left submatrices in the mass and stiffness matrices, the block with index $\cal{TF}$ to the $4 \times 2$ upper-right submatrices in the stiffness matrix (\ref{algebraic_system}), the block with index $\cal{FT}$ to the $2 \times 4$ lower-left submatrices, and the block with index $\cred{\cal{F}}$ to the $2 \times 2$ lower-right submatrices. Notice that $\cred{\cal{K}_{\cal{F}}}$ has a block-diagonal structure.

\section{Numerical results}\label{sec::numerical_results}

In this section, we present 2D numerical results obtained using the HHO-dG discretizations of the elasto-acoustic problem described above. In particular, we compare equal- and mixed-order settings for the HHO variables and $\cal{O}(1)$- and $\cal{O}(\frac{1}{h})$-stabilizations. We analyze first the spectral properties of the algebraic problem (\ref{compact_AR}). Next, we focus on a test case with analytical solution, so as to verify convergence rates. Finally, we study the case of a Ricker wavelet as an initial condition and we compare our results to \cred{a reference solution obtained by a numerical computation using Green functions.} 

The implementation is carried out in the open-source software Diskpp (available at \url{https://github.com/wareHHOuse/diskpp}), which is described in \cite{CDE_2018}.

In what follows, to precisely specify the level of discretization for each test case, we introduce two computational parameters: the spatial refinement level $\ell$, defined as $h = 2^{-\ell}$, and the time refinement level $n$, defined as $\Delta t = 0.1 \times 2^{-n}$. All the test cases are set up in space-time domains so that $\ell_\Omega \approx 1$ and $T_{\rm{f}} \approx 1$.

\subsection{Spectral analysis}

The goal here is to conduct a numerical spectral analysis of the space semi-discrete problem (\ref{compact_AR}). For that purpose, we consider the generalized eigenvalue problem associated with (\ref{compact_AR}). We define the Schur complement with respect to the face-face block of the stiffness matrix as
\begin{equation}
\cal{K}_{\sc{schur}} := \cred{\cal{K}_{\T}} - \cal{K}_{\T\F}~ \cred{\cal{K}_{\F}^{-1}} ~\cal{K}_{\F\T} .
\end{equation}
The corresponding eigenvalue problem is expressed as
\begin{equation}
\cal{K}_{\sc{schur}}^{\cred{\top}} ~ \cred{\cal{M}_{\T}^{-1}} ~ \cal{K}_{\sc{schur}} ~ \mathbf{X} = \gamma ~ \cred{\cal{M}_{\T}}~\mathbf{X},
\label{gen_eig_pb}
\end{equation}
where $\mathbf{X}$ is the eigenvector and $\gamma$ the eigenvalue. We define the spectral radius as the largest eigenvalue from (\ref{gen_eig_pb}).

Our objective is to investigate the influence of the stabilization on the spectral radius in both equal- and mixed-order settings. To this end, we consider $\cal{O}(1)$-stabilizations, and we introduce additional weights $\eta^\sc{f}$ and $\eta^\sc{s}$ scaling the stabilization bilinear forms. Specifically, we set
\begin{equation}
\cred{\widetilde{s}_{\cal{M}^\sc{f}}(\hat{p}_{\Mf}, \hat{q}_{\Mf}) := \eta^\sc{f} s_{\cal{M}^\sc{f}}(\hat{p}_{\Mf}, \hat{q}_{\Mf}), \qquad \widetilde{s}_{\cal{M}^\sc{s}}(\hat{\bd{v}}_{\Ms} , \hat{\bd{w}}_{\Ms}) := \eta^\sc{s} s_{\cal{M}^\sc{s}}(\hat{\bd{v}}_{\Ms}, \hat{\bd{w}}_{\Ms}).}
\end{equation}
\begin{figure}[!htb]
\centering
\includegraphics[width=0.325\textwidth]{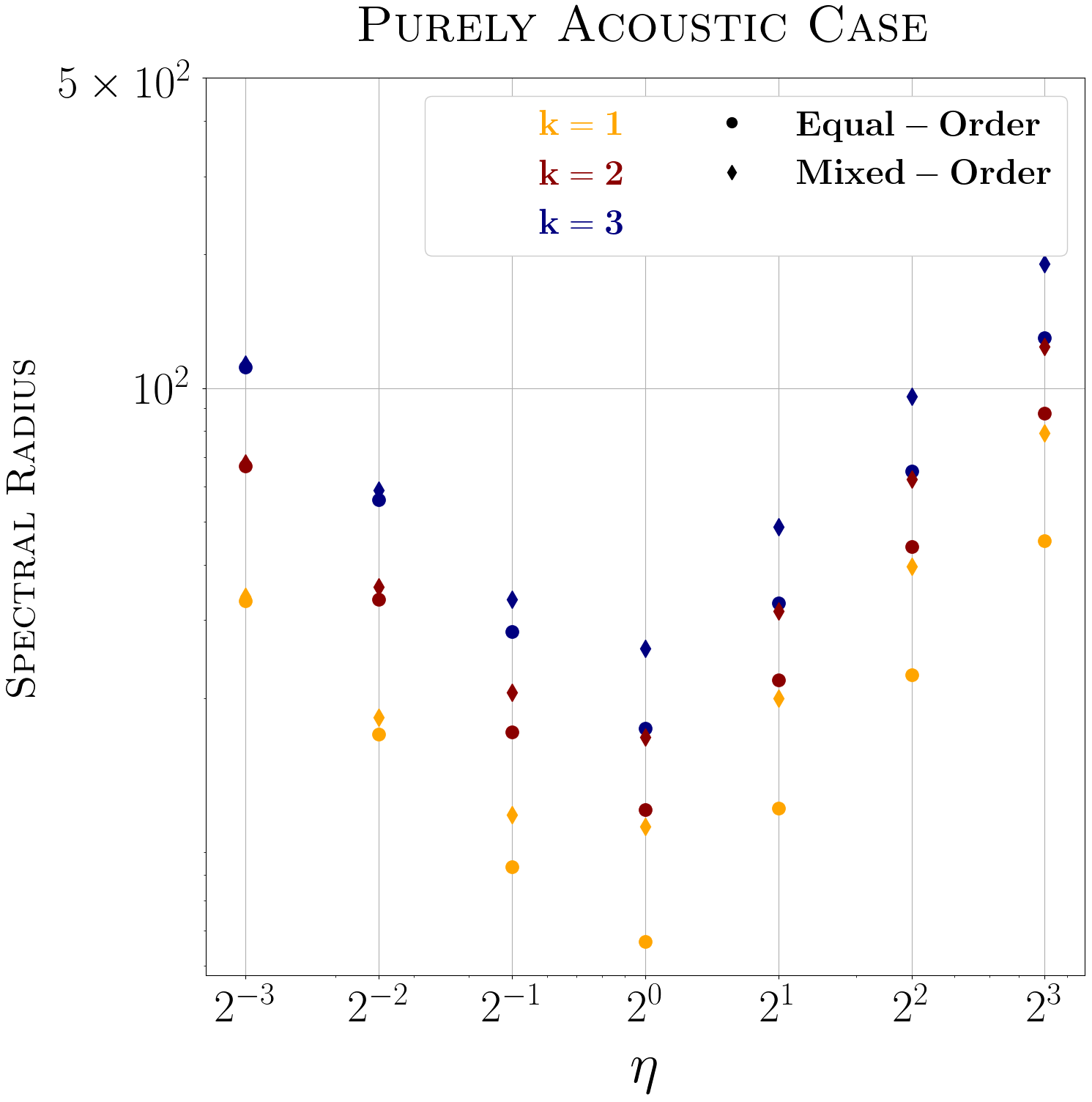}
\includegraphics[width=0.325\textwidth]{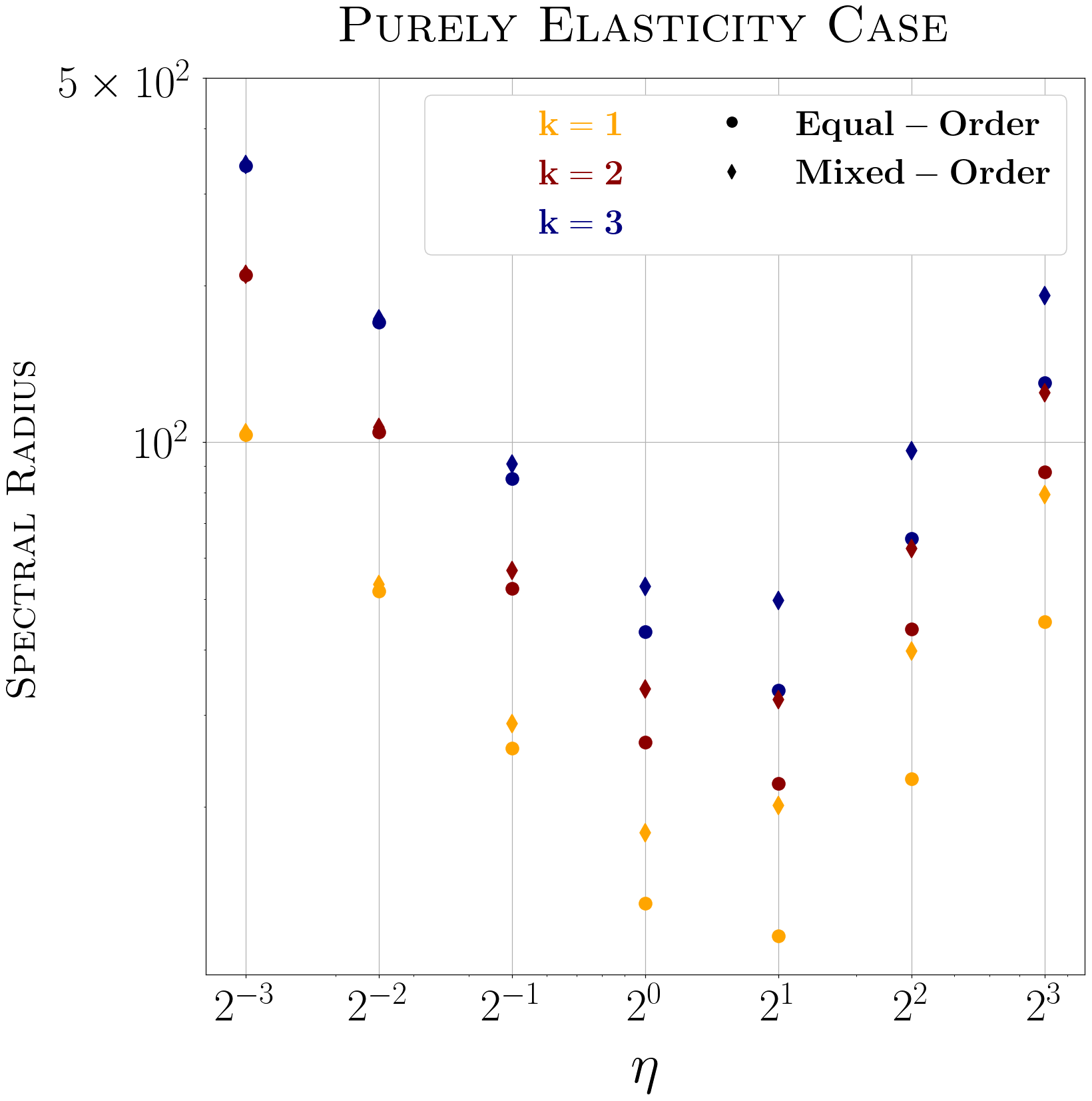}
\includegraphics[width=0.325\textwidth]{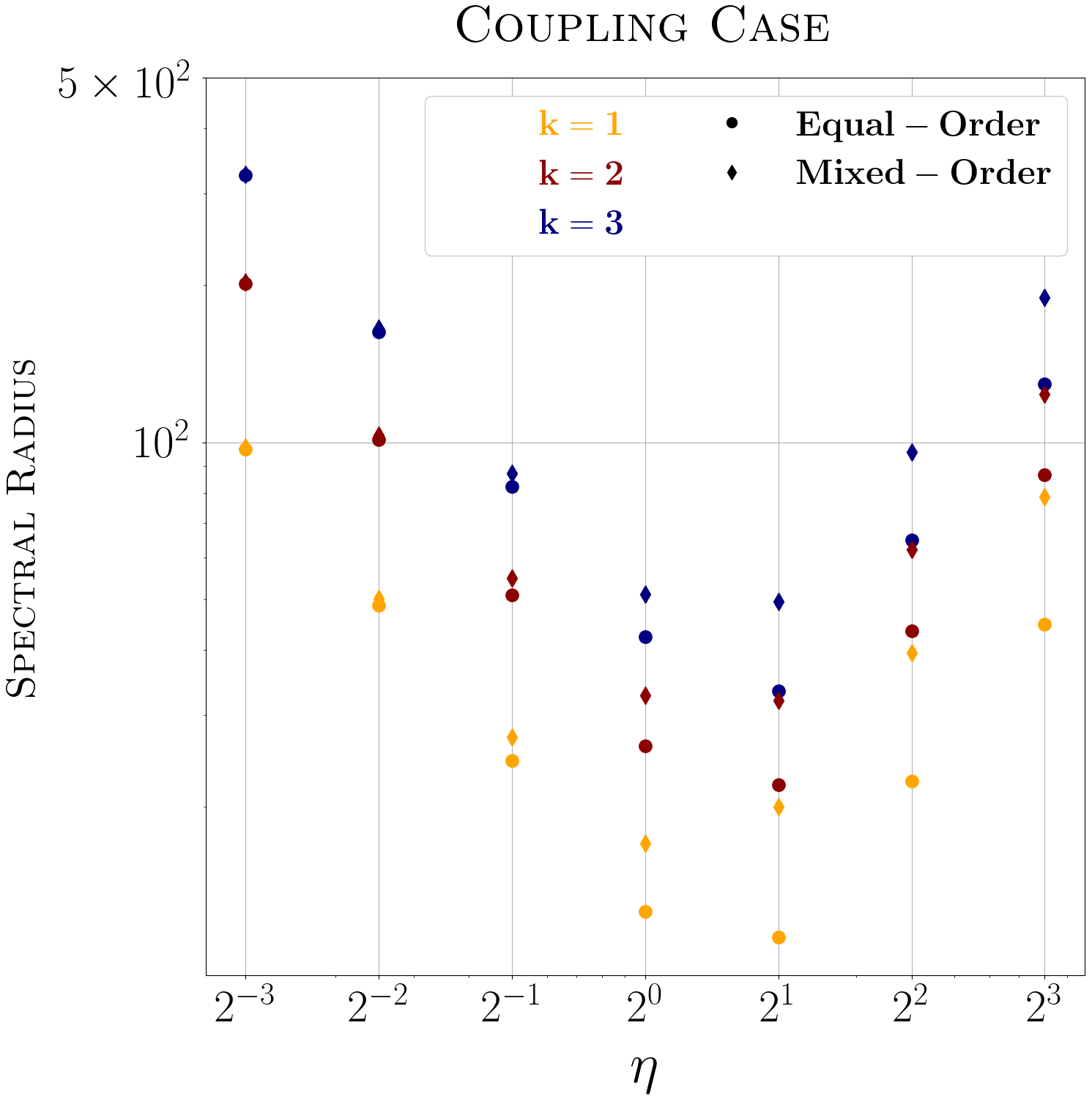}
\caption{Spectral radius in the equal-order and mixed-order settings for the pure acoustic, pure 
linear elasticity, and elasto-acoustic coupling cases with $k \in \{1{:}3\}$}
\label{spectral_radius_3_cases_plots}
\end{figure} 
First, in \hyperref[spectral_radius_3_cases_plots]{\Cref{spectral_radius_3_cases_plots}}, this analysis is conducted for three settings: pure acoustic with weight $\eta^\sc{f}$, pure elastic with weight $\eta^\sc{s}$, and elasto-acoustic coupling with weight $\eta^\sc{f} = \eta^\sc{s} = \eta$. All the spectral \cred{radii} are normalized by the size of the mesh, evaluated as $\sqrt{\# \text{cells}}$. The first observation is that the spectral radius essentially behaves as $\rm{max}(\eta+c_1, \frac{1}{\eta}+c_2)$ for some suitable constants $c_1$ and $c_2$. Therefore, choosing an $\cal{O}(\frac{1}{h})$-stabilization, which corresponds to selecting a high value for $\eta$, leads to a large spectral radius. This is unfavorable if explicit time-stepping schemes are used as it tightens the CFL restriction. Therefore, for an explicit time discretization, a stabilization of order $\cal{O}(1)$ is preferable. The second observation is that, regardless of the chosen discretization (equal- or mixed-order), the spectral radius in the purely elastic case is higher than in the purely acoustic case. This observation is consistent with the fact that elastic waves propagate at higher velocities than acoustic waves. Moreover, the behavior observed in the coupled elasto-acoustic case essentially aligns with that of the purely elastic case. The third observation is that, in all cases, the spectral radius for the equal-order setting is lower than the one for the mixed-order setting, indicating a potentially better CFL condition for the equal-order scheme.  

According to the spectral analysis of the two pure cases, we can set the weights for the acoustic and elastic stabilizations, $\eta^\sc{f}_\star$ and $\eta^\sc{s}_\star$, as follows so as to minimize the spectral radius in each pure case:
\begin{equation}
\eta^\sc{f}_\star = 
\left\{
\begin{aligned}
& 0.88, \\
& 0.80, \\
\end{aligned} 
\right.
\qquad
\eta^\sc{s}_\star = 
\left\{
\begin{aligned}
& 1.54, \qquad \textrm{equal-order},\\
& 1.38, \qquad \textrm{mixed-order}.\\
\end{aligned} 
\right.
\label{optimal_weights}
\end{equation}
In the rest of the paper we keep these settings. \hyperref[spectral_radius_eo_vs_mo_coupling]{\Cref{spectral_radius_eo_vs_mo_coupling}} then reports the spectral radius in the coupled elasto-acoustic case for equal- and mixed-order settings, $k \in \{1{:}3\}$ and with weights prescribed as $\eta^\sc{f} = 2^w \eta^\sc{f}_\star$ and $\eta^\sc{s} = 2^w \eta^\sc{s}_\star$ with $w \in \{-3{:}3\}$ and $\eta^\sc{f}_\star$, $\eta^\sc{s}_\star$ from (\ref{optimal_weights}). \cred{We observe that, in all cases, the lowest spectral radius is obtained for $w=0$.}
\ifHAL
\renewcommand{\a}{0.75}
\else
\renewcommand{\a}{0.7}
\fi
\begin{table}[!htb]
\centering
\resizebox{\a\textwidth}{!}{%
\begin{tabular}{c|c|ccccccc|}
\cline{2-9}
& \diagbox{$k$}{$\eta/\eta_{\star}$} & 1/8 & 1/4 & 1/2 & 1 & 2 & 4 & 8 \\ \hline
\multicolumn{1}{|c|}{} & 1 & 55.3 & 27.8 & 14.2 & 9.9 & 19.7 & 39.5 & 78.9 \\
\multicolumn{1}{|c|}{$\sc{Equal-Order}$} & 2 & 114.8 & 57.8 & 29.5 & 19.5 & 38.3 & 76.3 & 152.5 \\
\multicolumn{1}{|c|}{} & 3 & 185.2 & 93.2 & 47.7 & 29.6 & 57.3 & 113.9 & 227.3 \\ \hline
\multicolumn{1}{|c|}{} & 1 & 94.2 & 48.3 & 26.3 & 16.5 & 20.7 & 41.0 & 81.7 \\
\multicolumn{1}{|c|}{$\sc{Mixed-Order}$} & 2 & 195.0 & 99.3 & 53.0 & 31.8 & 33.2 & 64.7 & 128.7 \\
\multicolumn{1}{|c|}{} & 3 & 314.3 & 159.6 & 84.3 & 50.0 & 51.3 & 99.4 & 197.4 \\ \hline
\end{tabular}}
\caption{Spectral radius in the equal- and mixed-order settings for the elasto-acoustic coupling with reference weights given by (\ref{optimal_weights}) and $k \in \{1{:}3\}$}
\label{spectral_radius_eo_vs_mo_coupling}
\end{table}

Another interesting aspect is the influence of the mesh geometry on the spectral radius. \hyperref[spectral_radius_geometry]{\Cref{spectral_radius_geometry}} reports the spectral radius for polynomial degrees $k \in \{1{:}6\}$ on simplicial, quadrilateral, and polygonal meshes. The latter are generated using the software PolyMesher \cite{TPPM_2012}. We can see that, for an equal-order setting, the spectral radius on quadrangles is slightly better than that on simplices and polygons, with the latter being slightly worse. However, in the mixed-order setting, quadrangles and simplices appear to be roughly equivalent and at higher orders, simplices slightly outperform quadrangles in terms of spectral radius. Polygons remain the worst case in the mixed-order setting, but the difference is not that significant (about $15\%$ increase).
\ifHAL
\renewcommand{\a}{1.0}
\else
\renewcommand{\a}{0.85}
\fi
\begin{table}[!htb]
\centering 
\resizebox{\a\textwidth}{!}{%
\begin{tabular}{c|cc|cc|cc|}
\cline{2-7}
&  \multicolumn{2}{c|}{$\sc{Simplicial meshes} \quad \triangle$} &  \multicolumn{2}{c|}{$\sc{Quadragular meshes} \quad \square$} &  \multicolumn{2}{c|}{$\sc{Polygonal meshes} \quad {\large \varhexagon}$} \\
& $\sc{Equal-Order}$ & $\sc{Mixed-Order}$ & $\sc{Equal-Order}$ & $\sc{Mixed-Order}$ & $\sc{Equal-Order}$ & $\sc{Mixed-Order}$ \\ \hline
\multicolumn{1}{|c|}{$k=1$} & 11.6                  & 17.6                  & 9.9                    & 16.5                   & 10.5                   & 20.1                   \\
\multicolumn{1}{|c|}{$k=2$} & 21.3                  & 31.4                  & 19.5                   & 31.8                   & 20.7                   & 37.2                   \\
\multicolumn{1}{|c|}{$k=3$} & 33.4                  & 47.8                  & 29.6                   & 50.0                   & 35.2                   & 59.4                   \\ 
\multicolumn{1}{|c|}{$k=4$} & 49.2                  & 69.5                  & 45.3                   & 74.0                   & 53.9                   & 86.6                   \\
\multicolumn{1}{|c|}{$k=5$} & 68.0                  & 93.7                  & 61.5                   & 100.6                  & 76.6                   & 118.9                   \\
\multicolumn{1}{|c|}{$k=6$} & 90.1                  & 123.2                 & 83.0                   & 134.0                  & 103.5                  & 156.3                    \\ \hline
\end{tabular}%
}
\caption{Spectral radius for different cell geometries in equal- and mixed-order settings with $k \in \{1{:}6\}$ and optimal setting for $\eta^\sc{f}$ and $\eta^\sc{s}$ from (\ref{optimal_weights})}
\label{spectral_radius_geometry}
\end{table}

\subsection{Time discretization schemes}

We consider both implicit and explicit time-stepping schemes. In the implicit case, we allow for both $\cal{O}(1)$- and $\cal{O}(\frac{1}{h})$-stabilizations, and for simplicity we focus \cred{on} the mixed-order setting since the equal-order setting appears to be slightly more expensive in the static case \cite{CEP_2021}. We consider $s$-stage Singly Diagonally Implicit Runge--Kutta (SDIRK) schemes of order $(s{+}1)$ with $s \in \{2,3\}$. The Butcher tableaux are reported in \hyperref[SDIRK]{\Cref{SDIRK}}. In the explicit case, owing to the CFL restriction, we only consider $\cal{O}(1)$-stabilization, but we include both equal- and mixed-order settings. We consider s-stage Explicit Runge--Kutta (ERK) schemes of order $s$ with $s \in \{2,3,4\}$. The Butcher tableaux are reported in \hyperref[ERK]{\Cref{ERK}}.
\begin{table}[!htb]
\centering 
\subcaptionbox{SDIRK$(2,3)$\label{SDIRK2}}[0.3\textwidth][c]{
\centering
\renewcommand{\arraystretch}{1.25} 
\begin{tabular}{c|cccc}
$\frac{1}{4}$ & $\frac{1}{4}$ & $0$   \\ 
$\frac{3}{4}$ & $\frac{1}{2}$ & $\frac{1}{4}$  \\ 
\hline        & $\frac{1}{2}$ & $\frac{1}{2} $
\end{tabular}
}
\subcaptionbox{SDIRK(3,4)\label{SDIRK3}}[0.3\textwidth][c]{
\centering
\renewcommand{\arraystretch}{1.25} 
\begin{tabular}{c|cccc}
$\theta$      & $\theta$             & $0$         & $0$        \\ 
$\frac{1}{2}$ & $\frac{1}{2}-\theta$ & $\theta$    & $0$        \\ 
$1-\theta$    & 2$\theta$            & $1-4\theta$ & $\theta$ \\ 
\hline        & $\xi$                & $1-2\xi$ & $\xi$
\end{tabular}
}
\caption{Butcher tableaux for the SDIRK$(s,s{+}1)$ schemes. SDIRK(3,4) is obtained with $\theta := \frac{1}{\sqrt{3}} \cos(\frac{\pi}{18}) + \frac{1}{2}$, $\xi := \frac{1}{6(2\theta-1)^2}$.}
\label{SDIRK}
\end{table}

\begin{table}[!htb]
\centering
\begin{minipage}{0.25\textwidth}
\centering
\subcaptionbox{ERK$(2)$\label{ERK2}}{
\renewcommand{\arraystretch}{1.25} 
\begin{tabular}{c|cccc}
$0$           & $0$           & $0$     \\ 
$\frac{1}{2}$ & $\frac{1}{2}$ & $0$ \\ 
\hline        & $0$ & $1$
\end{tabular}
}
\end{minipage}%
\hspace{0.05\textwidth} 
\begin{minipage}{0.25\textwidth}
\centering
\subcaptionbox{ERK$(3)$\label{ERK3}}{
\renewcommand{\arraystretch}{1.25} 
\begin{tabular}{c|cccc}
$0$           & $0$           & $0$           & $0$ \\ 
$\frac{1}{2}$ & $\frac{1}{2}$ & $0$           & $0$ \\ 
$1$           & $-1$          & $2$           & $0$ \\ 
\hline        & $\frac{1}{6}$ & $\frac{2}{3}$ & $\frac{1}{6}$
\end{tabular}
}
\end{minipage}%
\hspace{0.05\textwidth} 
\begin{minipage}{0.25\textwidth}
\centering
\subcaptionbox{ERK$(4)$\label{ERK4}}{
\renewcommand{\arraystretch}{1.25} 
\begin{tabular}{c|cccc}
$0$           & $0$           & $0$           & $0$           & $0$\\
$\frac{1}{2}$ & $\frac{1}{2}$ & $0$           & $0$           & $0$\\ 
$\frac{1}{2}$ & $0$           & $\frac{1}{2}$ & $0$           & $0$\\ 
$1$           & $0$           & $0$           & $1$           & $0$ \\ 
\hline        & $\frac{1}{6}$ & $\frac{1}{3}$ & $\frac{1}{3}$ & $\frac{1}{6}$
\end{tabular}
}
\end{minipage}
\caption{Butcher tableaux for the ERK$(s)$ schemes.}
\label{ERK}
\end{table}

\subsection{Convergence rates for smooth analytical solutions}

In this section, we study the convergence rates on smooth analytical solutions. Both the acoustic and elastic media have the same density and similar wave speeds. Specifically, we consider the simulation time $T_{\rm{f}} = 1$ and
\begin{itemize}
\item $\domain{f} := (0,1)\times(0,1)$, with density $\rho^\sc{f} := 1$, \cred{fluid bulk} modulus $\kappa := 1$, and velocity of the pressure waves $c_{\sc{p}}^\sc{f} := 1$,
\item $\domain{s} := (-1,0)\times(0,1)$, with density $\rho^\sc{s} := 1$, and Lamé parameters so that $c_\sc{p}^\sc{s} := \sqrt{3}$ and $c_\sc{s} := 1$.
\end{itemize}
The analytical solution is expressed in terms of the potentials $u$ (acoustic) and $\bd{u} := (u_x,u_y)$ (elastic) so that
\begin{subequations}
\begin{alignat}{3}
p &:= \partial_t u &\qquad\qquad \bd{m} &:= \nabla u & \qquad \qquad &\text{in } \domain{f},\\
\bd{v} & := \partial_t \bd{u} & \qquad \qquad \bbm{C}^{-1}{:}\bbm{s} & := \nabla_{\rm{sym}} \bd{u} &\qquad\qquad & \text{in } \domain{s}.
\end{alignat}
\end{subequations}   
The source terms, the (non)homogeneous Dirichlet boundary conditions, and the initial conditions are defined according to the following choices for the potentials, which indeed satisfy the coupling conditions (\ref{coupling1})-(\ref{coupling2}):
\begin{enumerate}
\item Polynomial in space, so that the temporal error is the only error component:
\begin{subequations} 
\begin{equation}
u := (1-x)x^2(1-y)y\sin(\sqrt{2}\pi t) \qquad u_x = u_y := (1+x)x^2(1-y)y\sin(\sqrt{2}\pi t);
\label{poly_in_space}
\end{equation}      
\item Polynomial in time, so that the spatial error is the only error component:
\begin{equation}
u = u_x = u_y := x\sin(\pi x)\sin(\pi y)t^2.
\label{poly_in_time}
\end{equation}        
\end{subequations}
\end{enumerate}

\ifHAL
\renewcommand{\a}{0.4}
\else
\renewcommand{\a}{0.3}
\fi
\begin{figure}[!htb]
\centering 
\includegraphics[width=\a\textwidth]{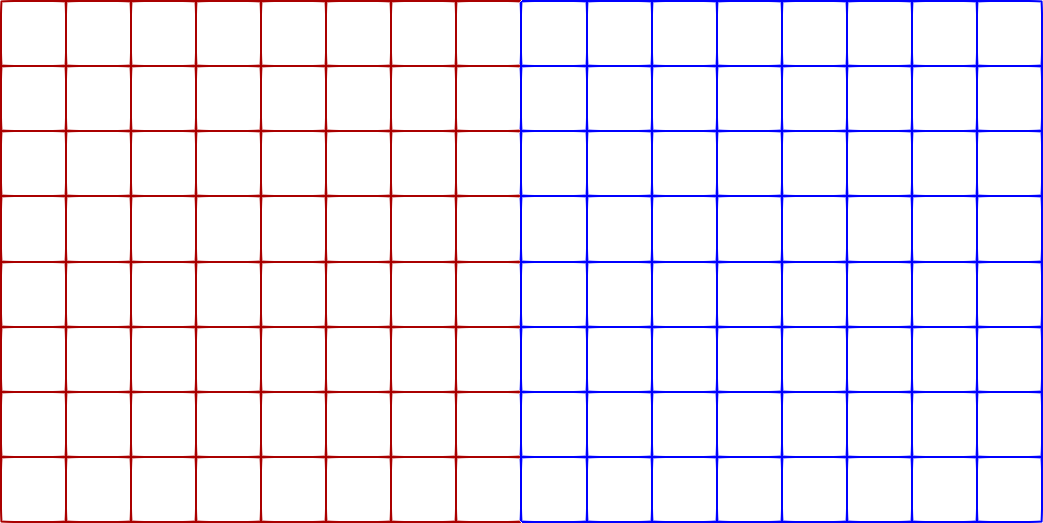}
\hspace{1cm}
\includegraphics[width=\a\textwidth]{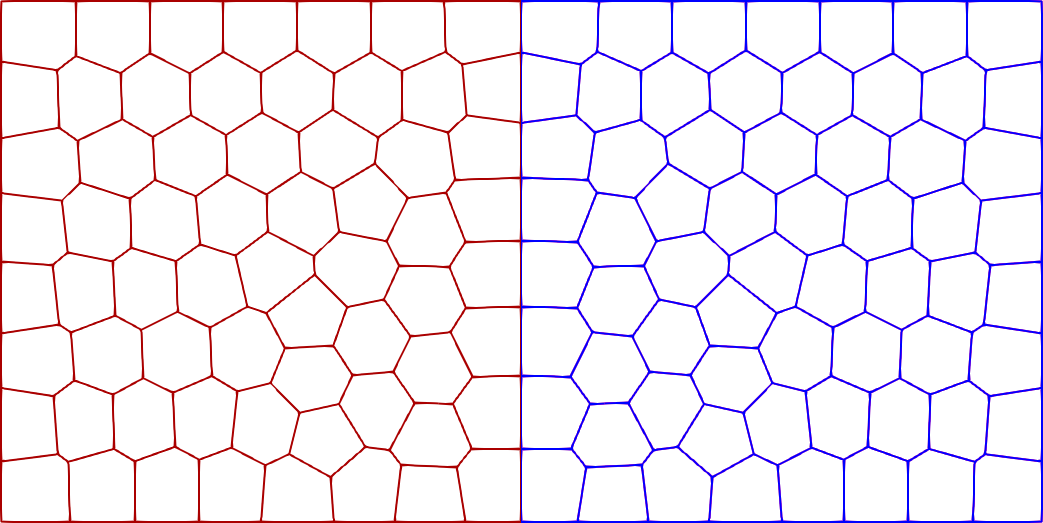}
\caption{Cartesian and polygonal meshes for $\ell = 2$}
\label{meshes}
\end{figure}
We consider two types of meshes: Cartesian and polygonal meshes. Some examples are shown in Figure \ref{meshes} for $\ell = 2$, with the elastic subdomain mesh on the left side in red, and the acoustic subdomain mesh on the right side in blue. In each subdomain, two types of unknowns contribute to the mechanical energy: (i) the cell components of the HHO unknowns; (ii) the dG unknowns. Thus, we set 
\begin{subequations}
\begin{align}
\|p_{\Tf},\bd{v}_{\Ts}\|^2_{\rm{HHO}} & := \|p_{\Tf}(t)\|^2_{L^2(\frac{1}{\kappa};\domain{f})} + \|\bd{v}_{\Ts}(t)\|^2_{\bd{L}^2(\rho^\sc{s}; \domain{s})},\\
\|\bd{m}_{\Tf},\bbm{s}_{\Ts}\|^2_{\rm{dG}} & := \|\bd{m}_{\Tf}(t)\|^2_{\bd{L}^2(\rho^\sc{f};\domain{f})} + \|\bbm{s}_{\Ts}(t)\|^2_{\bbm{L}^2(\bbm{C}^{-1};\domain{s})}.
\label{errors}
\end{align}
\end{subequations}
In what follows, we report these two contributions separately, since they can feature different convergence rates.

We first consider convergence rates in time. For this purpose, we use the analytical solution (\ref{poly_in_space}). \hyperref[fig::conv_rates_time]{\Cref{fig::conv_rates_time}} shows that, as expected, optimal convergence rates in time are reached in the mixed-order setting and $\cal{O}(\frac{1}{h})$-stabilization: order $s$ for ERK$(s)$ schemes and order $(s{+}1)$ for SDIRK$(s, s{+}1)$ schemes. The same results are obtained in the equal-order setting, for $\cal{O}(1)$-stabilization, or on polyhedral meshes (results omitted for brevity).
\begin{figure}[!htb] 
\centering
\includegraphics[width=0.325\textwidth]{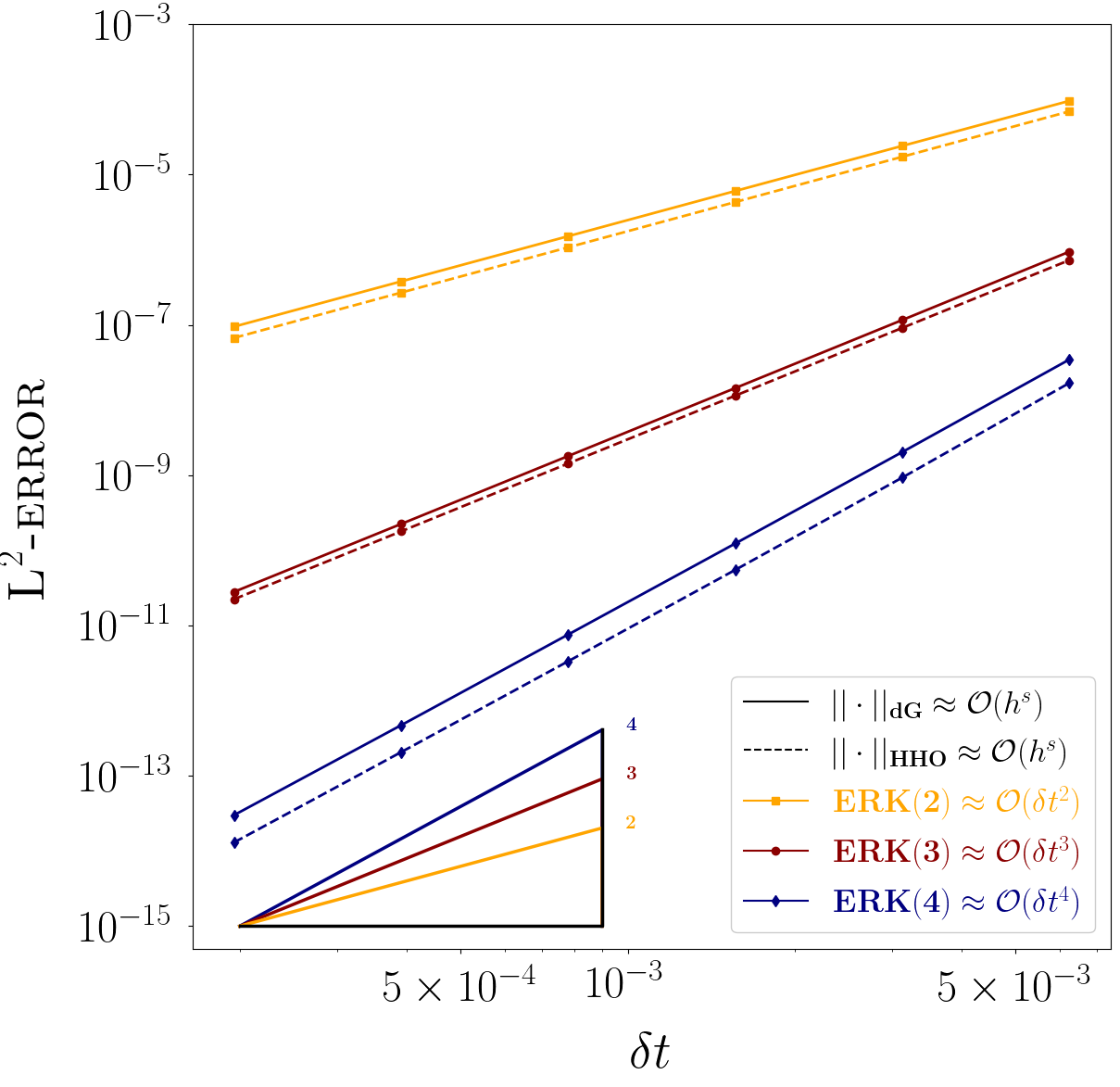}
\hspace{1cm}
\includegraphics[width=0.325\textwidth]{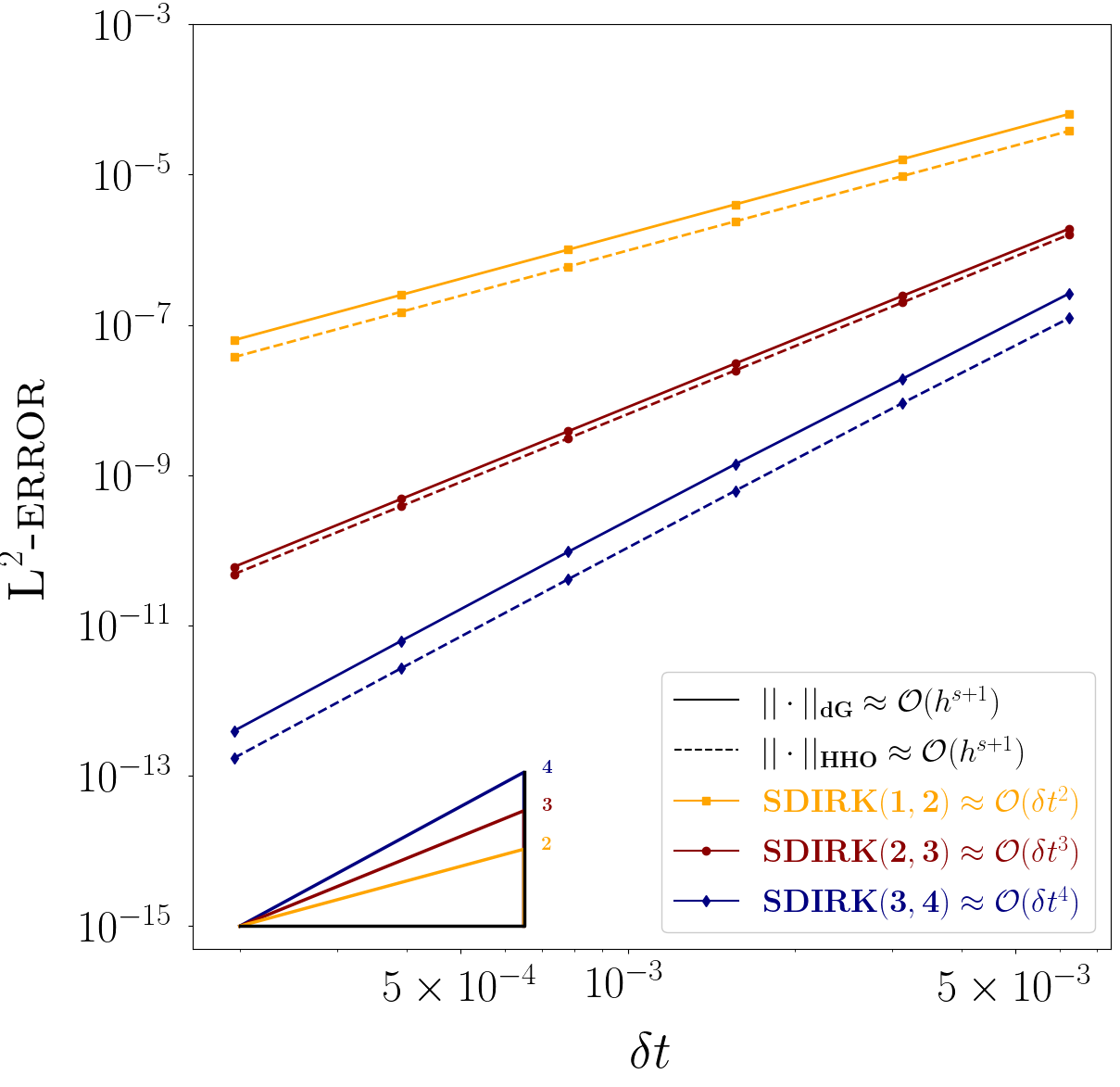}
\caption{Errors as a function of the time step for the analytical solution (\ref{poly_in_space}) with a mixed-order discretization, $k=4$, $\ell=1$, $n \in \{4, 5, 6, 7, 8, 9, 10\}$ and $\cal{O}(\frac{1}{h})$-stabilization. \textbf{Left}: SDIRK($s,s{+}1$) with $s \in \{1, 2, 3\}$. \textbf{Right}: ERK($s$) with $s \in \{2, 3, 4\}$.}
\label{fig::conv_rates_time}
\end{figure}   

We consider now convergence rates in space on Cartesian meshes. For this purpose, we use the analytical solution (\ref{poly_in_time}). The left and central panels of \hyperref[fig::conv_rates_cartesian]{\Cref{fig::conv_rates_cartesian}} present the errors (\ref{errors}) as a function of the mesh size for ERK schemes with $\cal{O}(1)$-stabilization in both equal- and mixed-order settings. The right panel of \hyperref[fig::conv_rates_cartesian]{\Cref{fig::conv_rates_cartesian}}, reports errors (\ref{errors}) as a function of the mesh size for SDIRK schemes in mixed-order setting with $\cal{O}(\frac{1}{h})$-stabilization are reported. For ERK schemes with $\cal{O}(1)$-stabilization in equal- and mixed-order settings, both HHO and dG unknowns exhibit the expected convergence rate of order $(k+1)$ (with sometimes a slight suboptimality for the dG unknowns). In contrast, SDIRK schemes with $\cal{O}(\frac{1}{h})$-stabilization achieve improved convergence rates on the HHO unknowns which now converge at order $(k+2)$, whereas dG unknowns still converge at order $(k+1)$. We still notice a superconvergence phenomenon at lower polynomial orders for the dG unknowns.
\begin{figure}[!htb]
\centering
\includegraphics[width=0.325\textwidth]{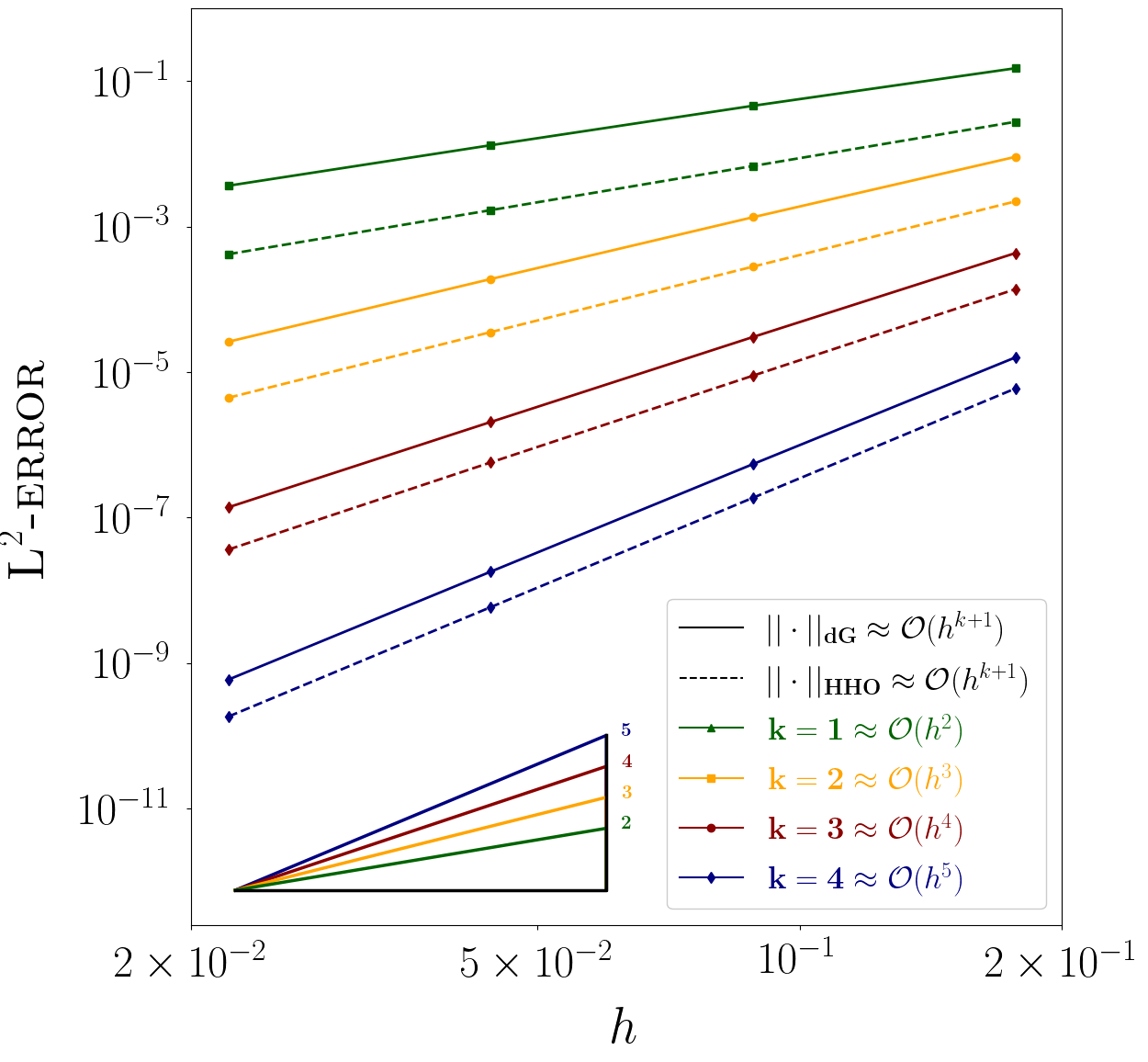}
\includegraphics[width=0.325\textwidth]{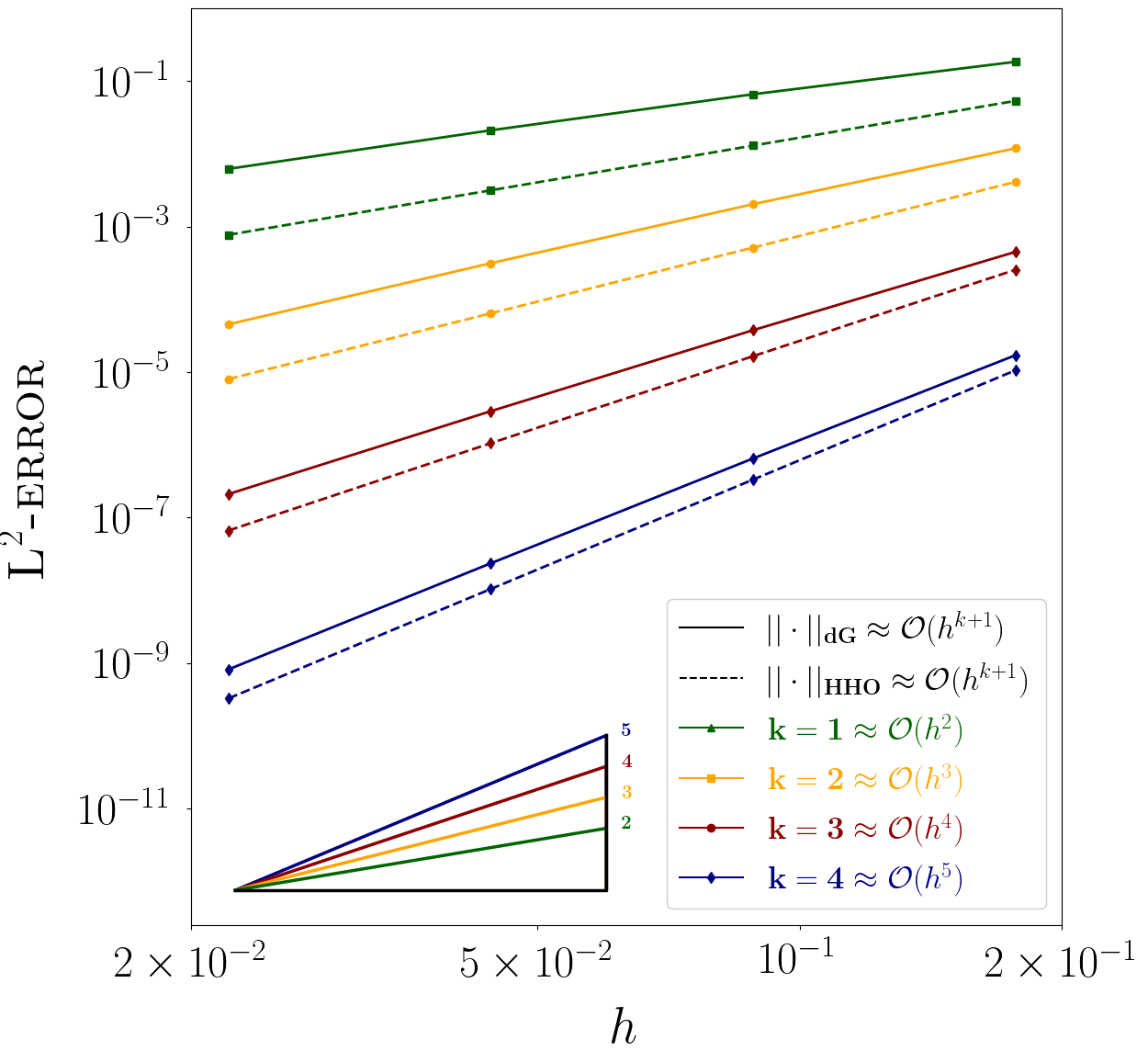}
\includegraphics[width=0.325\textwidth]{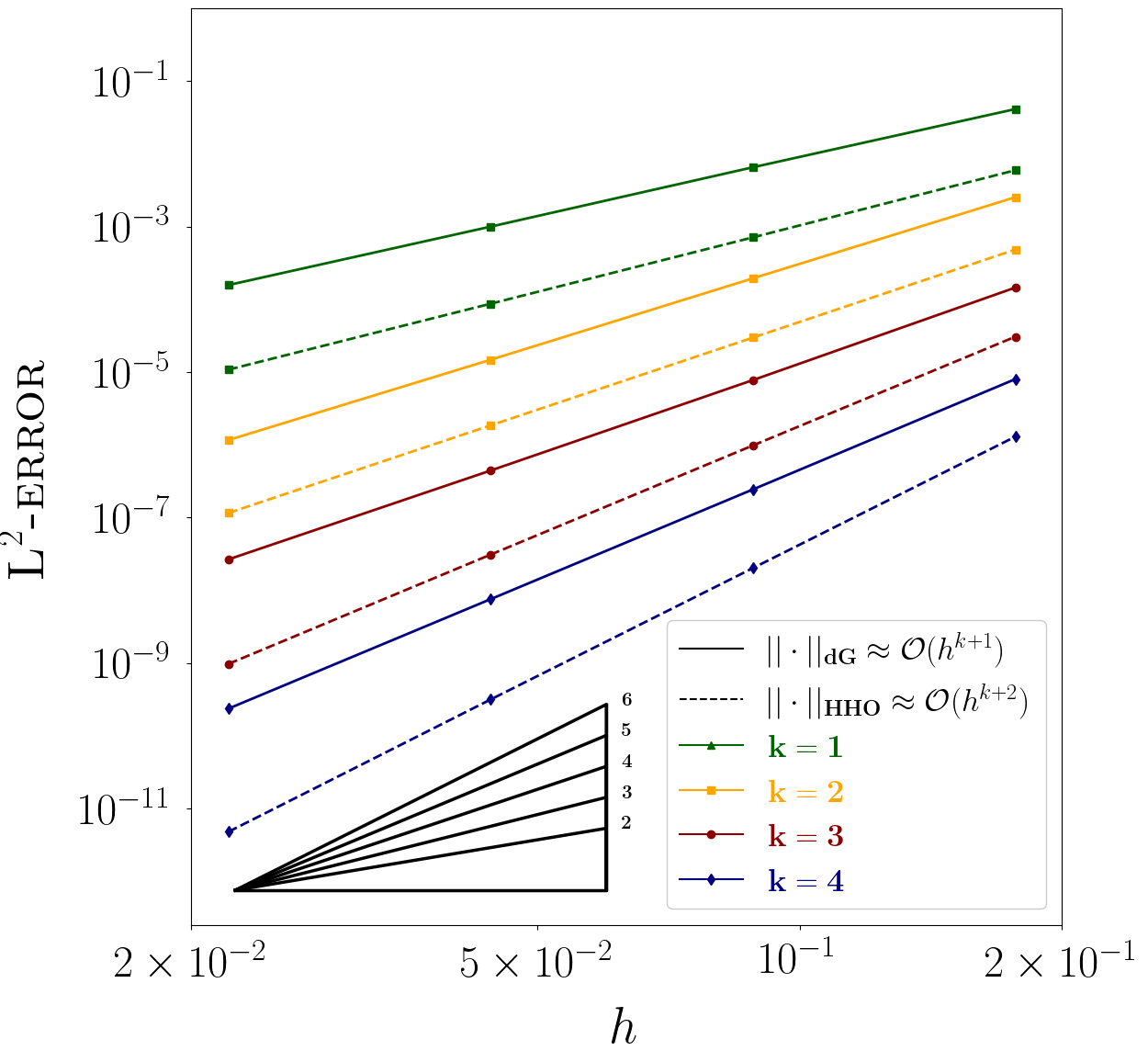}
\caption{Errors (\ref{errors}) as a function of the mesh-size (Cartesian meshes) for the analytical solution (\ref{poly_in_time}). \textbf{Left:} Equal-order with $\cal{O}(1)$-stabilization. \textbf{Center:} Mixed-order with $\cal{O}(1)$-stabilization. \textbf{Right:} Mixed-order with $\cal{O}(\frac{1}{h})$-stabilization. {Computational parameters:} $n=8$, $\ell \in \{2,3,4,5\}$.}
\label{fig::conv_rates_cartesian} 
\end{figure}

Finally, we evaluate spatial convergence rates on polyhedral meshes. \hyperref[fig::conv_rates_poly]{\Cref{fig::conv_rates_poly}} presents the results for ERK schemes in equal- and mixed-order settings with $\cal{O}(1)$-stabilization and SDIRK schemes in mixed-order setting with $\cal{O}(\frac{1}{h})$-stabilization. The conclusions on polyhedral meshes corroborate those on Cartesian meshes.
\begin{figure}[!htb]
\centering
\includegraphics[width=0.325\textwidth]{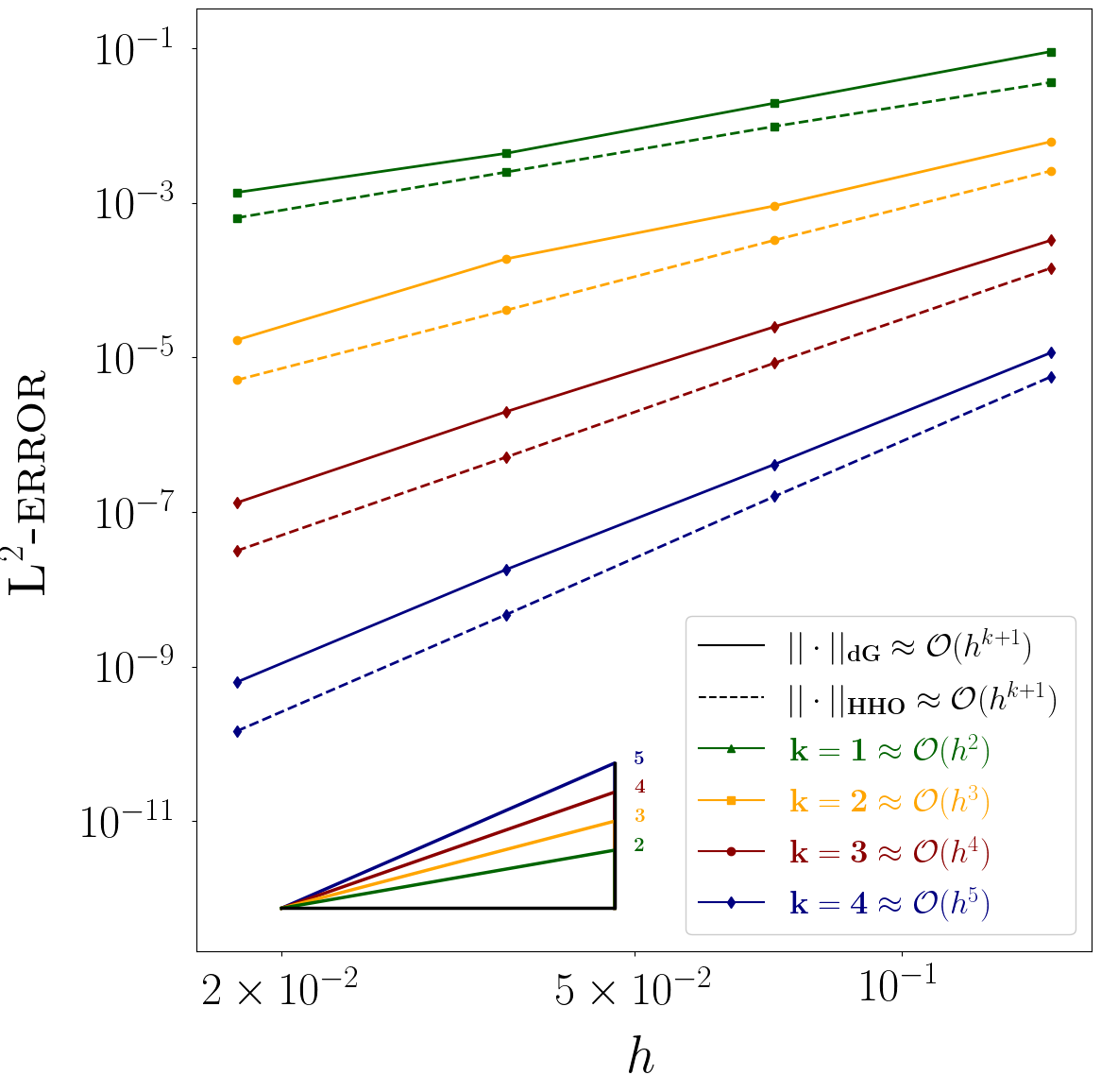}
\includegraphics[width=0.325\textwidth]{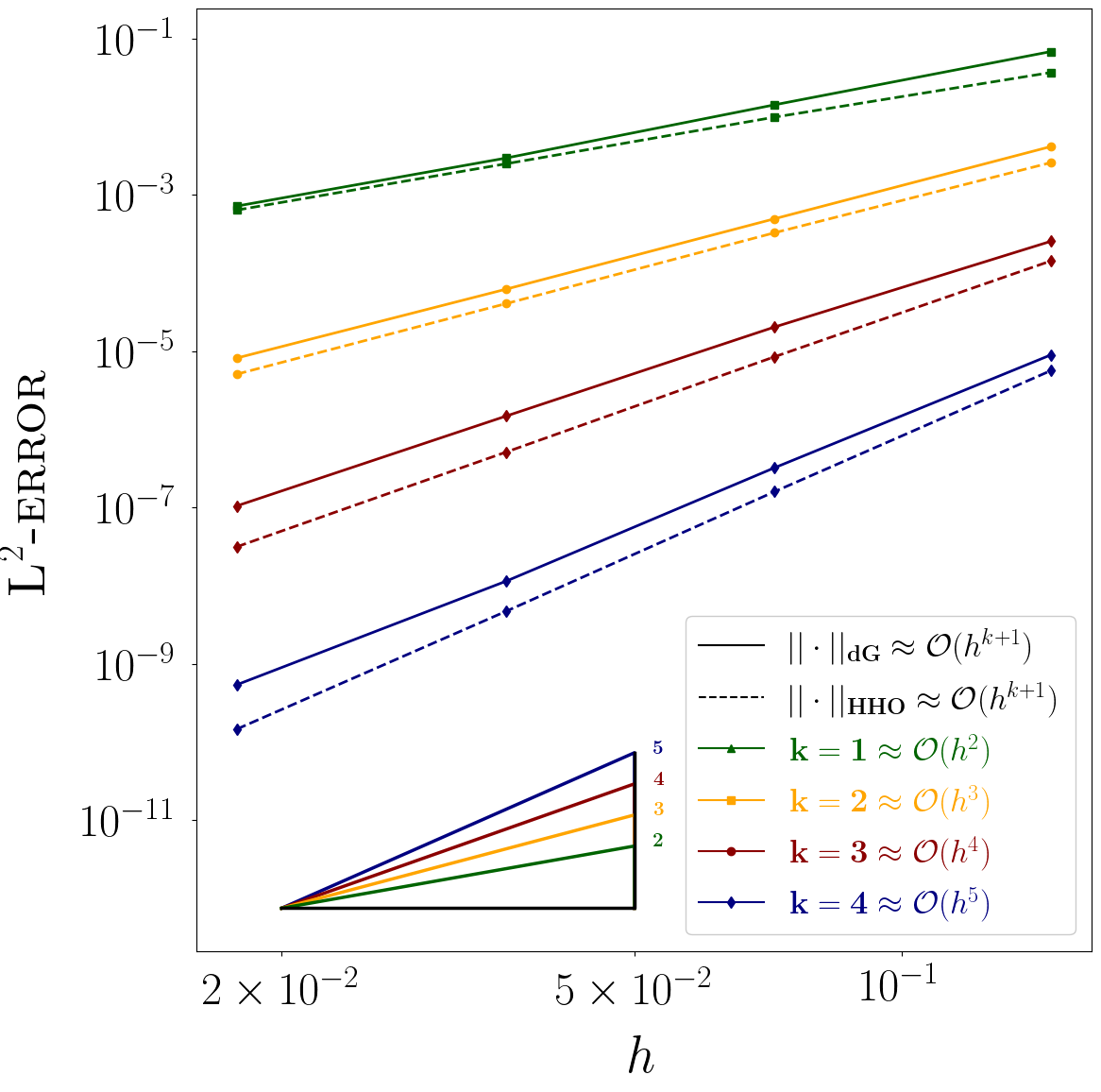}
\includegraphics[width=0.325\textwidth]{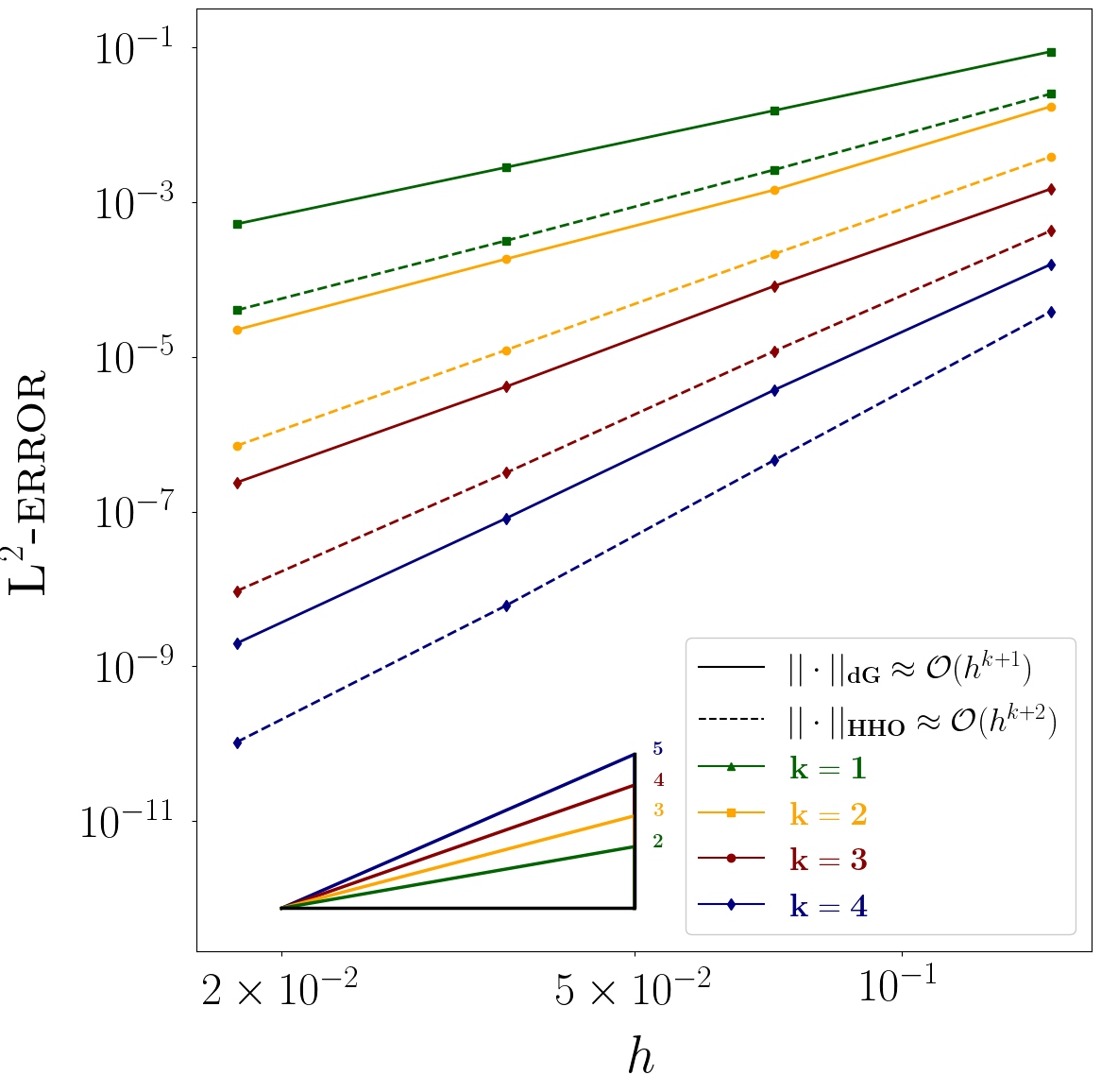}
\caption{Errors (\ref{errors}) as a function of the mesh size (Polyhedral meshes) for the analytical solution (\ref{poly_in_time}). \textbf{Left:} Equal-order with $\cal{O}(1)$-stabilization. \textbf{Center:} Mixed-order with $\cal{O}(1)$-stabilization. \textbf{Right:} Mixed-order with $\cal{O}(\frac{1}{h})$-stabilization. {Computational parameters:} $n=8$, $\ell \in \{2,3,4,5\}$.}
\label{fig::conv_rates_poly} 
\end{figure} 

\FloatBarrier
\subsection{Numerical study of Ricker wavelet as initial condition}

This test case deals with the propagation of an elasto-acoustic wave through a heterogeneous domain. \cred{Referring} to \hyperref[domain]{\Cref{domain}}, let $H$ be the height of the domain, $H_e$ the height of the elastic subdomain and $L$ the length of both subdomains. We set $\domain{f} := (0, L) \times (0, H-H_e)$ and $\domain{s} := (0, L) \times (-H_e, 0)$. Two sensors $\cal{S}^{\sc{f}}$ and $\cal{S}^{\sc{s}}$ are positioned on a vertical line as in the Figure \ref{domain} so as to have one sensor in each subdomain. For the acoustic sensor $\cal{S}^{\sc{f}}$, the acoustic pressure $p_{\Tf}$ is considered, and for the elastic sensor $\cal{S}^{\sc{s}}$ the $x$- and $y$-composents of the elastic velocity $\bd{v}_{\Ts}$ are considered. Homogeneous Dirichlet boundary conditions are enforced, the source terms are null, and the initial condition corresponds to a velocity Ricker wavelet centered at the point $(x_c, y_c) \in \Omega^\sc{f}$ (in purple on Figure \ref{domain}) given by the following expression:
\begin{equation}
\bd{m_0}(x, y) := \theta \exp \bigg(-\pi^2 \frac{r^2}{\cred{\Lambda}^2}\bigg)(x-x_c, y-y_c)^{\cred{\top}}
\end{equation}
with $\theta := 10 \left[\frac{1}{\rm{s}}\right]$, $\cred{\Lambda} := \frac{c^\sc{f}_\sc{p}}{f_c}[\rm{m}]$ with $f_c := 10\left[\frac{1}{\rm{s}}\right]$ and $r^2 := (x-x_c)^2 + (y-y_c)^2$. This initial condition corresponds to a velocity Ricker wave centered at the point $(x_c, y_c) \in \Omega$. 
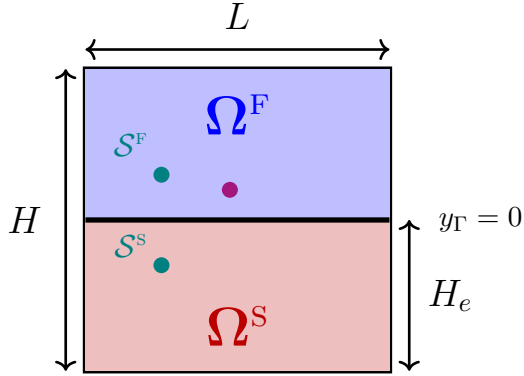
\begin{figure}[!htb]
\centering 
\begin{tikzpicture}[scale=4]
\draw[line width = 2, color = black] (0,0) rectangle (1,1);
\fill[color=ceared!25] (0, 0) rectangle (1, 0.5);
\fill[color=blue!25] (0, 0.5) rectangle (1,1);
\draw[line width = 2, color = black] (0,0.5) -- (1,0.5);
\draw[<->, line width = 1] (-0.065,0) -- (-0.065,1);
\draw[<->, line width = 1] (0,1.065) -- (1,1.065);
\draw[<->, line width = 1] (1.065,0) -- (1.065,0.5);
\node at (0.5,0.85)  {\huge \color{blue}   $\bd{\domain{f}}$};
\node at (0.5,0.15)  {\huge \color{ceared} $\bd{\domain{s}}$};
\node at (0.5,1.175)  {\Large $L$};
\node at (-0.2,0.5) {\Large $H$};
\node at (1.2,0.25) {\Large $H_e$};
\node at (1.3,0.5)  {\normalsize $y_{\G}=0$};
\fill[color_dofs]  (0.475,0.6) circle (0.75pt);
\fill[teal] (0.25,0.65) circle (0.75pt);
\fill[teal] (0.25,0.35) circle (0.75pt);
\node at (0.15,0.75) {\large \color{teal}   $\cal{S}^{\sc{f}}$};
\node at (0.15,0.415) {\large \color{teal}   $\cal{S}^{\sc{s}}$};
\end{tikzpicture}
\caption{Ricker wavelet test case}
\label{domain}
\end{figure}  

\ifHAL
\FloatBarrier
\else
\fi

\subsubsection{Academic setting}

We first consider an academic case in which the acoustic and elastic media have the same density and propagate S-waves at the same speed as compressional acoustic waves, \textit{i.e.},
\begin{equation}
\rho^\sc{f} = \rho^\sc{s} = 1, \qquad c_{\sc{p}}^\sc{s} = \sqrt{3}, \qquad c_{\sc{p}}^\sc{f} = c_{\sc{s}}^\sc{s} = 1 .
\label{academic_properties}
\end{equation}  
As the material properties are similar, we consider two subdomains of the same dimension, with $L=H=1$, $H_e=0.5$ and $x_c := 0$, $y_c := 0.125$ for the origin of the pulse in the acoustic subdomain. The simulation time is set to $T_{\rm{f}} := 10$.

The following results are obtained using the SDIRK$(3,4)$ time scheme with a mixed-order setting and $\cal{O}(1)$-stabilization. \hyperref[snapshot_academic_pulse]{\Cref{snapshot_academic_pulse}} displays the two-dimensional pressure distribution in the acoustic subdomain and the \cred{Euclidean} norm of the velocity in the elastic subdomain at times $t \in \{0, 0.25, 0.27, 0.32\}$. We can see that the simulation propagates correctly the acoustic Ricker wavelet through the elasto-acoustic interface. The difference in properties, intrinsic to the nature of the media, causes a small reflection of the acoustic wave when it meets the interface. Moreover, we observe the transmission of the compressional wave as a P-wave (with the larger celerity), as well as the creation of an S-wave in the elastic part of the domain.

\ifHAL
\begin{figure}[!htb]
\centering
\includegraphics[width=0.3\textwidth]{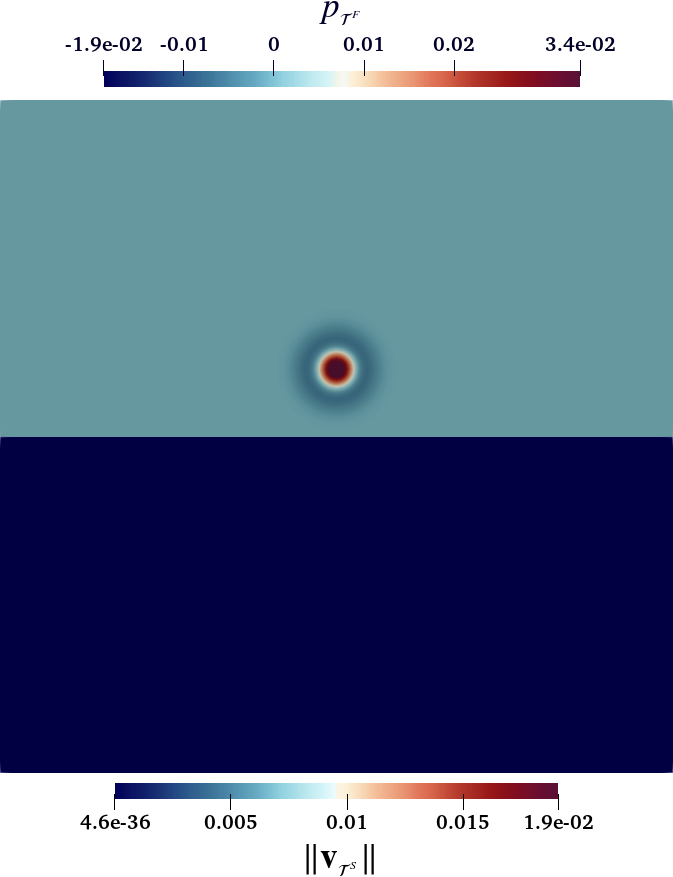}
\hspace{1cm}
\includegraphics[width=0.3\textwidth]{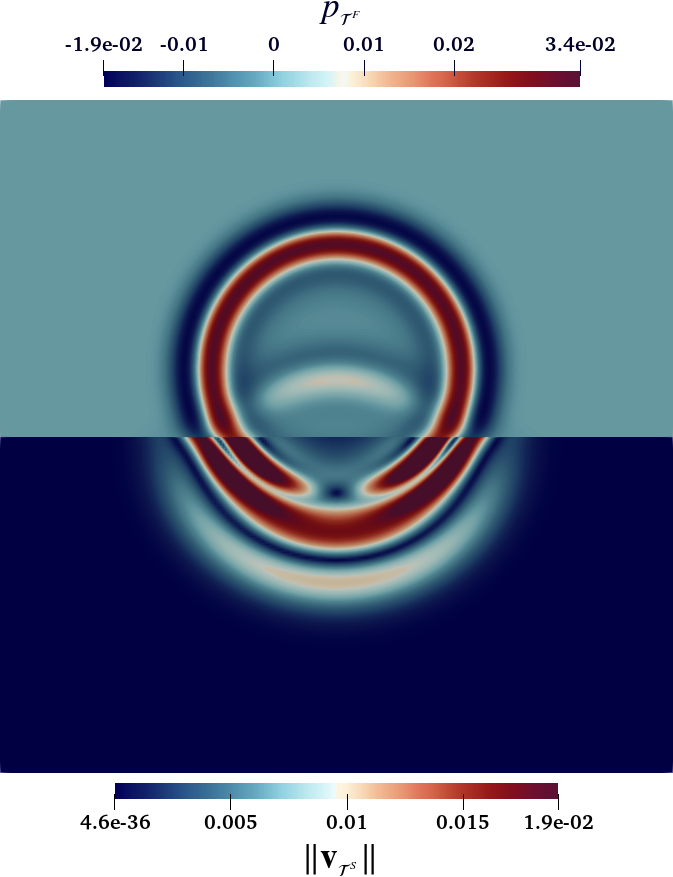}\\[0.35cm]
\includegraphics[width=0.3\textwidth]{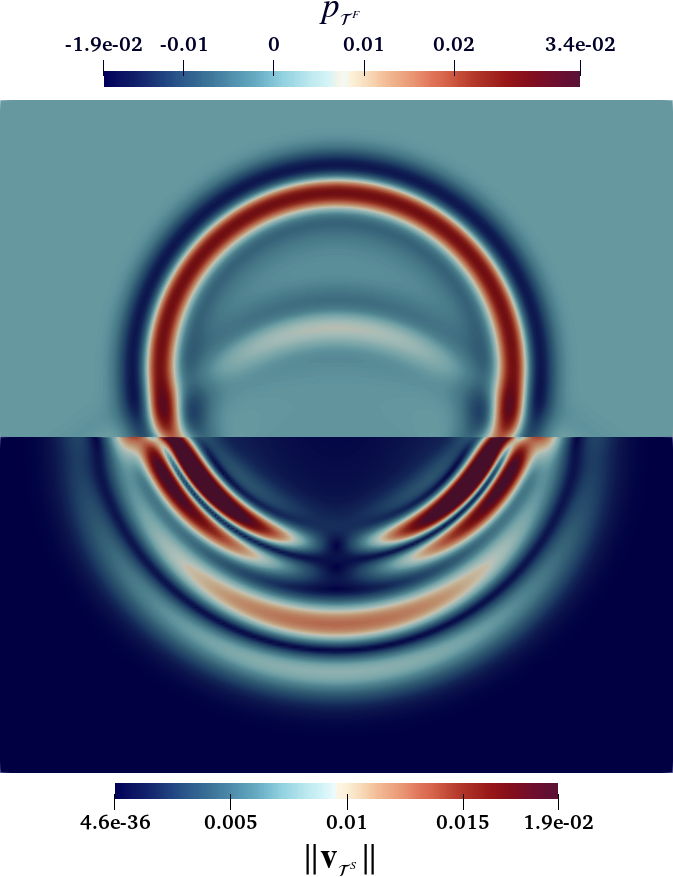}
\hspace{1cm}
\includegraphics[width=0.3\textwidth]{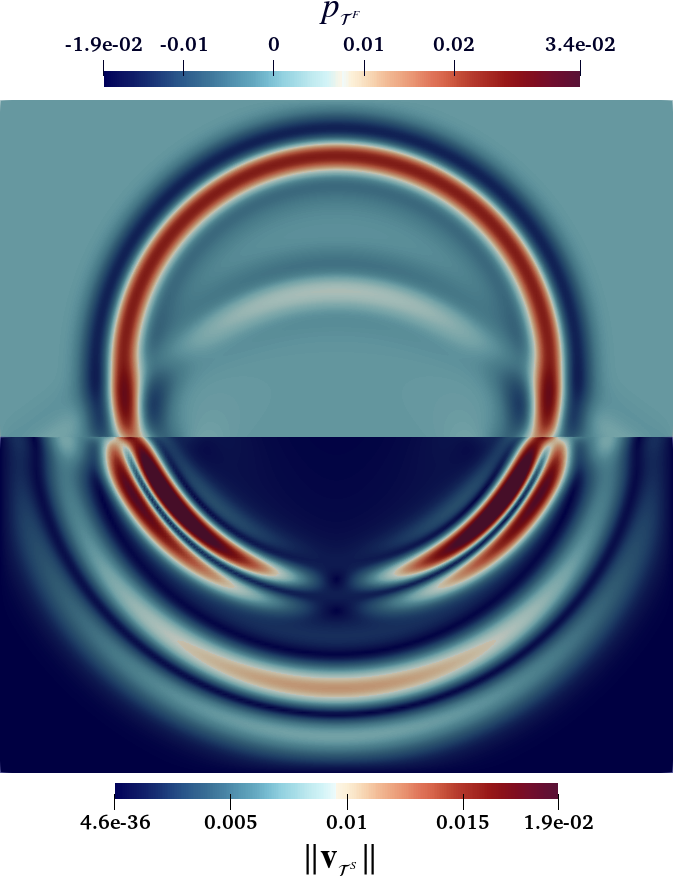}
\caption{Spatial distribution of the acoustic pressure (upper side) and the elastic velocity norm (lower side) at times $t \in \{0, 0.25, 0.27, 0.32\}$ predicted by SDIRK$(3,4)$ scheme with mixed-order setting, $\cal{O}(\frac{1}{h})$-stabilization, $k=1$, $\ell = 7$, and $n=9$.}
\label{snapshot_academic_pulse}
\end{figure}
\else
\begin{figure}[!htb]
\centering
\includegraphics[width=0.23\textwidth]{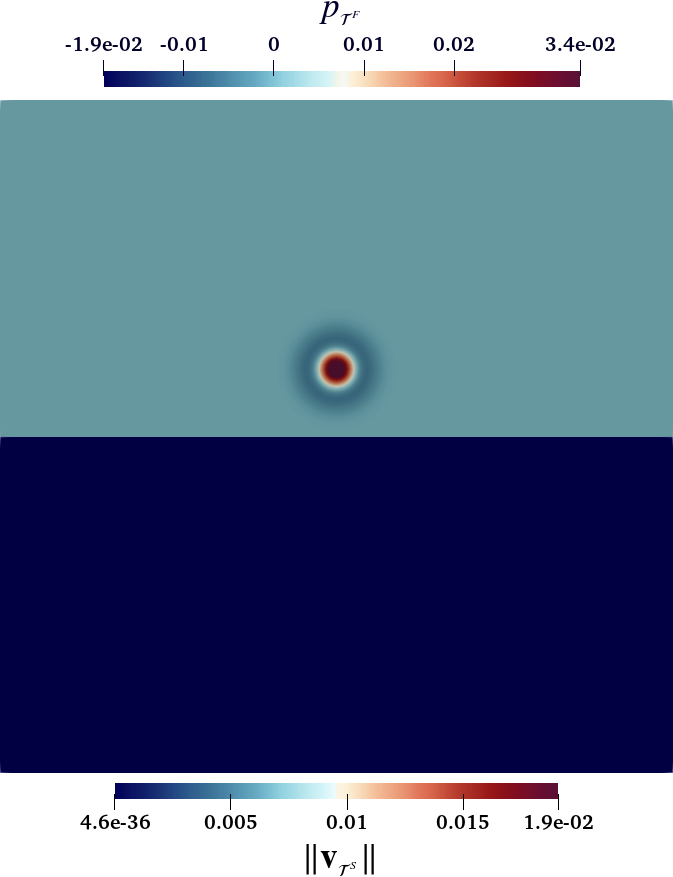}
\includegraphics[width=0.23\textwidth]{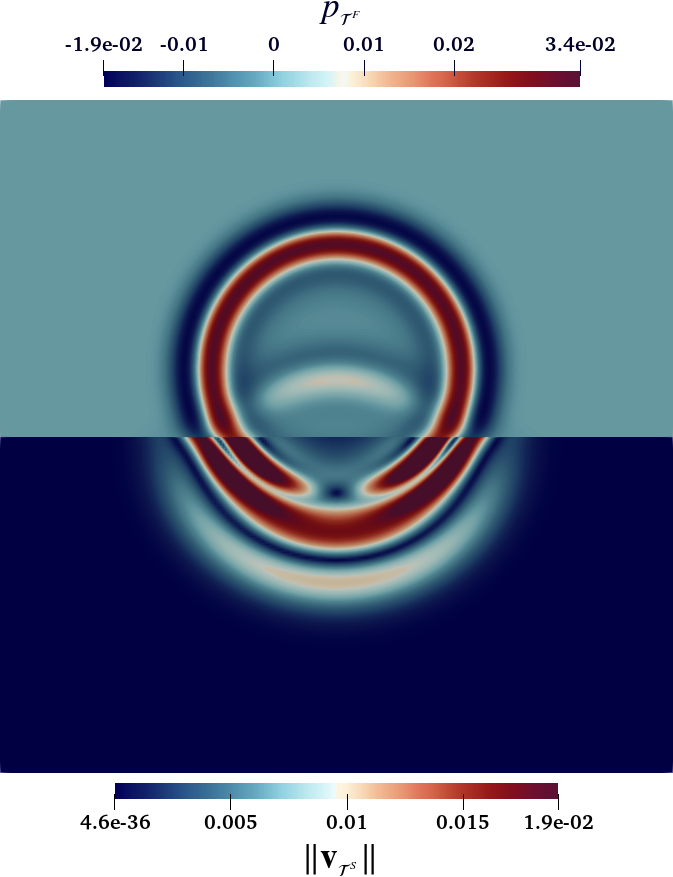}
\includegraphics[width=0.23\textwidth]{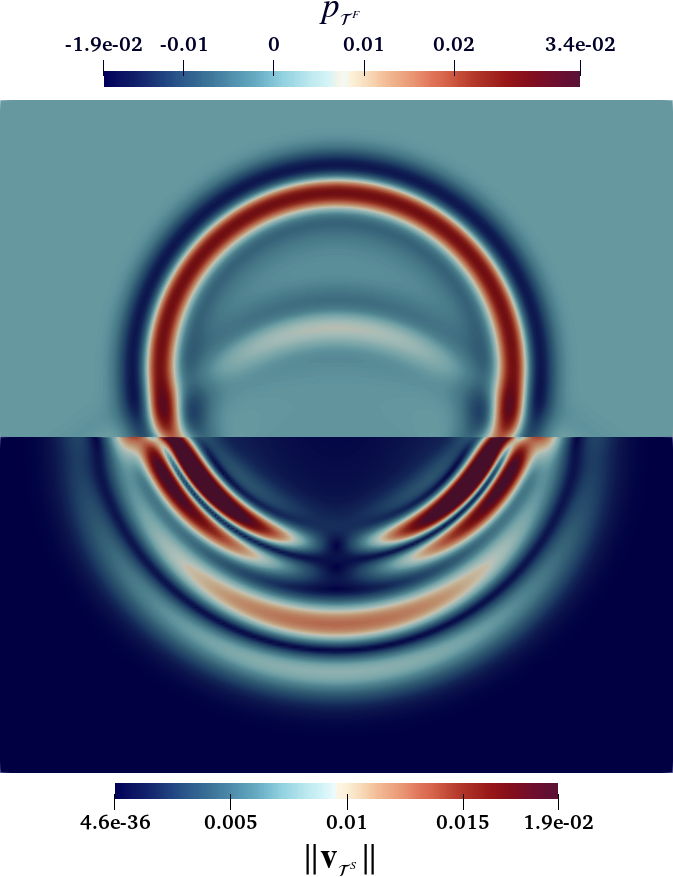}
\includegraphics[width=0.23\textwidth]{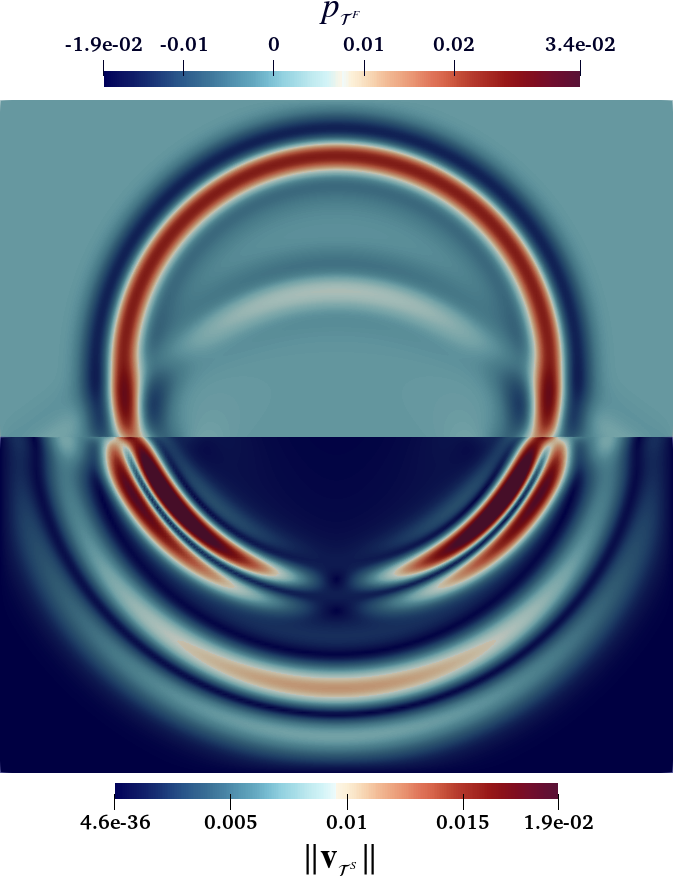}
\caption{Spatial distribution of the acoustic pressure (upper side) and the elastic velocity norm (lower side) at times $t \in \{0, 0.25, 0.27, 0.32\}$ predicted by SDIRK$(3,4)$ scheme with mixed-order setting, $\cal{O}(\frac{1}{h})$-stabilization, $k=1$, $\ell = 7$, and $n=9$.}
\label{snapshot_academic_pulse}
\end{figure}
\fi

A comparison to a semi-analytical solution provided by the open source software Gar6more (\url{https://gitlab.inria.fr/jdiaz/gar6more2d}) is performed in \hyperref[Gar6more_comparison_academic]{\Cref{Gar6more_comparison_academic}}. For two-dimensional infinite or semi-infinite domains, this code computes the analytical solution of elasto-acoustic waves propagating in homogeneous or heterogeneous media. \cred{In particular, the software analytically convolves the Green function of the problem  with the source function. The solution is called semi-analytical since the convolution is done by a numerical integration.} 
In \hyperref[Gar6more_comparison_academic]{\Cref{Gar6more_comparison_academic}}, we report the solution for times $t \in [0,0.25]$ at the two sensors $\cal{S}^{\sc{f}} := (-0.15,0.1)$ and $\cal{S}^{\sc{s}} := (-0.15,-0.1)$ with the material properties defined in (\ref{academic_properties}), for two rather coarse meshes: $\ell = 4$ on the left column and $\ell = 5$ on the right column. \hyperref[Gar6more_comparison_academic]{\Cref{Gar6more_comparison_academic}} shows that, for a moderate property contrast, even on the rather coarse mesh corresponding to $\ell=4$, we can approximate the analytical solution with high accuracy by increasing the polynomial order. Altogether, only the case $k=1$ on the coarse mesh does not allow to obtain an accurate representation of the solution.
\ifHAL
\renewcommand{\a}{0.4}
\else
\renewcommand{\a}{0.36}
\fi
\begin{figure}[!htb]
\centering
\includegraphics[width=\a\textwidth]{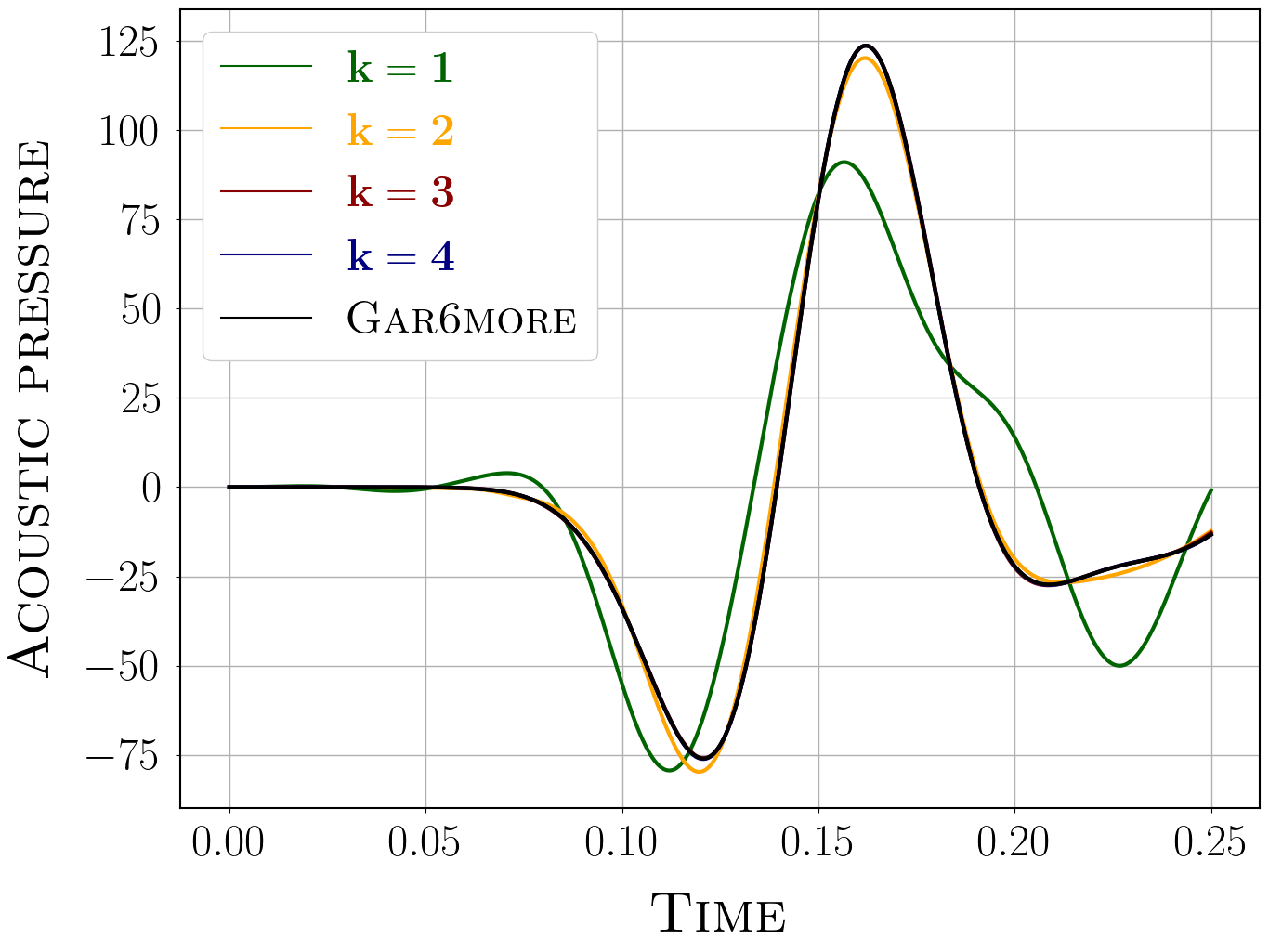}
\hspace{1.5cm}
\includegraphics[width=\a\textwidth]{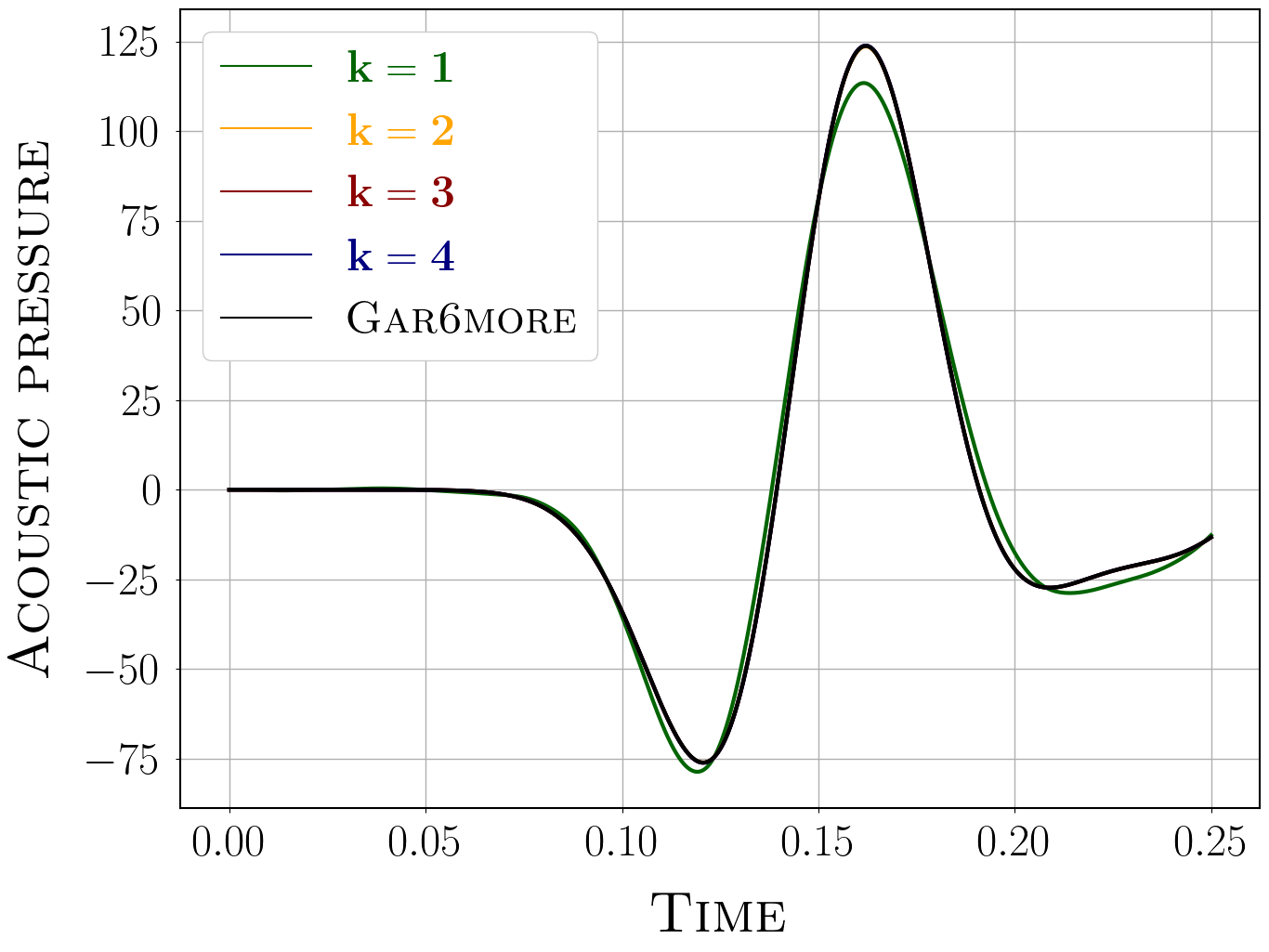}\\[0.25cm]
\includegraphics[width=\a\textwidth]{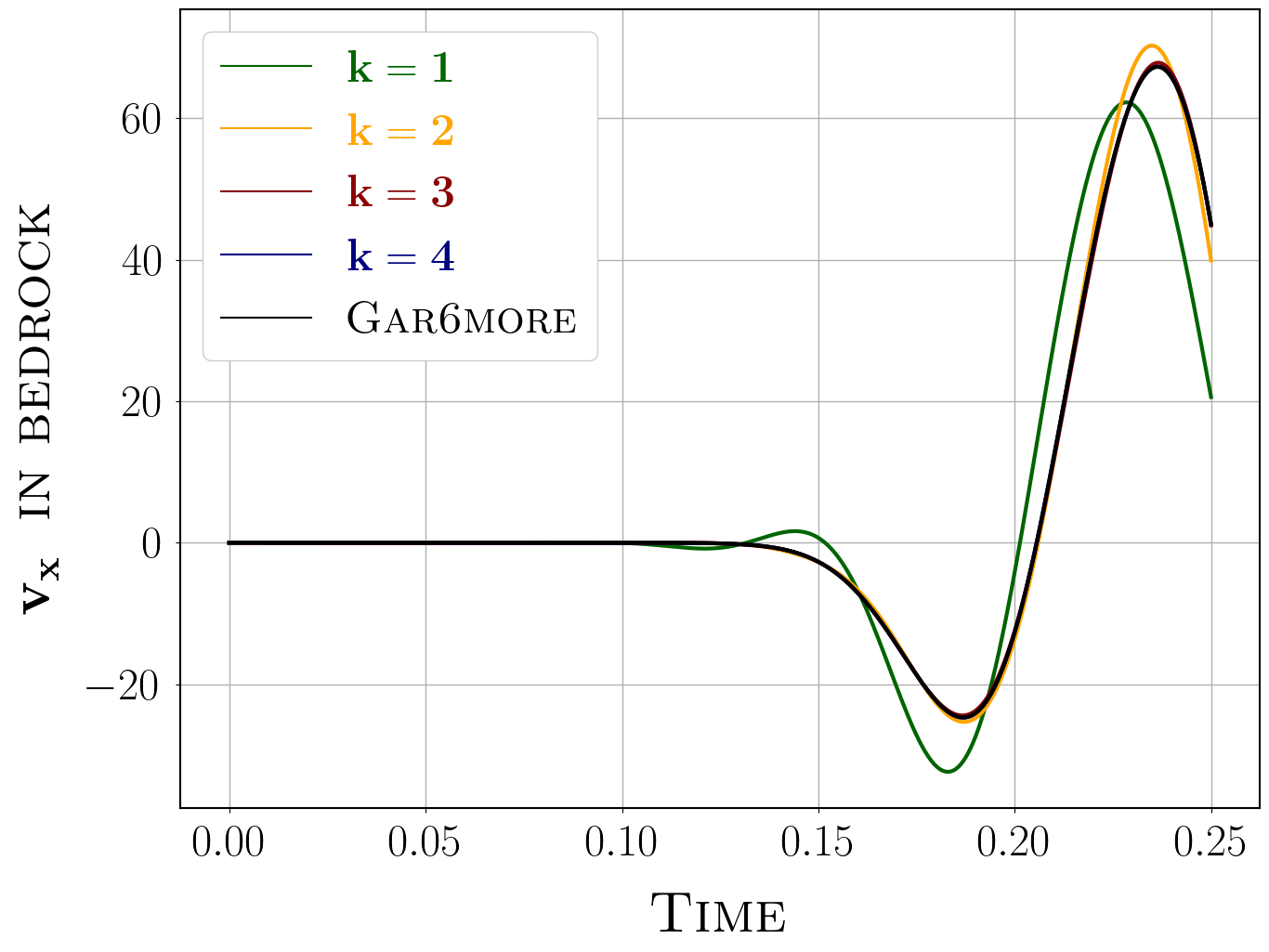}
\hspace{1.5cm}
\includegraphics[width=\a\textwidth]{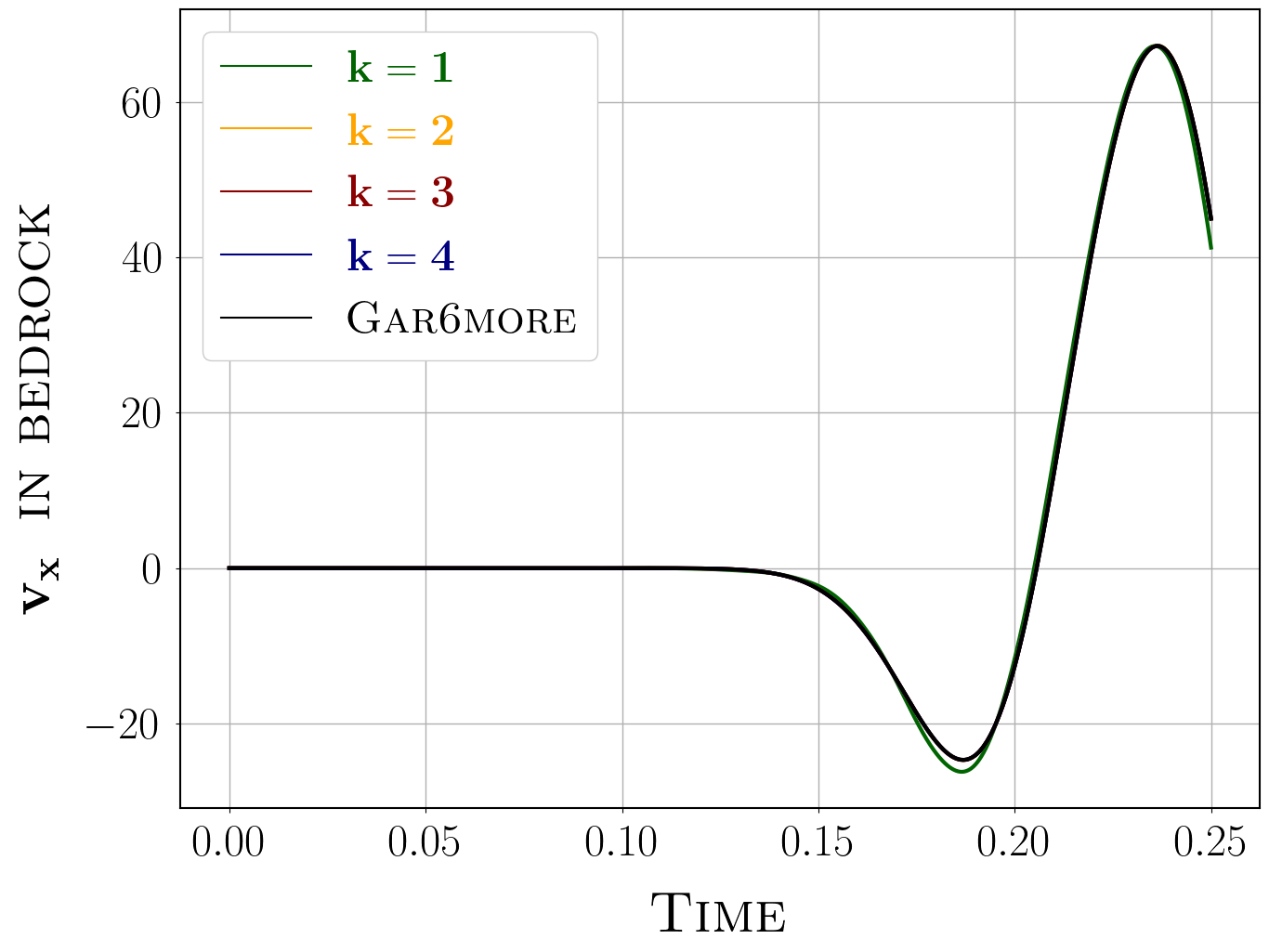}\\[0.25cm]
\includegraphics[width=\a\textwidth]{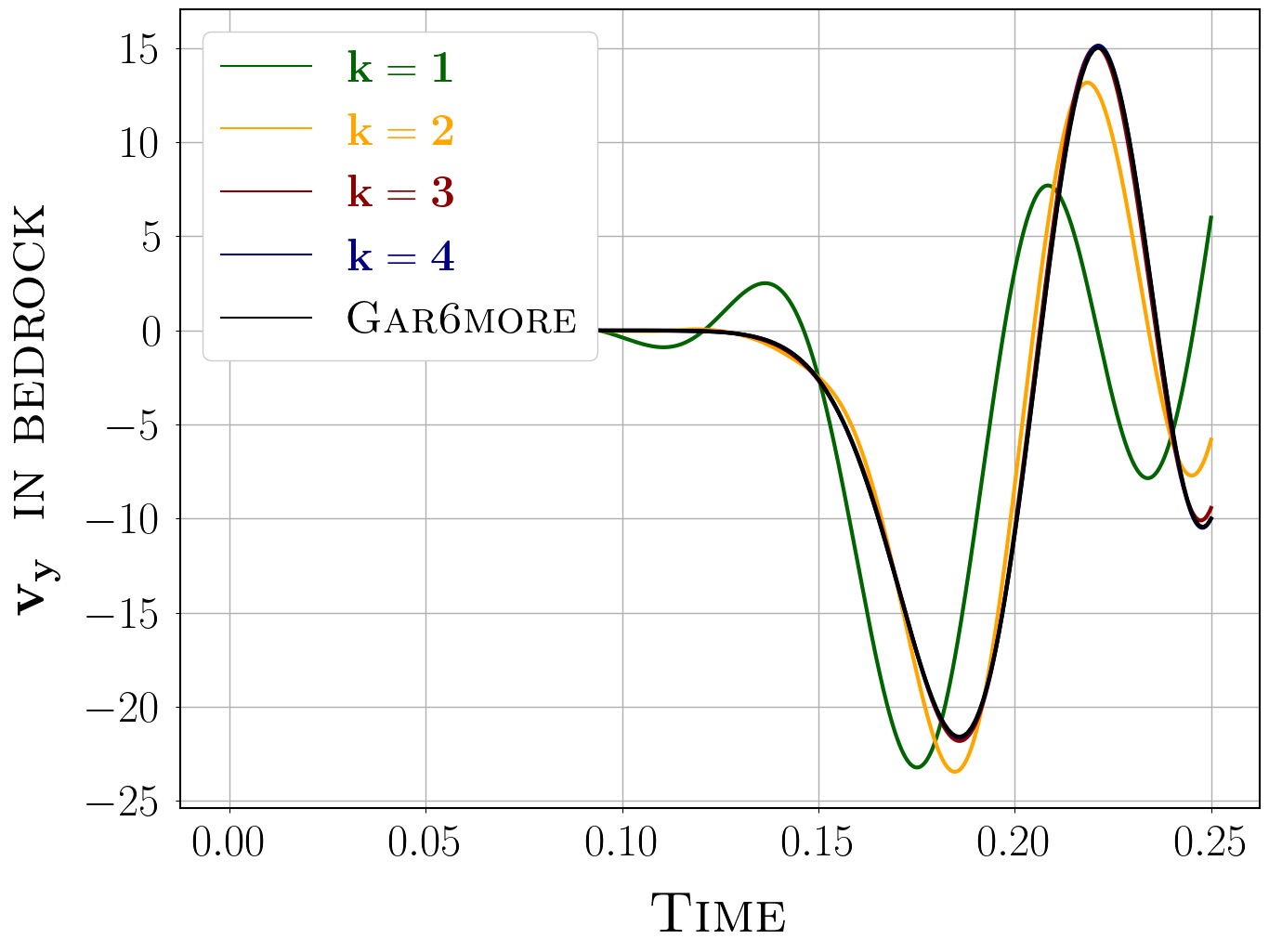}
\hspace{1.5cm}
\includegraphics[width=\a\textwidth]{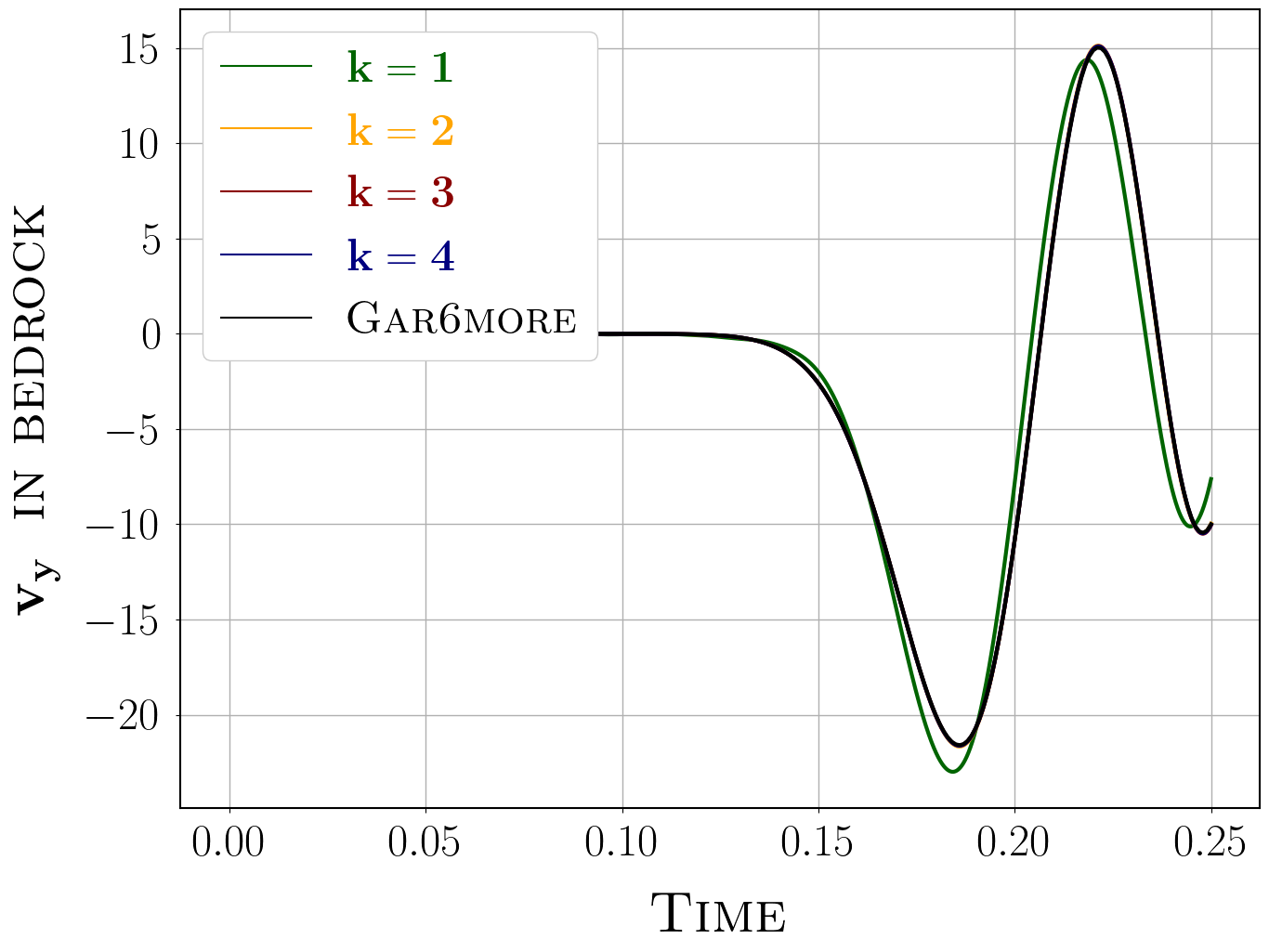}
\caption{Ricker wavelet with homogeneous material properties (see (\ref{academic_properties})). Comparison of the solution over time with the semi-analytical solution at sensors $\cal{S}^{\sc{f}}$ ($1^{st}$ row), and $\cal{S}^{\sc{s}}$ ($2^{nd}$ and $3^{rd}$ rows) for $n = 9$ and $\ell = 4$ (left column) and $\ell = 5$ (right column).}
\label{Gar6more_comparison_academic}
\end{figure}

\ifHAL
\FloatBarrier
\else
\fi
\subsubsection{Realistic (contrasted) setting}

We now investigate a test case with a strong property contrast focusing on two cases corresponding to granite and water (\ref{granite_eau}) and granite and air (\ref{granite_air}):
\begin{subequations}
\begin{alignat}{5}
\label{granite_eau}
& \rho^{\sc{s}} = 2.6 \rho^{\sc{f}} = 1.3, & \qquad & c_{\sc{p}}^{\sc{s}} = 4 c_{\sc{p}}^{\sc{f}} = 2, & \qquad & c_{\sc{s}}^{\sc{s}} = 2 c_{\sc{p}}^{\sc{f}} = 1,\\ 
& \rho^{\sc{s}} = 2200 \rho^{\sc{f}} = 800, & \qquad & c_{\sc{p}}^{\sc{s}} = 17.5c_{\sc{p}}^{\sc{f}} = 6.36, & \qquad & c_{\sc{s}}^{\sc{s}} = 9 c_{\sc{p}}^{\sc{f}} = 3.27.
\label{granite_air}
\end{alignat} 
\end{subequations}
We report the solution for times $t \in [0,0.5]$ at the two sensors $\cal{S}^{\sc{f}} := (-0.05,0.1)$ and $\cal{S}^{\sc{s}} := (-0.05, -0.1)$. \hyperref[Gar6more_comparison_contrast_eau]{\Cref{Gar6more_comparison_contrast_eau}} reports the results for case (\ref{granite_eau}) and \hyperref[Gar6more_comparison_contrast_air]{\Cref{Gar6more_comparison_contrast_air}} those for (\ref{granite_air}) on two meshes: $\ell = 5$ on the left column and $\ell = 6$ on the right column (this corresponds to one more level of refinement than for the low-contrast case). 
\ifHAL
\renewcommand{\a}{0.4}
\else
\renewcommand{\a}{0.36}
\fi
\begin{figure}[!htb]
\centering
\includegraphics[width=\a\textwidth]{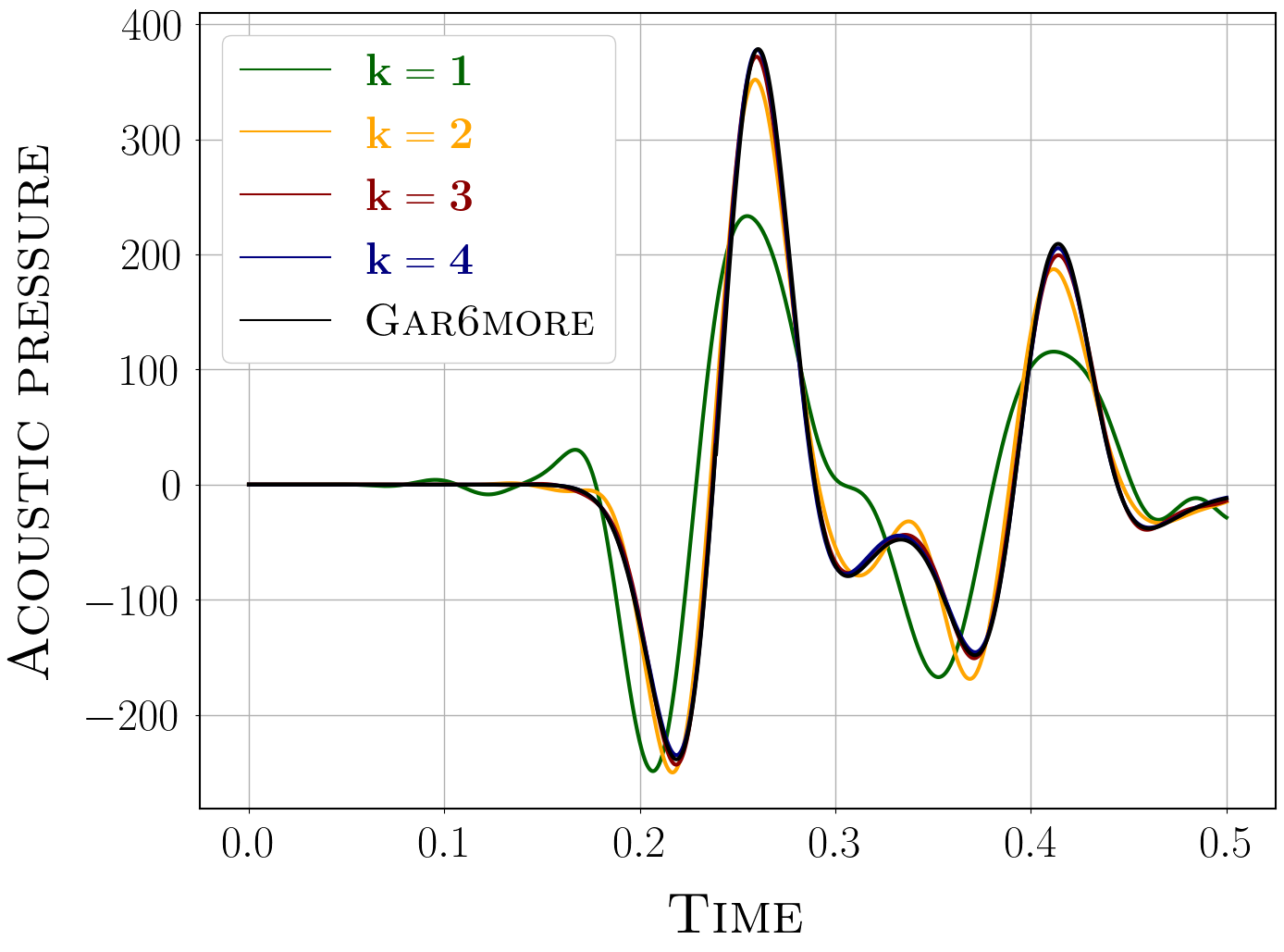}
\hspace{1.5cm}
\includegraphics[width=\a\textwidth]{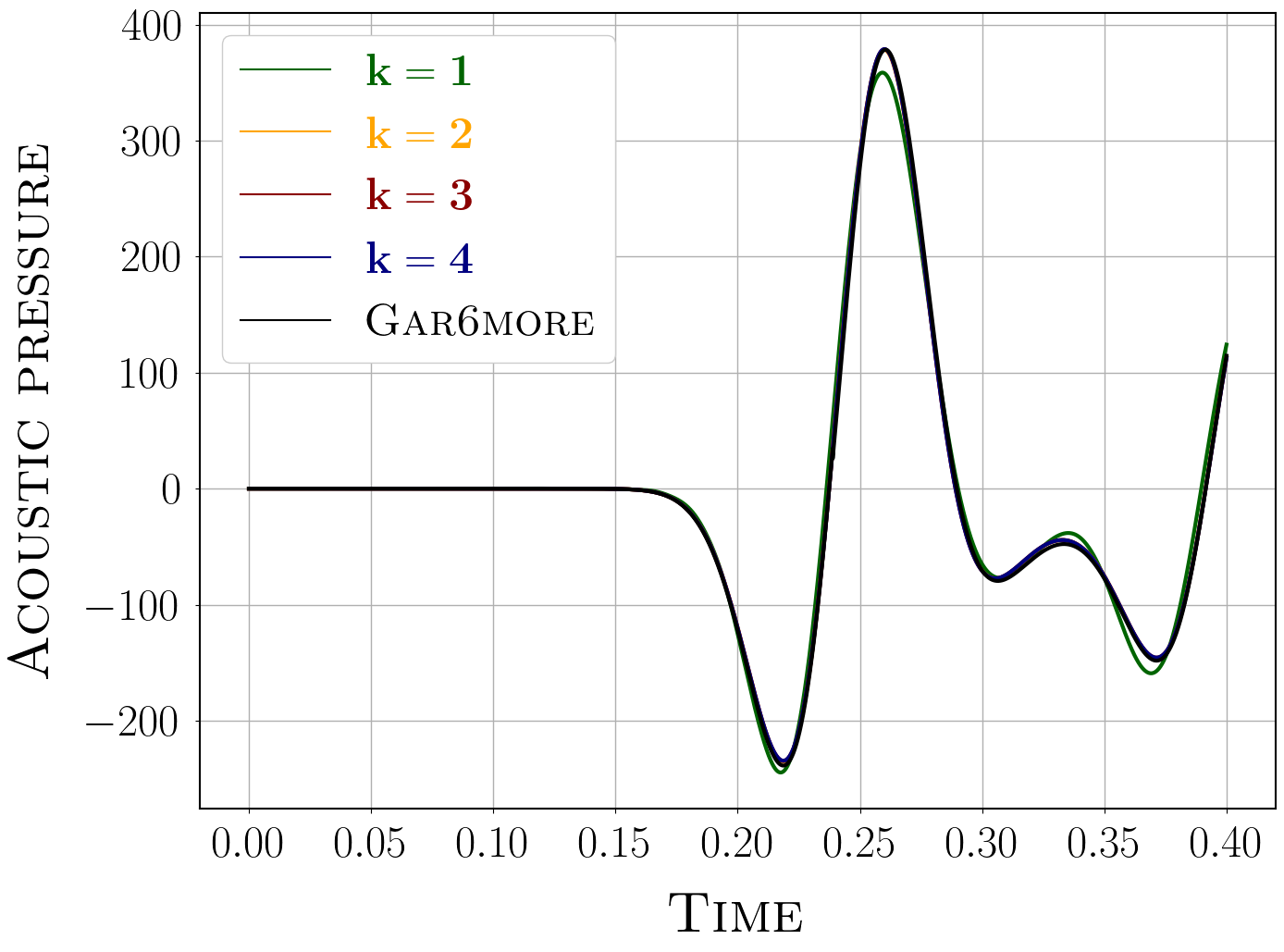}\\[0.25cm]
\includegraphics[width=\a\textwidth]{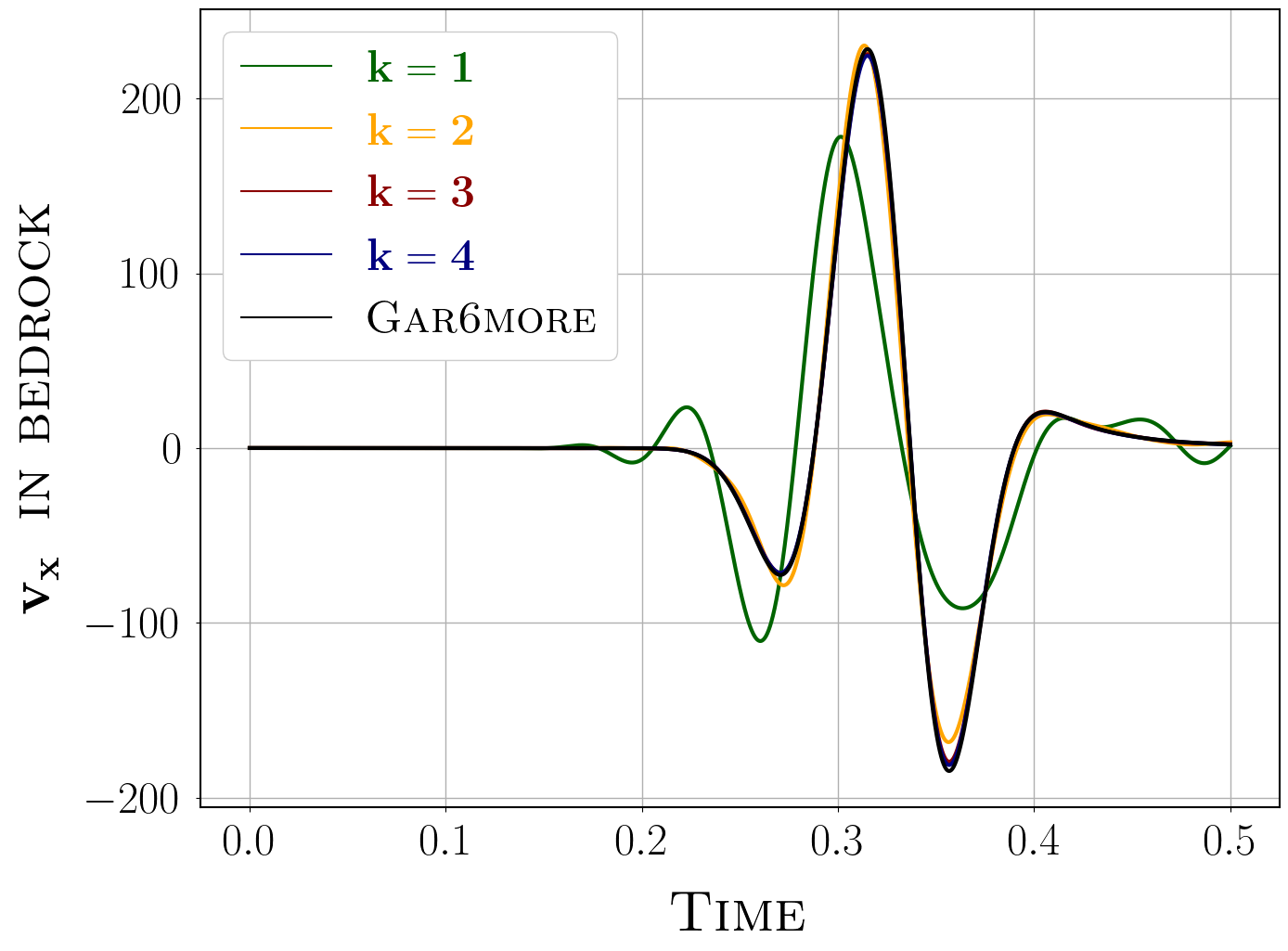}
\hspace{1.5cm}
\includegraphics[width=\a\textwidth]{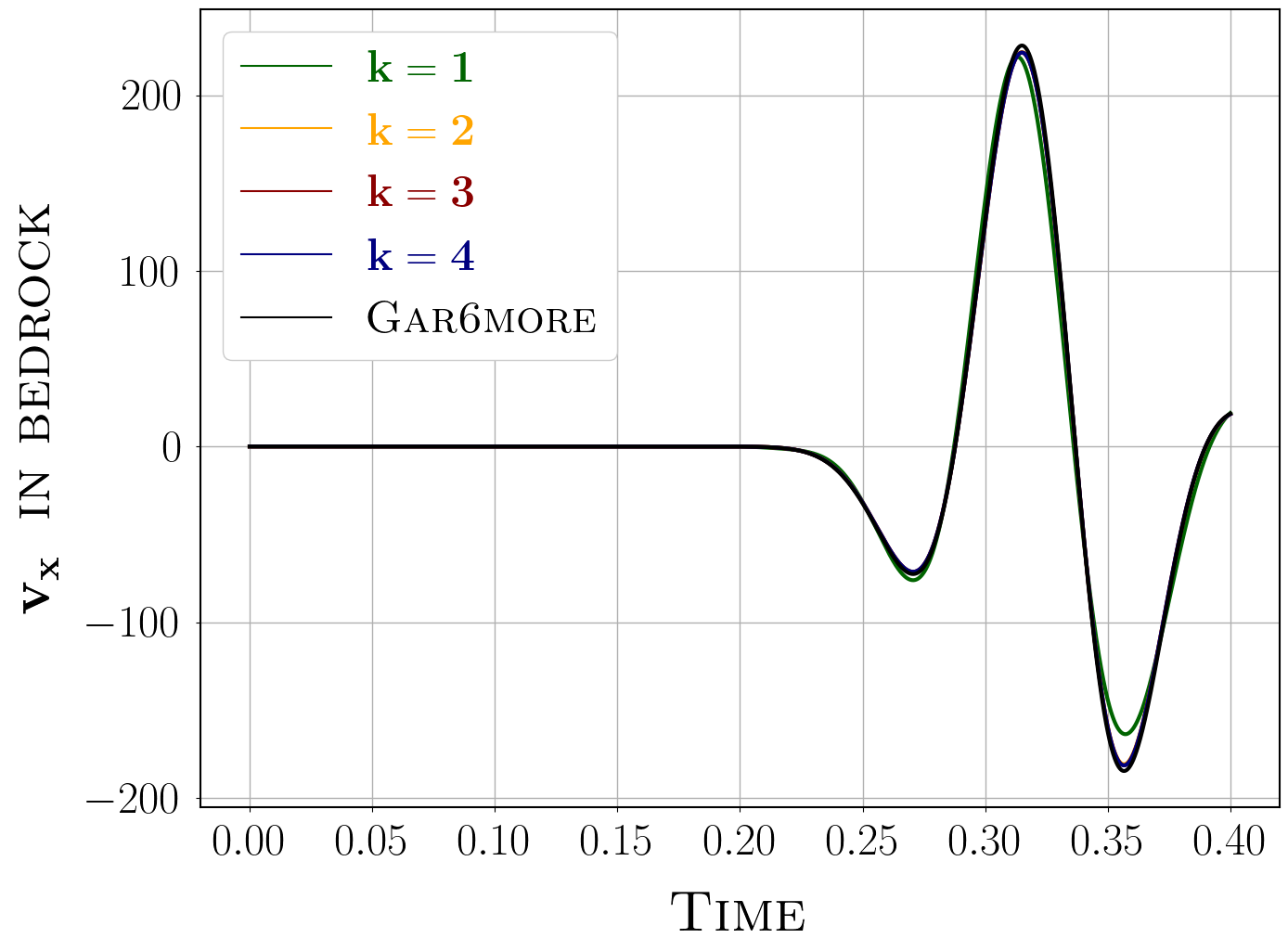}\\[0.25cm]
\includegraphics[width=\a\textwidth]{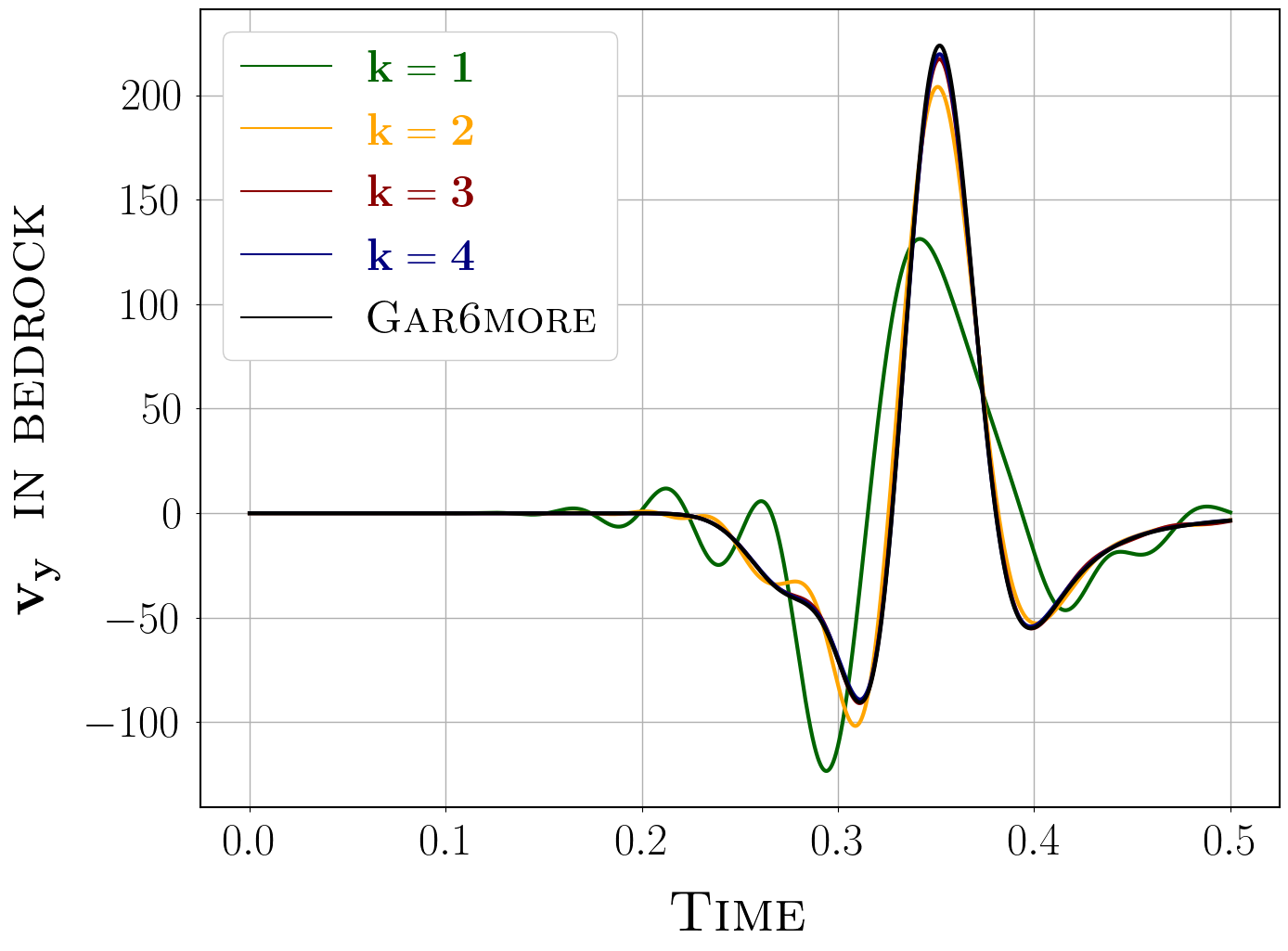}
\hspace{1.5cm}
\includegraphics[width=\a\textwidth]{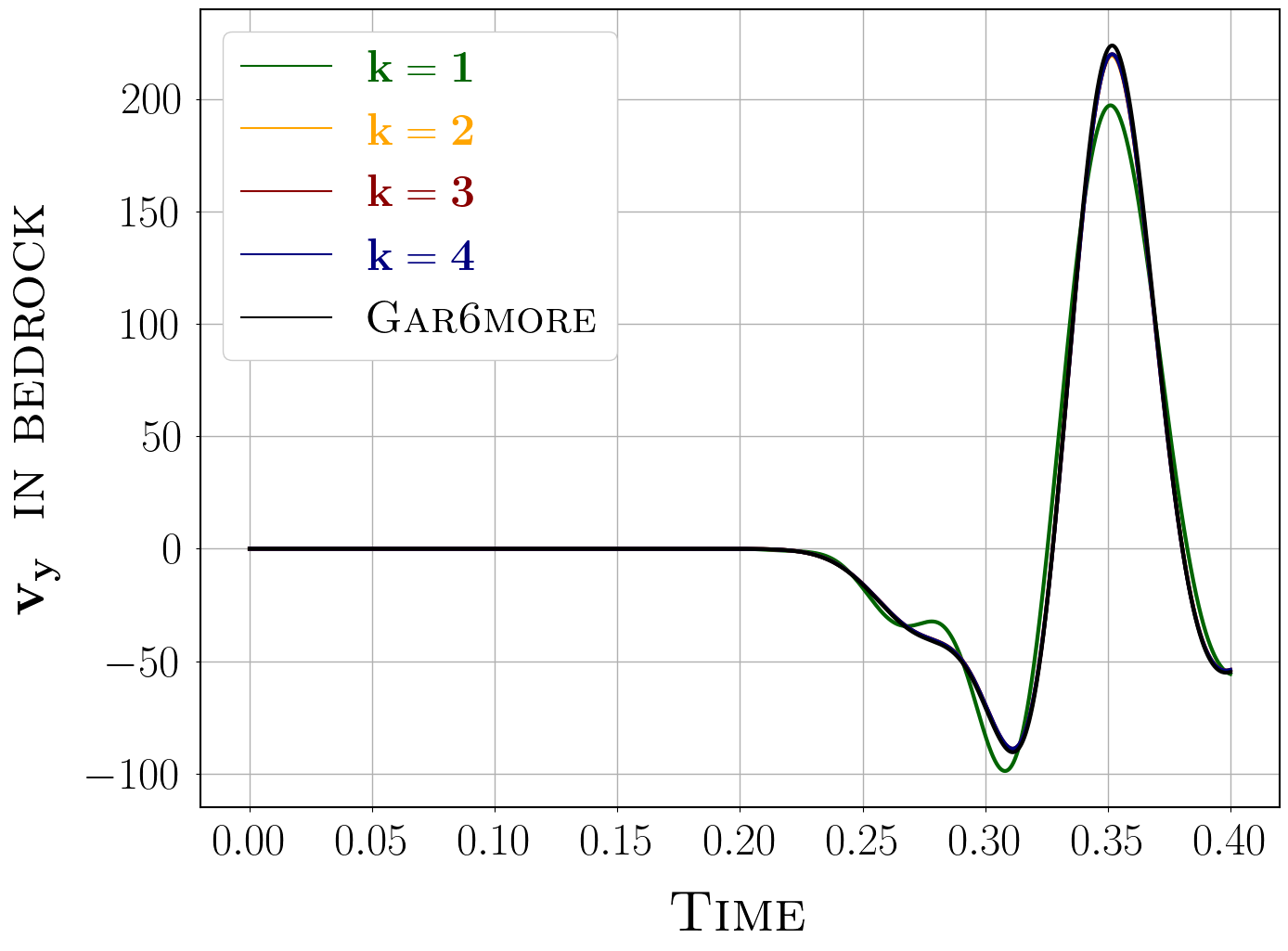}
\caption{Ricker wavelet with granite-water contrast (see (\ref{granite_eau})). Comparison of the solution over time with the semi-analytical solution at sensors $\cal{S}^{\sc{f}}$ ($1^{st}$ row) and $\cal{S}^{\sc{s}}$ ($2^{nd}$ and $3^{rd}$ rows) for $n = 9$ and $\ell = 5$ (left column) and $\ell = 6$ (right column).}
\label{Gar6more_comparison_contrast_eau}
\end{figure}
\ifHAL
\renewcommand{\a}{0.4}
\else
\renewcommand{\a}{0.36}
\fi
\begin{figure}[!htb]
\centering
\includegraphics[width=\a\textwidth]{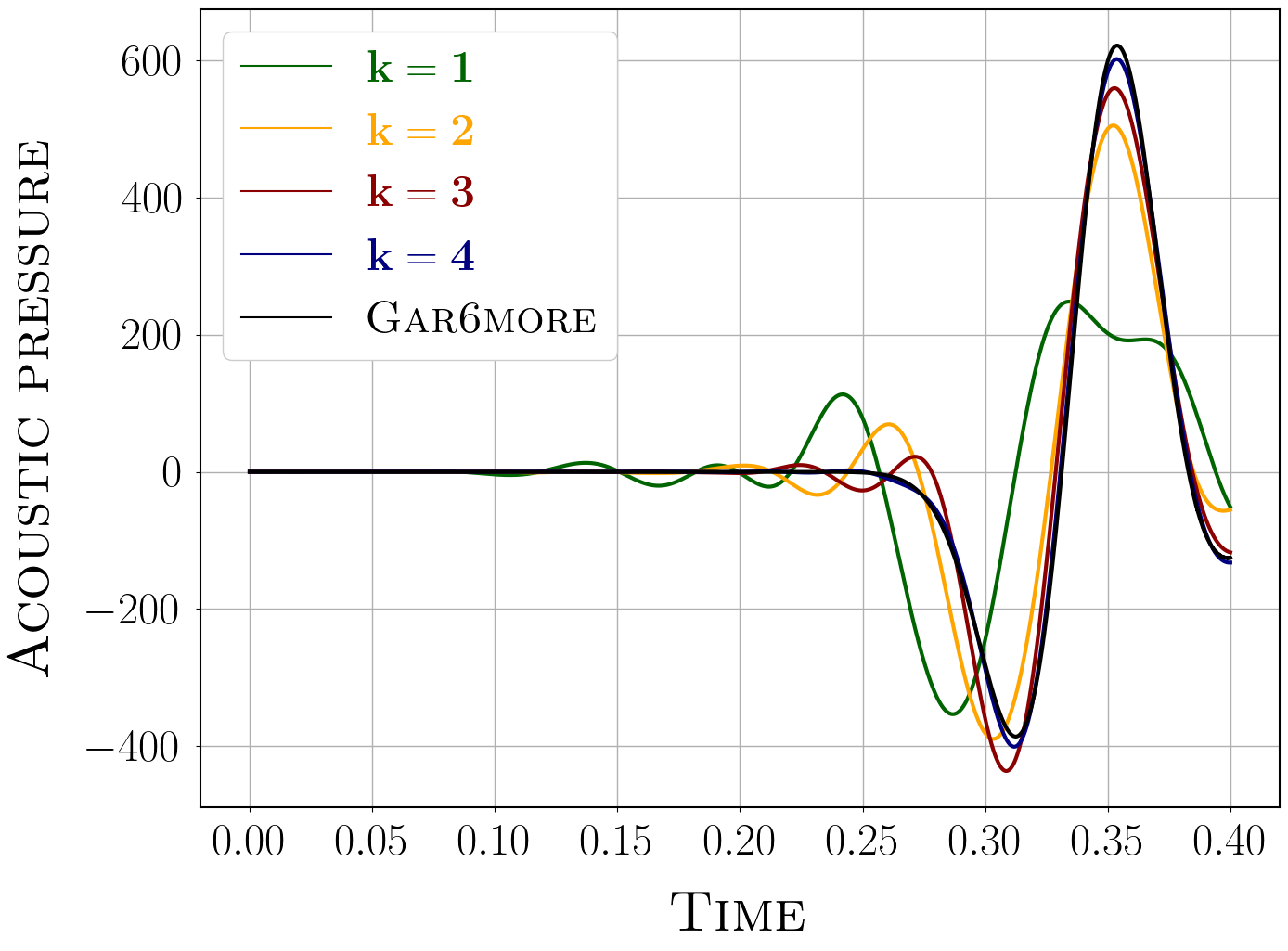}
\hspace{1.5cm}
\includegraphics[width=\a\textwidth]{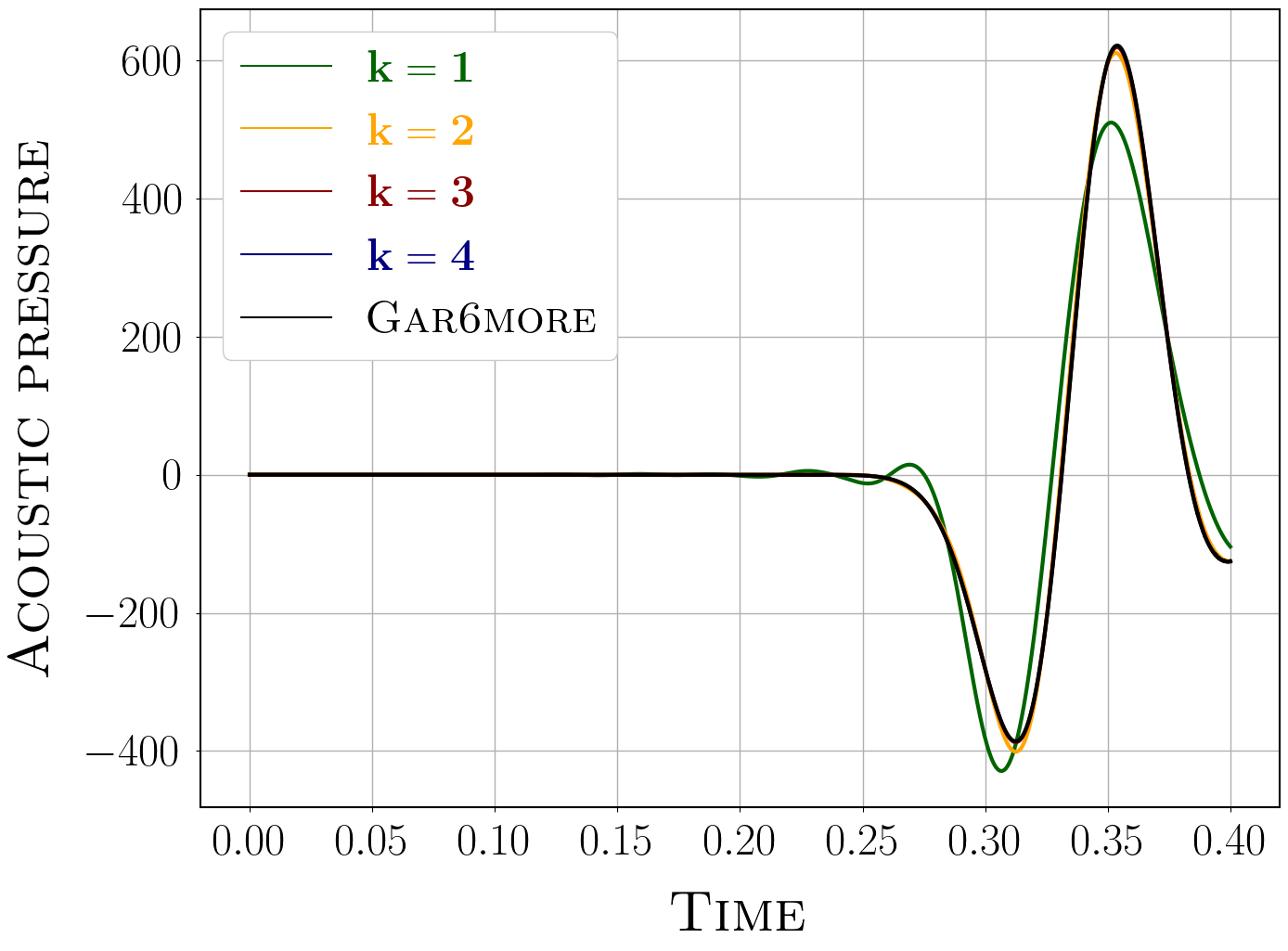}\\[0.25cm]
\includegraphics[width=\a\textwidth]{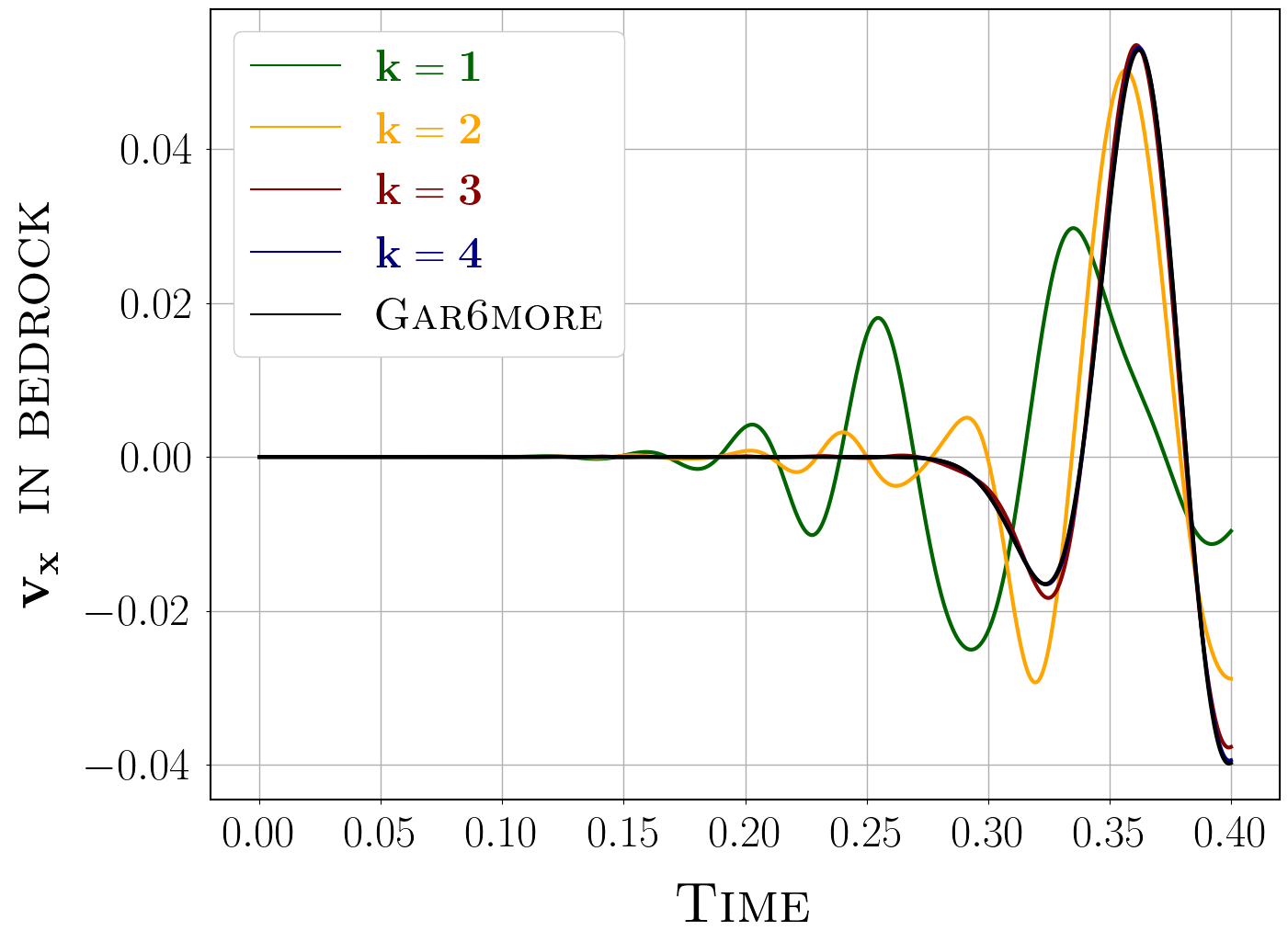}
\hspace{1.5cm}
\includegraphics[width=\a\textwidth]{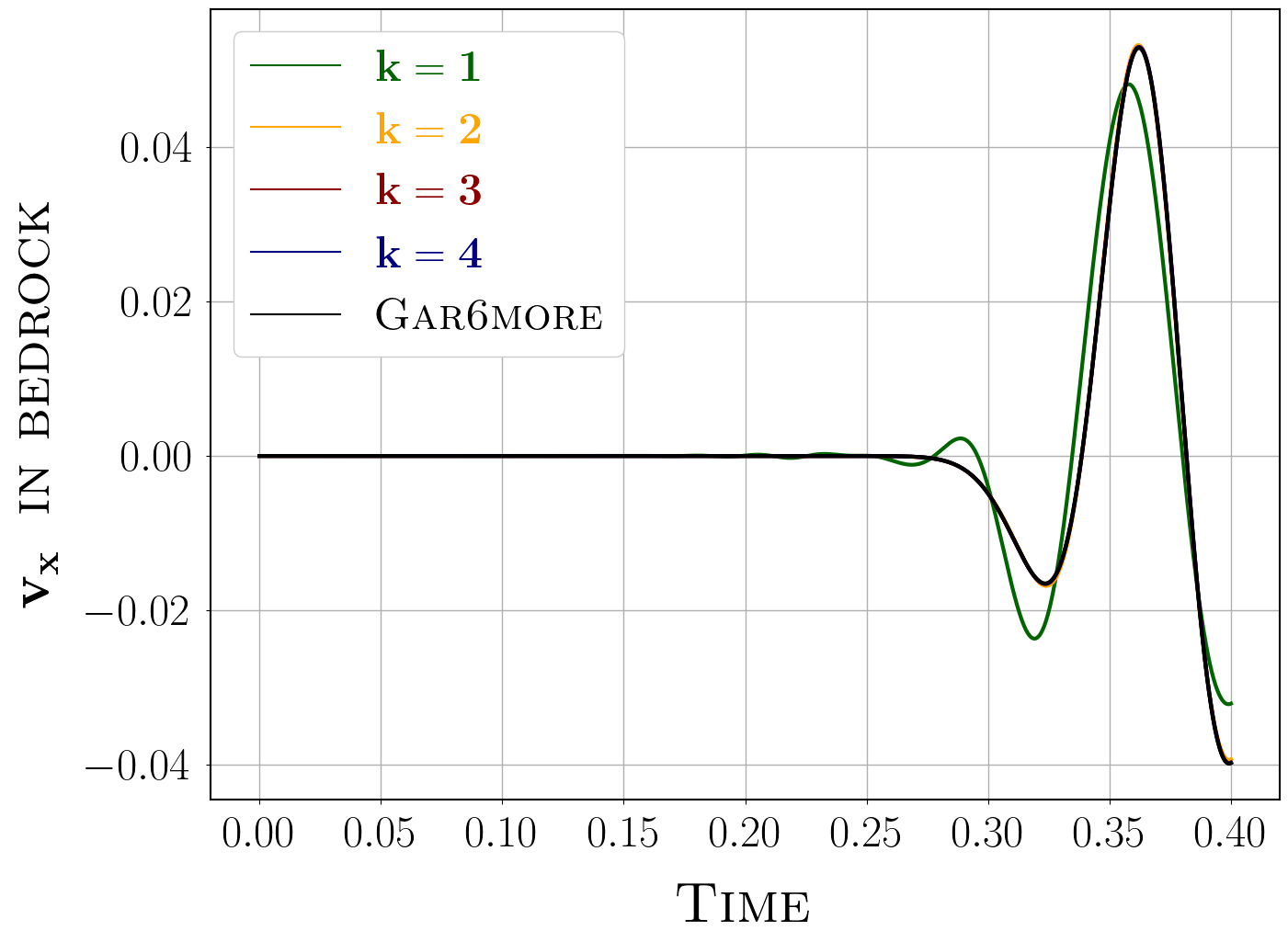}\\[0.25cm]
\includegraphics[width=\a\textwidth]{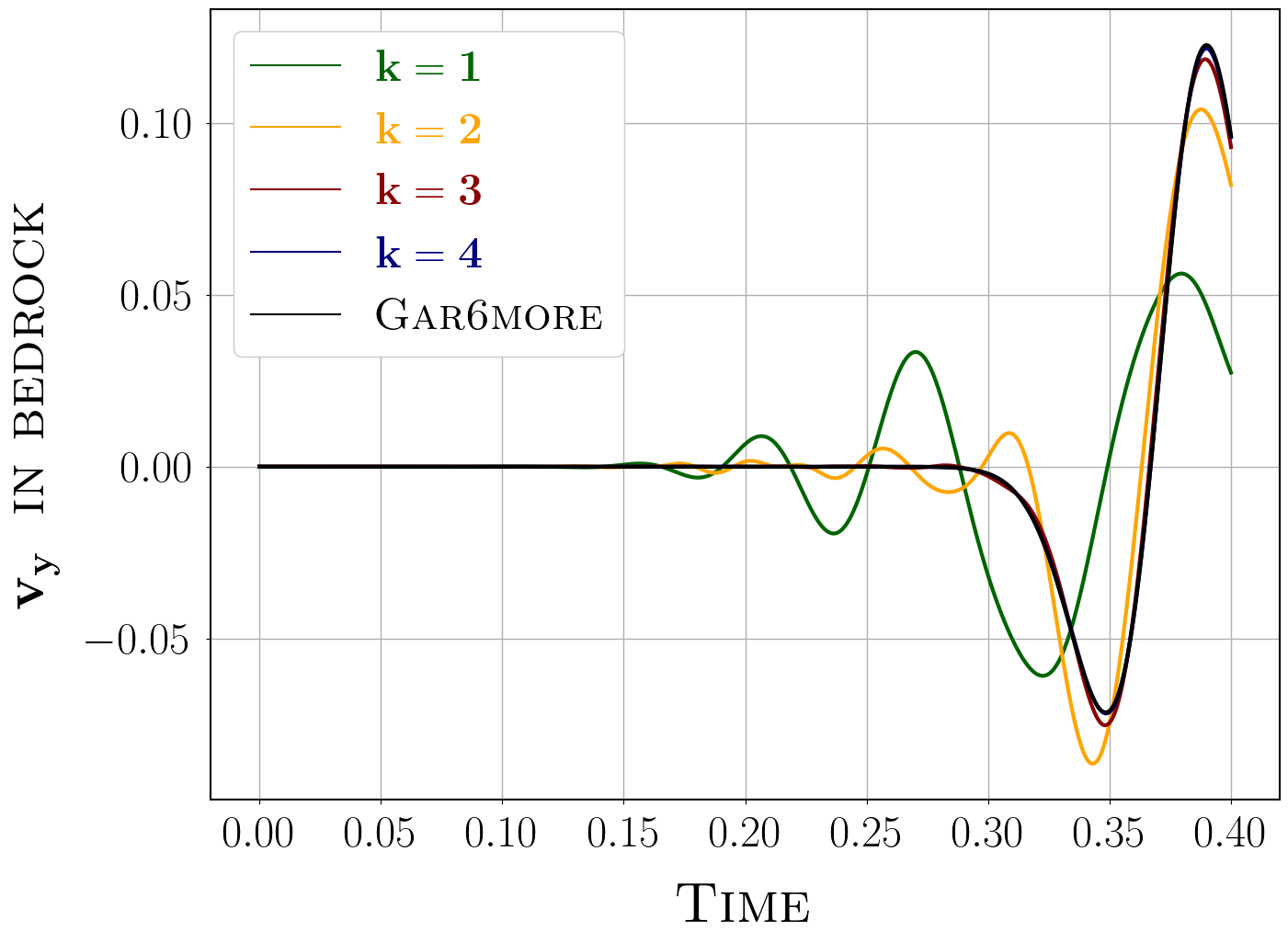}
\hspace{1.5cm}
\includegraphics[width=\a\textwidth]{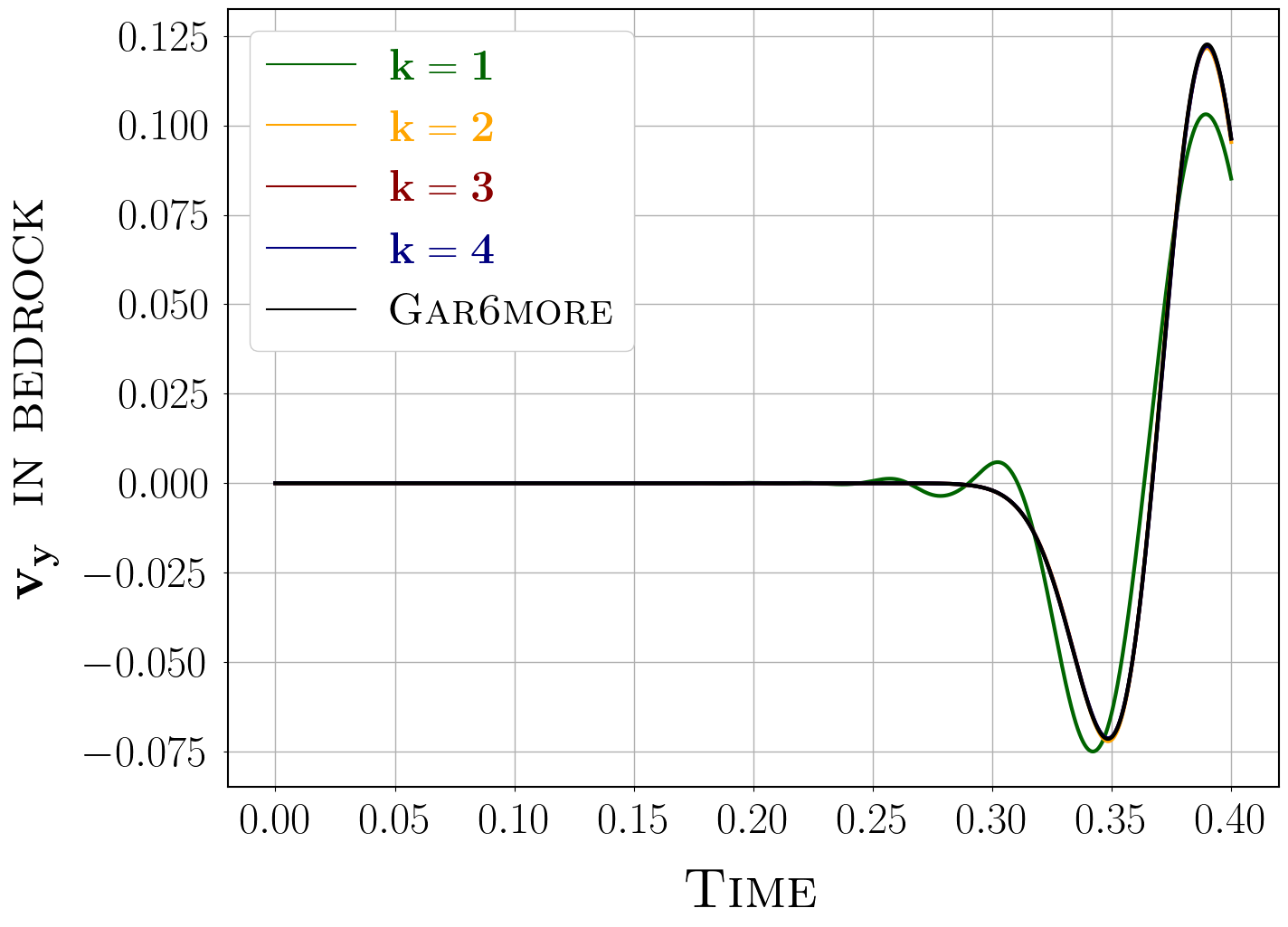}
\caption{Ricker wavelet with granite-air contrast (see (\ref{granite_air})). Comparison of the solution over time with the semi-analytical solution at sensors $\cal{S}^{\sc{f}}$ ($1^{st}$ row) and $\cal{S}^{\sc{s}}$ ($2^{nd}$ and $3^{rd}$ rows) for $n = 9$ and $\ell = 5$ (left column) and $\ell = 6$ (right column).}
\label{Gar6more_comparison_contrast_air}
\end{figure}

The results displayed in Figures \ref{Gar6more_comparison_contrast_eau} and \ref{Gar6more_comparison_contrast_air} exhibit the same characteristics as in the previous case, thus demonstrating that the present scheme effectively handles strong property contrasts by accurately describing the solution (excluding the case $k=1$ on the coarse mesh). Finally, as expected, the greater the contrast, the larger the amplitude difference between the acoustic and elastic signals. Indeed, while the signals in \hyperref[Gar6more_comparison_academic]{\Cref{Gar6more_comparison_academic}} have similar amplitudes, \hyperref[Gar6more_comparison_contrast_eau]{\Cref{Gar6more_comparison_contrast_eau}} and \hyperref[Gar6more_comparison_contrast_air]{\Cref{Gar6more_comparison_contrast_air}} show that as the contrast increases, the signals in the receiving medium become weaker due to strong wave reflections at the interface.

In Figures \ref{energy_academic_pulse} and \ref{energy_water_granit_pulse}, we focus on the time evolution of the discrete energy of the global system. In \hyperref[energy_academic_pulse]{\Cref{energy_academic_pulse}}, the left panel shows the energy repartition as a function of time for $t \in [0,10]$, polynomial order $k=3$ space refinement $\ell = 6$, and time refinement $n = 9$. In the right panel, we study the relative energy loss as \cred{the ratio of energy at times $t \in [0,1]$ with the energy at initial time for different} polynomial degrees $k \in \{1,2,3\}$ and space refinement \cred{levels} $\ell \in \{4,5,6\}$. In the left panel, we observed that, in the absence of contrast, the energy initially concentrated in the acoustic subdomain is partially transferred to the elastic subdomain. On longer time scales, energy oscillates with a moderate amplitude around an equal distribution between the two subdomains. In the right panel, we can see that increasing the polynomial order and/or the space refinement level significantly reduces energy dissipation. For all reasonable discretizations, the energy dissipation stays below $1\%$ for all $t \in [0,1]$. In the left panel of \hyperref[energy_water_granit_pulse]{\Cref{energy_water_granit_pulse}}, we show the energy repartition for the contrasted case (\ref{granite_eau}) with $\ell=6$ and $k=3$. We observe that, owing to the increase in property contrast, significantly less energy is transmitted from the acoustic medium to the elastic medium, which is consistent with the decrease in signal amplitude observed in Figures \ref{Gar6more_comparison_contrast_eau} and \ref{Gar6more_comparison_contrast_air}. The same test cases were conducted with the initial pulse located in the elastic medium, and a similar distribution of the energy was observed between the emitting medium and the receiving medium. For the sake of brevity, these results are not reported here.
\begin{figure}[!htb]
\centering 
\includegraphics[height=0.4\textwidth]{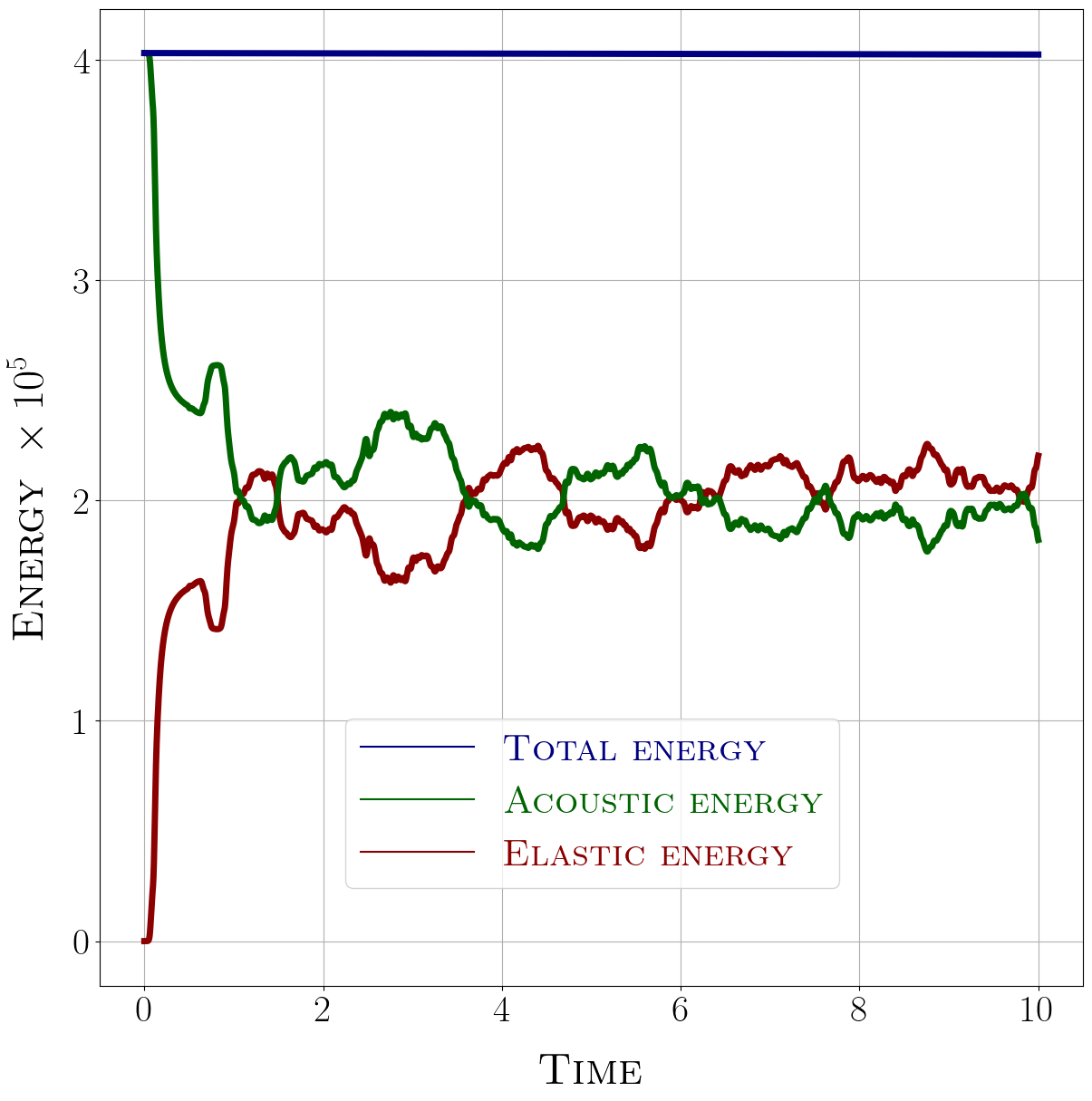}
\hspace{1cm}
\includegraphics[height=0.4\textwidth]{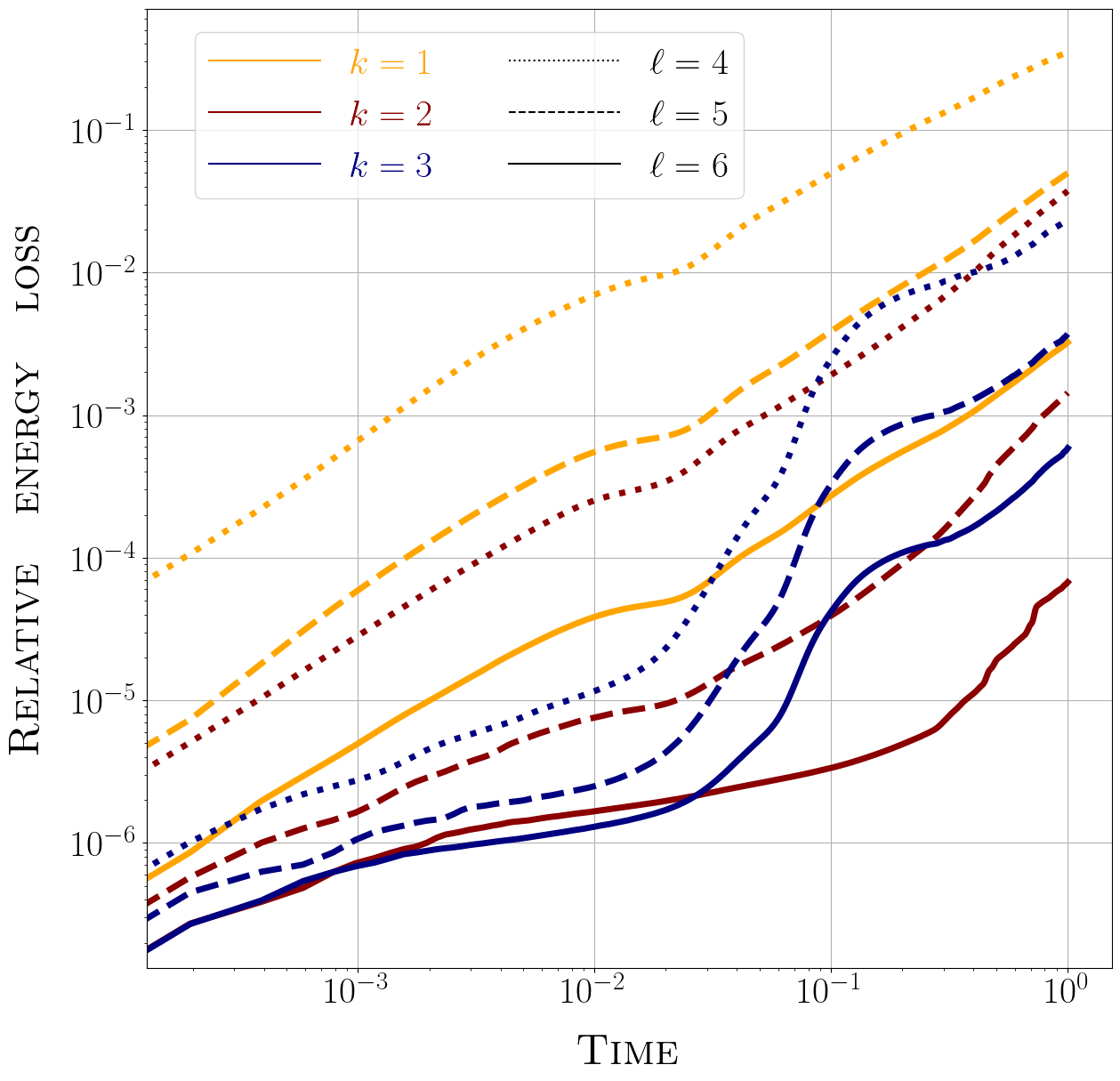}
\caption{SDIRK$(3,4)$ scheme with $n=9$. \textbf{Left:} Energy repartition as a function of the time for $k=3$ and $\ell = 6$. \textbf{Right:} Relative energy loss as function of the time for $k \in \{1,2,3\}$ and $\ell \in \{4,5,6\}$.}
\label{energy_academic_pulse}
\end{figure}
\begin{figure}[!htb]
\centering
\includegraphics[width=0.425\textwidth]{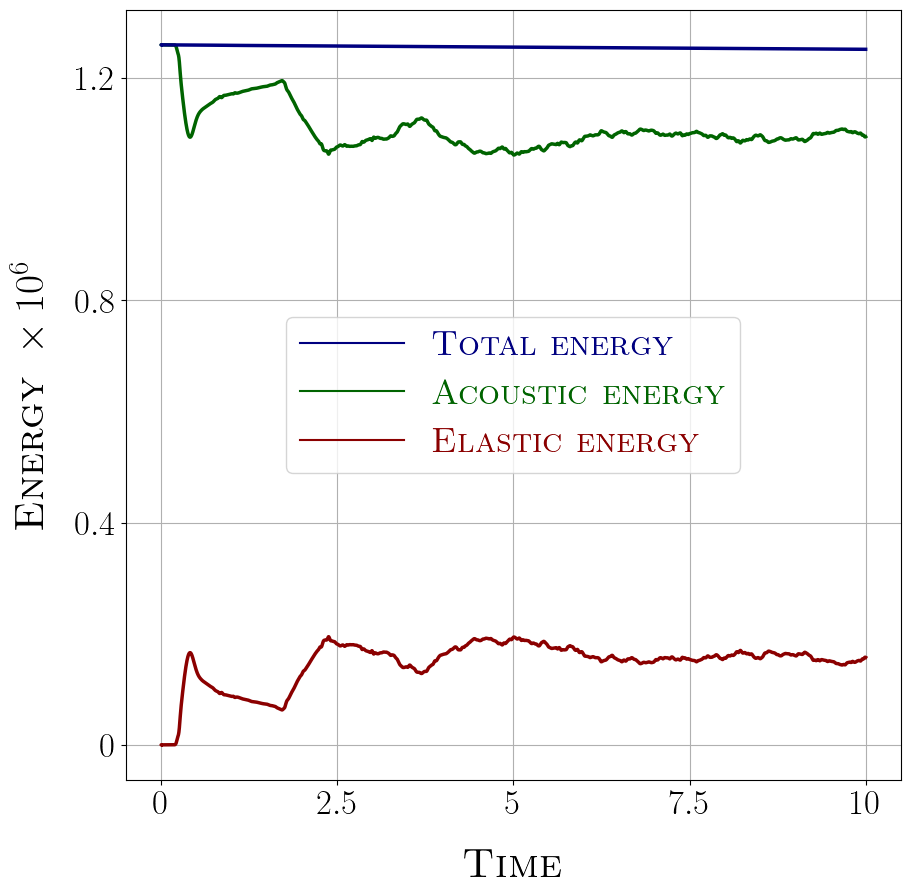}
\hspace{1cm}
\includegraphics[width=0.425\textwidth]{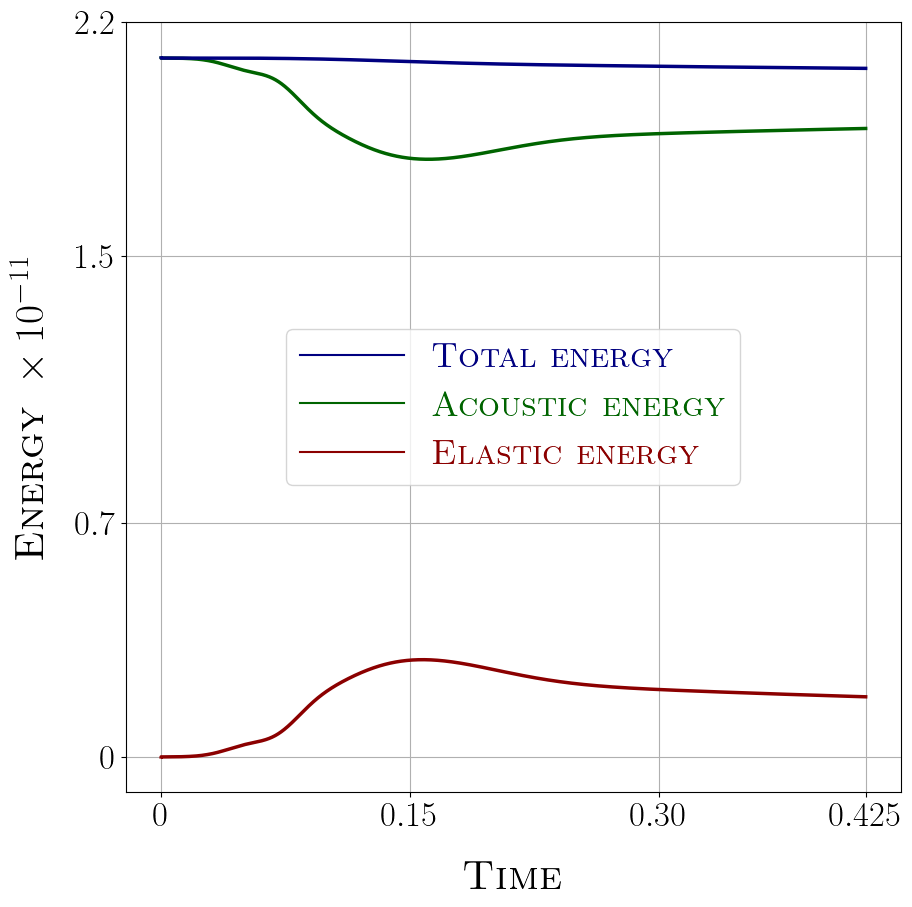}
\caption{Ricker wavelet for granite-water test case: Energy repartition as a function of the time predicted by SDIRK$(3,4)$ scheme. \textbf{Left:} test case (\ref{granite_eau}), $k=3$, $n=9$ and $\ell = 6$. \textbf{Right:} test case (\ref{real_granite_water}), $k=3$, $n=8$, $h_x=8.93~\rm{m}$ and $h_y=6.25~\rm{m}$.}
\label{energy_water_granit_pulse}  
\end{figure}

To conclude, we perform the granite-water test case (\ref{granite_eau}) with realistic values for the material properties and the geometry. We set 
\begin{equation}
\begin{alignedat}{2}
\rho^\sc{f} & := 1025~\rm{kg.m}^{-3}, \qquad & c_{\sc{p}}^{\sc{f}} & := 1500~\rm{m.s}^{-1}, \\
\rho^\sc{s} & := 2690~\rm{kg.m}^{-3}, \qquad & c_{\sc{p}}^\sc{s} & := 6000~\rm{m.s}^{-1}, \qquad c_{\sc{s}}^{\sc{s}} := 3000~\rm{m.s}^{-1},
\end{alignedat}
\label{real_granite_water}
\end{equation}
as well as $L:=5000~\rm{m}$, $H:=3500~\rm{m}$, $H_e:=2500~\rm{m}$ for the dimensions of the domain and $x_c:=0~\rm{m}$, $y_c:=500~\rm{m}$ for the center of the pulse in the acoustic subdomain, and a simulation time $T_{\mathrm{f}} := 0.425 ~\mathrm{s}$. 
\hyperref[snapshot_granit_pulse]{\Cref{snapshot_granit_pulse}} displays the two-dimensional distributions of the pressure in the acoustic region and of the \cred{Euclidean} velocity norm in the elastic region at times $t \in \{0.1275,0.3825\}$ predicted by the SDIRK$(3,4)$ scheme, with computational parameters $k=3$ and $n = 8$. A quadrangular mesh is considered with $h_x=8.93~\rm{m}$ and $h_y=6.25~\rm{m}$. The computational domain has been chosen sufficiently large in order to avoid the bouncing off the walls due to the homogeneous Dirichlet conditions and to allow the waves to develop. We can see that the simulation captures well the penetration of the wave into the elastic domain. The lateral conical wavefronts are accurately represented, along with the interface, Rayleigh-type waves, which are characterized by the constructive interferences between P-waves and polarized S-waves at the interface. The right panel of \hyperref[energy_water_granit_pulse]{\Cref{energy_water_granit_pulse}} shows the energy repartition related to this test case.
\begin{figure}[!htb]
\centering
\includegraphics[width=0.495\textwidth]{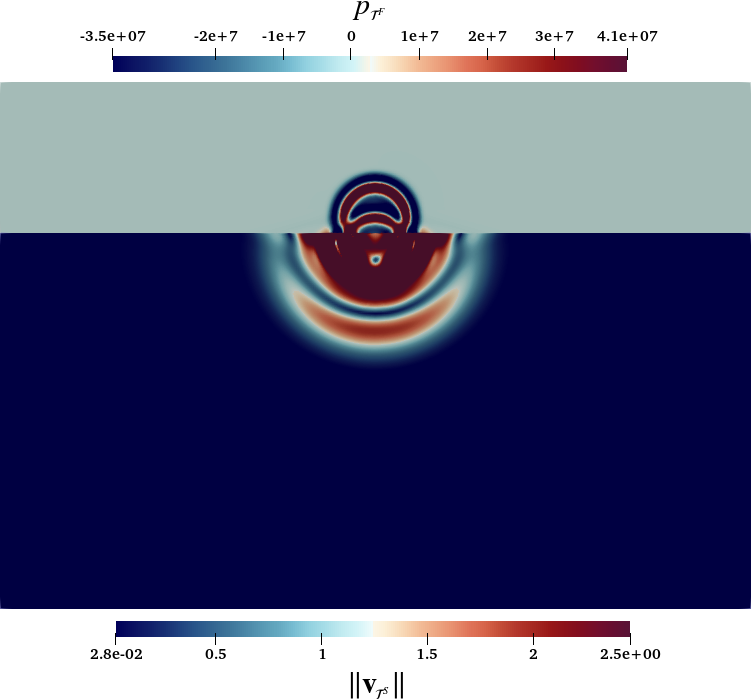}
\includegraphics[width=0.495\textwidth]{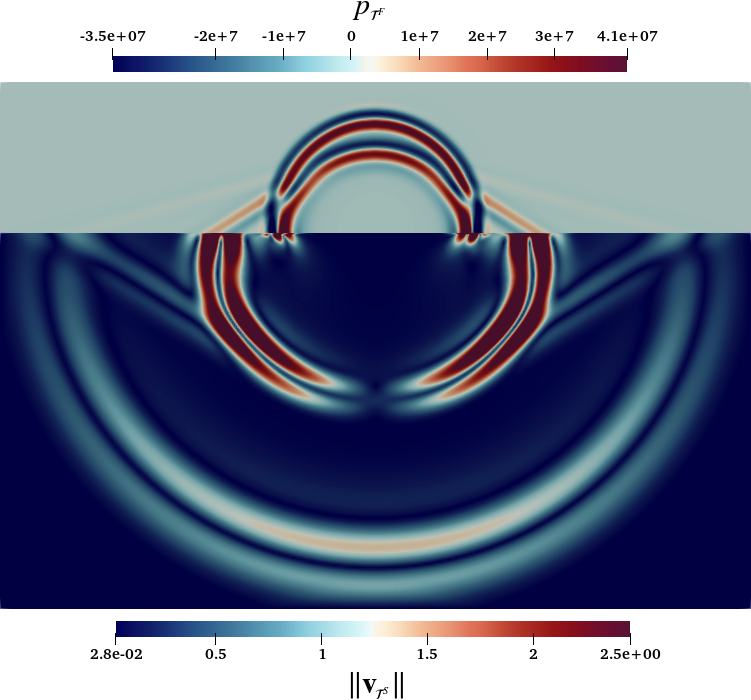}
\caption{Ricker wavelet for granite-water test case (\ref{real_granite_water}): Spatial distribution of the acoustic pressure (upper side) and the elastic velocity norm (lower side) \cred{at times} $t \in \{0.1275, 0.3825\}$ predicted by SDIRK$(3,4)$ with mixed-order setting, $\cal{O}(\frac{1}{h})$-stabilization, $k=3$, $n=8$, $h_x=8.93~\rm{m}$ and $h_y=6.25~\rm{m}$.}
\label{snapshot_granit_pulse}
\end{figure}

In conclusion, all these results highlight the robustness of the proposed scheme in accurately modeling wave propagation through media with significant density and wave velocity contrasts.

\cred{\section*{Conclusion}
In this paper, we studied the elasto-acoustic coupled wave problem in first-order form approximated using dG (for the dual variable) and HHO (for the primal variable) space discretizations including several variants for HHO: equal- or mixed-order and stabilization parameters $\mathcal{O}(1)$ or $\mathcal{O}(\frac{1}{h})$. We performed a space semi-discrete analysis based on the H$^+$ interpolation operator for the dual variable instead of the standard $L^2$ projection operator. This choice makes the time-continuous analysis simpler than in \cite{BDES_2021}. Compared to dG which stabilizes both primal and dual variables and achieves $\mathcal{O}(h^{k+\frac{1}{2}})$ convergence rates (see, e.g. \cite{BEF_2010}), we stabilize only the primal variable and achieve $O(h^{k+1})$ convergence rates. The spectral analysis study in Section~6.1 showcases the effect of different stabilization weights on the CFL condition.  In the numerical experiments of Section~6.3, we considered both implicit and explicit time-stepping schemes, and  confirmed the theoretical convergence rates on an academic test case. Moreover, we observed $\mathcal{O}(h^{k+2})$ convergence rates for the primal variable in the $L^2$-norm when using the mixed-order HHO method with $\mathcal{O}(\frac{1}{h})$-stabilization. In Section~6.4, we  presented a realistic Ricker wavelet test case with contrasted material properties and demonstrated that our scheme accurately captures wave propagation across interfaces. An interesting perspective for further work is to explore nonlinear coupled wave models. One of the
advantages of HHO with respect to dG in this context is that the integration of nonlinear behavior laws is
only required at the quadrature nodes in the cells, but not on the faces \cite{AEP_2018}. Another perspective for future work is to analyze the fully discrete scheme involving a DIRK time scheme combined with the HHO space discretization.}
\ifHAL
\addcontentsline{toc}{section}{References}
\else
\fi

\ifHAL
\FloatBarrier
\else
\fi
\bibliographystyle{abbrv}
\bibliography{references.bib}

\end{document}

%% file: preamble.tex
\usepackage[utf8]{inputenc}           
\usepackage[T1]{fontenc}             
\usepackage[left=2cm, right=2cm, top=2cm, bottom=2cm]{geometry} 
\usepackage[english]{babel}  
\usepackage{lmodern}

\usepackage{etoolbox}
\usepackage{authblk}

\usepackage[subfigure]{tocloft}      
\usepackage{titletoc}
\usepackage{textcomp}                
\usepackage{gensymb}                 
\usepackage{enumerate}               
\usepackage{enumitem}                
\usepackage{titlesec}                
\usepackage{color}                   
\usepackage[dvipsnames]{xcolor}      
\definecolor{ceared}{HTML}{BB0000}   
\definecolor{color_dofs}{HTML}{A3167C}
\usepackage{indentfirst}             
\usepackage{dsfont}
\usepackage{fancyhdr}                
\usepackage{chngpage}                
\usepackage[symbol]{footmisc}        
\usepackage{multicol}
\usepackage{float}

\usepackage{diagbox}
\usepackage[font=small, skip=10pt]{caption} 
\usepackage{subcaption}                     
\usepackage{tikz}                           
\usepackage{tkz-euclide}                    
\usetikzlibrary{patterns, patterns.meta}    
\usepackage{graphicx}
\graphicspath{{images/}}
\usepackage{pgfplots}
\usetikzlibrary{shapes, positioning}
\pgfplotsset{compat=1.16}
\usepackage{multirow}
\usepackage{placeins}

\usepackage{amsmath,amsfonts,amssymb} 
\usepackage{stmaryrd}
\usepackage{mathtools}                
\usepackage{amsthm}                   
\usepackage{empheq,etoolbox}          
\usepackage{pmat}                     
\usepackage{arydshln}                 
\usepackage{wasysym}                  
\usepackage{mathrsfs}
\usepackage{yfonts}
\patchcmd{\subequations}{\theparentequation\alph{equation}}
{\theparentequation\alph{equation}}{}{}
\newtheorem{theorem}{Theorem}[section]

\newtheorem{lemma}[theorem]{Lemma}
\newtheorem{remark}{Remark}[section]

\usepackage{nameref} 
\usepackage[hidelinks]{hyperref} 
\usepackage{cleveref}

\usepackage{pdfpages}
\usepackage{bbm}
\usepackage{scalerel}
\usepackage{bbold}

\titleformat*{\section}{\LARGE\bfseries}
\titleformat*{\subsection}{\Large\bfseries}
\titleformat*{\subsubsection}{\large\bfseries}
\titlespacing*{\section}{0pt}{0.25cm}{1em}
\titlespacing*{\subsection}{0pt}{0.75em}{1em}

%% file: preamble_m2an.tex
\usepackage{bbm}
\usepackage{amsthm}    
\usepackage{scalerel}
\usepackage{mathtools} 
\usepackage{amsmath, amsfonts, amssymb} 
\usepackage{pmat}        
\usepackage{graphicx}  
\graphicspath{{images/}}
\usepackage{wasysym}                        
\newtheorem{theorem}{Theorem}[section]
\newtheorem{lemma}[theorem]{Lemma}
\newtheorem{remark}{Remark}[section]

\usepackage[hidelinks]{hyperref} 
\usepackage{cleveref}

\usepackage{float}
\usepackage[font=small, skip=10pt]{caption} 
\usepackage{subcaption}                 
\usepackage[dvipsnames]{xcolor}      
\definecolor{ceared}{HTML}{BB0000}   
\definecolor{color_dofs}{HTML}{A3167C}
\usepackage{placeins}

\usepackage{diagbox}

\usepackage{tikz}                           
\usepackage{tkz-euclide} 
\usepackage{amssymb}  
\usetikzlibrary{shapes, positioning}

\usepackage{enumitem} 

%% file: macros.tex
\renewcommand{\sc}{\textsc}
\renewcommand{\cal}[1]{\mathcal{#1}}
\renewcommand{\rm}{\mathrm}
\newcommand{\bb}[1]{\mathbb{#1}}
\newcommand{\bbm}{\mathbbm}
\newcommand{\bd}[1]{\boldsymbol{#1}}
\newcommand{\T}{\cal{T}}
\newcommand{\F}{\cal{F}}
\newcommand{\wh}[1]{\widehat{#1}}
\newcommand{\dofs}[3]{\rm{#1}_{\cal{#2}^\sc{#3}}}    
\newcommand{\mass}[2]{\cal{M}_{\T^\sc{#2}}^{#1}}  
\newcommand{\stab}[3]{{\Sigma}_{\cal{#2}^{\sc{#1}}\cal{#3}^{\sc{#1}}}} 
\newcommand{\stabdiag}[2]{{\Sigma}_{\cal{#2}^{\sc{#1}}}} 
\newcommand{\grad}[3]{\cal{G}^{#1}_{\cal{#2}^{\sc{#3}}}} 
\newcommand{\gradd}[4]{\cal{G}^{#1}_{\cal{#2}^{\sc{#4}}\cal{#3}^{\sc{#4}}}} 
\newcommand{\strain}[2]{\cal{H}^{#1}_{\cal{#2}}} 
\newcommand{\coupling}[1]{\cal{C}^{#1}_{\F^\Gamma}} 

\newcommand{\G}{\Gamma}                   
\newcommand{\domain}[1]{\Omega^{\sc{#1}}} 
\newcommand{\hho}{\textup{\sc{hho}}}
\newcommand{\hdg}{\scriptstyle{\textup{\sc{h$+$}}}}
\newcommand{\Ts}{\cal{T}^\sc{s}}
\newcommand{\Tf}{\cal{T}^\sc{f}}

\newcommand{\sF}{\sc{f}}
\newcommand{\sS}{\sc{s}}
\newcommand{\bet}{\bd{N}}
\newcommand{\be}{\textbf{e}}
\newcommand{\bG}{\bd{g}}
\newcommand{\br}{\bd{r}}
\newcommand{\bm}{\bd{m}}

\newcommand{\bw}{\bd{w}}
\newcommand{\cF}{\mathcal{F}}
\newcommand{\dt}{\partial{T}}

\newcommand{\tih}{\tilde{h}}
\newcommand{\ol}{\overline}
\newcommand{\sd}{\scaleto{\partial}{4pt}}
\newcommand{\sg}{\scaleto{\Gamma}{4pt}}
\newcommand{\bH}{\bd{H}}
\newcommand{\Mf}{\cal{M}^{\sF}}
\newcommand{\Ms}{\cal{M}^{\sS}}
\newcommand{\Ff}{\cal{F}^{\sF}}
\newcommand{\Fs}{\cal{F}^{\sS}}
\newcommand{\bmt}{\bd{m}_{\Tf}}
\newcommand{\etaf}{\bet_{\Tf}}
\newcommand{\etas}{\bbm{N}_{\Ts}}
\newcommand{\etafT}{\bet_{\Tf}}
\newcommand{\etasT}{\bbm{N}_{\Ts}}
\newcommand{\ehf}{\hat{e}_{\Mf}}
\newcommand{\ehs}{\hat{\textbf{e}}_{\Ms}}
\newcommand{\Usigmaf}{\bd{M}^k(\Tf)}
\newcommand{\UhatfD}{\widehat{P}{}_0^k(\Mf)}
\newcommand{\UhatsD}{\widehat{\bd{V}}{}_0^k(\Ms)}
\newcommand{\Usigmas}{\bbm{S}_{\rm{sym}}^k(\Ts)}
\newcommand{\cred}{\textcolor{black}}

\newcommand{\normalizecell}[1]{\strut#1\strut}
\numberwithin{equation}{section}

\newcommand{\bbPi}{%
  \cred{\rule{0.3pt}{0.66em}\hspace{-0.02cm}\Pi}%
}